%% file: final.tex
\documentclass[11pt]{amsart}
\usepackage{amssymb}
\usepackage{calrsfs}
\usepackage[all]{xy}
\usepackage{tikz}
\usepackage{xcolor}
\usepackage[letterpaper,margin=1in]{geometry}

\newtheorem{dummy}{dummy}[section]
\newtheorem{lemma}[dummy]{Lemma}

\newtheorem{theorem}[dummy]{Theorem}

\newtheorem{corollary}[dummy]{Corollary}
\newtheorem{proposition}[dummy]{Proposition}
\theoremstyle{definition}
\newtheorem{definition}[dummy]{Definition}
\newtheorem{example}[dummy]{Example}
\newtheorem{remark}[dummy]{Remark}

\newcommand{\bA}{\mathbb{A}}
\newcommand{\C}{k}

\newcommand{\bG}{\mathbb{G}}

\newcommand{\bN}{\mathbb{N}}
\newcommand{\bP}{\mathbb{P}}

\newcommand{\R}{\mathbb{R}}

\newcommand{\Z}{\mathbb{Z}}
\newcommand{\bZ}{\mathbb{Z}}

\newcommand{\cB}{\mathcal{B}}
\newcommand{\cC}{\mathcal{C}}

\newcommand{\cE}{\mathcal{E}}
\newcommand{\cF}{\mathcal{F}}

\newcommand{\cO}{\mathcal{O}}
\newcommand{\cP}{\mathcal{P}}

\DeclareMathOperator{\val}{val}

\DeclareMathOperator{\Perf}{Perf}
\DeclareMathOperator{\Qcoh}{QCoh}
\DeclareMathOperator{\Spec}{Spec}
\DeclareMathOperator{\Dsing}{DSing}
\DeclareMathOperator{\Coh}{Coh}

\DeclareMathOperator{\MF}{MF}

\newcommand{\lf}{\mathrm{lf}}

\newcommand{\Fbal}{\mathrm{Fuk}^{bal}}
\newcommand{\Fbalc}{\mathrm{Fuk}^{bal,c}}
\newcommand{\Fex}{\mathrm{Fuk}^{ex}}
\newcommand{\Fexc}{\mathrm{Fuk}^{ex,c}}
\newcommand{\Fw}{\mathrm{Fuk}^{w}}
\newcommand{\Fc}{\mathrm{Fuk}^{c}}
\newcommand{\Flf}{\mathrm{Fuk}^{\lf}}

\newcommand{\sbu}{[\![u]\!]}
\newcommand{\rbu}{(\!(u)\!)}

\begin{document}

\title{Fukaya categories of higher-genus surfaces and pants decompositions}
\begin{abstract} In this paper we prove a local-to-global principle for the Fukaya category of a closed Riemann surface $\Sigma$ of genus $g \geq 2$. We  show that $\mathrm{Fuk}(\Sigma)$ can be glued from the Fukaya category of the pair-of-pants making up a pants decomposition of $\Sigma$. This extends our earlier results for the case of punctured Riemann surfaces. %This extends to the compact setting similar results for  punctured Riemann surfaces which we obtained in an earlier work.  
Our result has several interesting consequences: we obtain   simple proofs of old and new HMS  statements for Riemann surfaces, and establish a  geometrization theorem for the objects of $\mathrm{Fuk}(\Sigma)$.
\end{abstract}

\author{James Pascaleff}
\address{James Pascaleff, Department of 
Mathematics, University of Illinois at Urbana-Champaign, IL, US}
\email{jpascale@illinois.edu}

\author{Nicol\`o Sibilla}
\address{Nicol\`o Sibilla, 
SISSA,
Via Bonomea 265, 
34136 Trieste, Italy}
\email{nsibilla@sissa.it}

\maketitle

{\small \tableofcontents}
%\author{James Pascaleff}
%\author{Nicol\`{o} Sibilla}
%\maketitle

\input{compact-surfaces}
\input{Fuk-exceptional}

\input{completionsII}
\input{coverings}

\input{hms}

\bibliographystyle{alpha}
\bibliography{fuk}
\end{document}

%% file: compact-surfaces.tex
\section{Introduction}
\label{sec:introduction}

The problem of homological mirror symmetry (HMS) where the A-side is taken to be a Riemann surface has been much studied. When the surface is compact of genus $g \geq 2$, this problem is an instance of mirror symmetry for varieties of general type \cite{kapustin2009homological, gross2017towards}. For $g = 2$, this form of HMS was proven by Seidel \cite{seidel2011homological}, and Efimov \cite{efimov2012homological} extended Seidel's method to higher genera.

The case of open or punctured Riemann surfaces has also received a lot of attention, including works of Bocklandt \cite{bocklandt2016noncommutative}, Abouzaid-Auroux-Efimov-Katzarkov-Orlov \cite{aaeko}, Lee \cite{lee2016homological}, Ruddat \cite{ruddat2017perverse}, Lekili-Polishchuk \cite{lekili2018auslander}, and Cho-Hong-Lau \cite{cho2018gluing}. The structure of Fukaya categories of surfaces was elucidated by Haiden-Katzarkov-Kontsevich \cite{haiden2017flat}, and also by the theory of topological Fukaya categories as presented for instance in \cite{dyckerhoff2018triangulated}. The viewpoint of this paper and our previous work \cite{pascaleff2019topological} is similar to those of Lee and Cho-Hong-Lau, since they seek to understand the Fukaya category of a surface by decomposing it into pairs of pants (see also \cite{seidel2012some}).

Returning to the compact case, the mirror geometry is a certain Landau-Ginzburg model defined on a three fold. In each case, the mirror model is presented in a fairly complicated way, and in particular the connection to the geometry of the genus $g$ surface is not readily apparent (of course, by actually reading the proof one sees the connection, but it is based on branched covering arguments and does not seem very intrinsic.)

Part of the reason why the setting of Seidel and Efimov's results is so intricate is that, although compact Riemann surfaces are among the simplest examples of symplectic manifolds, describing their Fukaya category is not obvious. The main issue is that the Fukaya category is a global invariant which cannot be computed from simpler local data, or at least this is the general expectation.

 It turns out that this is not always true: mathematicians have realised that there is a large class of symplectic manifolds for which the Fukaya category has good local-to-global properties. %The last ten  years have witnessed the emergence of techniques  which allow us to compute the Fukaya category of  non-compact symplectic manifolds from local combinatorial data. 
  One of the first results in this direction was Nadler and Zaslow's description of the Fukaya category of cotangent bundles in terms of constructible sheaves on the base \cite{nadler2009constructible}. Later in \cite{kontsevich2009symplectic}  Kontsevich %proposed a  conjectural picture according to which 
conjectured that the Fukaya category of Weinstein manifolds should localize over a Lagrangian skeleton. This is an area of intense current research;   
%that makes it possible to make local-to-global computations for the Fukaya category of \emph{non-compact} manifolds.  %These developments already span a rich history: a 
%Among the first contributions to this circle of ideas, let us mention work of  Nadler--Zaslow on the cotangent bundle, and Kontsevich's conjectural picture in the Weinstein setting. 
a local-to-global (sheaf-theoretic) description of the Fukaya category of Weinstein symplectic manifolds has now been  established in many cases, in \cite{ganatra2020covariantly} and in subsequent articles by the same authors. 

This point of view has been applied with great success to HMS. The sheaf-theoretic models of the Fukaya category are, as a rule, much easier to compute since everything relies on relatively simple local calculations; and there are no pseudo-holomorphic discs involved.  
%, and they do not require working with pseudo-holomorphic discs. 
The HMS statements one is interested in often localize: complicated equivalences can be obtained by putting together more accessible local statements. We followed this strategy in our previous article \cite{pascaleff2019topological} to study Hori--Vafa mirror symmetry for toric Calabi--Yau threefolds. The mirror symplectic manifolds are punctured surfaces, and their Fukaya category can be studied very efficiently via sheaf-theoretic methods. 

How much of this picture can be expected to hold in the compact setting? In this article we address this question in the simplest case of compact Riemann surfaces of genus $g \geq 2$. We will give a complete local-to-global description of their Fukaya category which parallels our earlier results for punctured Riemann surfaces. We will use this description to study HMS for Riemann surfaces: we obtain new and simpler proof of Seidel and Efimov's results and obtain many new HMS equivalences; in fact, in a way our methods allow us to study at once all possible mirror partners of compact Riemann surfaces, and all possible HMS equivalences.

Except for interesting proposals of Tamarkin \cite{tamarkin2015microlocal} and Tsygan \cite{tsygan2009oscillatory} from a different point of view, to the best of our knowledge this is the first description of the Fukaya category of a compact manifold in terms of local-to-global data.
%\footnote{\textcolor{red}{A different approach to express the Fukaya category of some compact varieties of general type in terms of Fukaya categories of exact symplectic manifolds has been proposed by Lekili--Ueda in Conjecture 1.5 of \cite{lekili2018homological}. It would be interesting to explore the connection between their proposal and the results we prove in this article.}}
The compact setting is seemingly very different from the exact setting. One of  key players in the exact story   are \emph{skeleta}, i.e. half-dimensional CW complexes sitting inside the symplectic manifold as Lagrangian deformation retracts. Kontsevich's proposal relied precisely on the fact that we could flow Lagrangians in an exact symplectic manifold as close as we wish to a skeleton, where the symplectic geometry is the same as that of  cotangent bundles. However, for obvious topological reasons, compact symplectic manifolds do not admit a skeleta.\footnote{Roughly, the choice of skeleton is analogous to choosing a Morse function all of whose critical points have index at most half the real dimension; on a compact manifold any Morse function must have a critical point of index equal to the dimension.}

Nonetheless our results are a proof of concept that  a sheaf theoretic approach is feasible also in the compact case, at least under appropriate assumptions.  We stress that the restriction on the genus ($g \geq 2$) is not accidental. The hyperbolic/general type setting is similar to the exact setting in that contributions from pseudo-holomorphic disks can be controlled. Formally, this has the consequence that, contrary to the general case, the Fukaya category can be defined over the ground field rather than the Novikov ring.  In future work we plan to study generalisations of our techniques to higher dimensions, where some of the tricks we will use in this paper will not be available.

Let us give a more detailed summary of our results and of the main ideas going into their proofs.%, let us make a comment. 
%% In the setting of exact symplectic manifolds, we always work with the so called %Throughout the paper, we always take as the definition of Fukaya category of exact symplectic surfaces their so called 
%%\emph{topological} Fukaya category.  % we will always The model of the Fukaya category we always use in the exact setting is the \emph{topological} Fukaya category. 
%When working with exact symplectic manifolds we take the \emph{topological} Fukaya category, which is a sheaf-theoretic construction, as a stand-in for the classical Floer-theoretic definition. %This is a sheaf-theoretic model, which turns out to be equivalent to the classical Floer-theoretic definition of the wrapped Fukaya category. 
%This is harmless since the topological Fukaya category is actually equivalent to the \emph{wrapped} Fukaya category. 
%%The comparison between the two is by now well understood, and 
%%As we mentioned 
% The definitive reference (in progress) for this kind of comparison results is provided by a series of papers by Ghanatra--Pardon--Shende \cite{ganatra2020covariantly}. On the other hand, in the case of compact symplectic manifolds, only the Floer-theoretic version of the Fukaya category has been defined so far. For the sake of this introduction we will blur the line between the two, and if $\Sigma$ is a surface, be it compact or not, we will denote its Fukaya category by $\mathrm{Fuk}(\Sigma)$. 
%A brief summary of these preliminaries can be found in Section \ref{potfc} of the paper.  
%

\subsection{Main results}
The starting point for this of project was our earlier work \cite{pascaleff2019topological}. Our main technical result there was a description of the Fukaya category of a Riemann surface with punctures $\Sigma$ in terms of the Fukaya category of the pairs-of-pants.  
Let $\mathcal{P}$ be a pants decomposition of $\Sigma$, and denote by $P \in \mathcal{P}$ the individual pairs-of-pants. In this model is convenient to view pants as overlapping along cylinders: so the intersection of two neighbouring pants is a symplectic cylinder $C$. We proved that $\mathrm{Fuk}(\Sigma)$ can be expressed as a (homotopy) equalizer  of dg categories
\begin{equation}
\label{fssvb}
\mathrm{Fuk}(\Sigma) \simeq \varprojlim \Big ( \xymatrix{\prod_{P\in \mathcal{P}} \mathrm{Fuk}(P)
\ar@<-.5ex>[r]_-*!/d0.5pt/{\labelstyle }
\ar@<.5ex>[r]^ -*!/u0.7pt/{\labelstyle } 
& \prod_{C=P_1 \cap P_2} \mathrm{Fuk}(C) \Big )
}
\end{equation} 
This was the key input in our proof of HMS for $\Sigma$  which reduces to checking it for the pairs-of-pants $P$, and then using equivalence (\ref{fssvb}) to generalize it to all genera and an arbitrary number of punctures.

Equivalence (\ref{fssvb})  is closely related to, but different from, the local-to-global properties which are built in in the definition of the topological Fukaya category: it is a genuinely different manifestation of the locality of the Fukaya category in the exact setting, since it does not follow from gluing together local descriptions arising from an open cover of the skeleton; instead (\ref{fssvb}) can be proved by taking \emph{closed} covers of the skeleton of $\Sigma$ with suitable properties. However the statement finally does not even refer to the skeleton, and only depends on a pants-decomposition: of course, up to equivalence, any pants-decomposition yields the same answer.

We are stressing this point because whereas the more familiar locality statements available in the exact setting depend on working on a skeleton, (\ref{fssvb}) does not. And in fact, whereas skeleta are not well-defined for compact Riemann surfaces and so approaches based on skeleta cannot extend beyond the exact setting, it turns out that  (\ref{fssvb}) holds also in the compact setting. This is the content of our first main result; we refer the reader to Theorem \ref{essential} in the main text for a more precise statement: 
\begin{theorem}
\label{mainintro}
Let $\Sigma$ be a compact Riemann surface of genus $g \geq 2$, equipped with a pants decomposition $\mathcal{P}$. Then there is an equivalence of categories
\begin{equation}
\label{fssvb2}
\mathrm{Fuk}(\Sigma) \simeq \varprojlim \Big ( \xymatrix{\prod_{P\in \mathcal{P}} \mathrm{Fuk}(P)
\ar@<-.5ex>[r]_-*!/d0.5pt/{\labelstyle }
\ar@<.5ex>[r]^ -*!/u0.7pt/{\labelstyle } 
& \prod_{C=P_1 \cap P_2} \mathrm{Fuk}(C) \Big )
}
\end{equation}
\end{theorem}
The idea behind the proof of Theorem \ref{mainintro} is very simple. It is a variation on a trick which has been used often in symplectic geometry, namely translating between the Fukaya category of a space and that of an unramified covering. Except that, contrary to what  usually happens, we  work with an \emph{infinite} covering $\pi:\widetilde{\Sigma} \to \Sigma$, which we call the \emph{maximal tropical cover} (see Section \ref{rsamtc} for the construction). %; (see Section \ref{rsamtc} for the construction) of $\widetilde{\Sigma}$. 
The group of deck transformations of $\pi$ is a free abelian group $O_\pi$ of rank equal to the genus of $\Sigma$. The surface $\widetilde{\Sigma}$ is in a sense a much more complicated object that $\Sigma$: it is not of finite type, and comes with an infinite pants decomposition  $\widetilde{\cP}$ induced by that of $\Sigma$. However, as all non-compact Riemann surfaces it is Stein: its Fukaya category is amenable to sheaf-theoretic methods, and we show that equivalence (\ref{fssvb}) from our previous paper applies.

In fact, since  
$\widetilde{\Sigma}$ is not finite type, one needs to finesse somewhat  the  kind of Fukaya category one wants to work with, because there are several meaningful options. It turns out that for our argument, the relevant object is what we call the \emph{locally finite} Fukaya category, and define in Section \ref{lffc}: this includes both compact and non-compact Lagrangians, but satisfying a local finiteness condition. We denote it $\mathrm{Fuk}^{\lf}(\widetilde{\Sigma})$; it carries a $O_{\pi}$-action because  $O_{\pi}$ acts via deck transformations on $\widetilde{\Sigma}$. The key observation is that there is an equivalence (see Theorem \ref{essential} in the main text)
\begin{equation}
\label{gammaeq}
\mathrm{Fuk}^{\lf}(\widetilde{\Sigma})^{O_{\pi}} \simeq \mathrm{Fuk}(\Sigma)
\end{equation}
between $O_{\pi}$-equivariant objects in $\mathrm{Fuk}^{\lf}(\widetilde{\Sigma})$, and the Fukaya category of $\Sigma$. The left-hand side is in a way more complicated but  we can understand it via sheaf theoretic methods and equivalence (\ref{fssvb}); while  $\mathrm{Fuk}(\Sigma)$ can be defined only in terms of Floer-theory. Equivalence (\ref{gammaeq}) readily implies Theorem \ref{mainintro}.

  \begin{remark}
    The strategy to prove Theorem \ref{mainintro} presented here is not the only one we can imagine. Heather Lee \cite{lee2016homological}
    uses carefully chosen Hamiltonians to prove gluing results for open surfaces, and Auroux-Smith have recently extended Lee's technique to the case where a compact surface is cut along a single simple closed curve \cite[Corollary 5.8]{auroux2020fukaya}. It seems highly plausible that these techniques could be extended to the case of cutting a compact surface along several circles simultaneously, which would allow one to deduce Theorem \ref{mainintro}. Instead of using Lee's techniques, we use more generic properties of Fukaya categories, such as the relationship between covering spaces and equivariance, and Viterbo restriction functors for wrapped categories of open surfaces (in addition to the detailed understanding of the structure of the Fukaya categories of open surfaces coming from Haiden-Katzarkov-Kontsevich \cite{haiden2017flat}).
  \end{remark}

Additionally our methods allow us to answer an open question from \cite{haiden2017flat}. One of the main results of \cite{haiden2017flat} is a geometrization theorem for the 
($\bZ$-graded) Fukaya category of a punctured surface $\Sigma$: the authors prove that all objects in $\mathrm{Fuk}(\Sigma)$ are \emph{geometric}, i.e. they are equivalent in the Fukaya category to a union of arcs and immersed curves equipped with a brane structure. Question 2 in Section 7 of \cite{haiden2017flat} asks whether geometrization holds in the compact setting. %A partial  answer was obtained for a class of spherical objects in the recent \cite{auroux2020fukaya}.  
In this article we give a complete positive answer (see Corollary \ref{geometric} in the main text)
\begin{theorem}
Let $\Sigma$ be a closed surface of genus $g \geq 2$. Then all objects in $\mathrm{Fuk}(\Sigma)$ are geometric.
\end{theorem}
We remark that the case $g=0$ of the statement is trivial, the case and $g=1$ is essentially equivalent via HMS to Atiyah's classification of vector bundles on elliptic curves.

  Auroux-Smith \cite{auroux2020fukaya} also considers the question of geometricity of objects in Fukaya categories of compact surfaces. Their result is different in that it has both a stronger hypothesis and a stronger conclusion: it says that a spherical object with nonzero Chern character is represented by a simple closed curve with rank one local system.

\subsubsection*{Homological mirror symmetry}
This description of the Fukaya category of a compact Riemann surface of genus $g \geq 2$ has  significant applications to HMS. In particular, we recover  Seidel and Efimov's beautiful HMS  results  for curves of genus $g \geq 2$ \cite{seidel2011homological} \cite{efimov2012homological}. Our method of proof  however is different,  and can help clarify why these mirror constructions actually work. Both Seidel and Efimov start from carefully designed, and somewhat ad hoc, superpotentials. The proof then goes through familiar but delicate steps in HMS: matching generators, and deformation theory.  In particular, as pointed out by Seidel in the Introduction of \cite{seidel2011homological}, these proofs  do not apply to 
the other constructions of LG mirrors of curves   proposed in the literature. 
 
% It is also possible to investigate directly how the mirror Landau-Ginzburg model we obtain is related to the models considered by Seidel and Efimov. These categories have decompositions directly analogous to the pair of pants decomposition of a surface, and we can relate many different models by systematically exploiting two properties of matrix factorization categories:
%\begin{enumerate}
%\item The category of matrix factorizations is, up to equivalence, determined by the formal neighborhood of the critical locus (Orlov).
%\item Categories of matrix factorizations satisfy \'{e}tale descent (Preygel).
%\end{enumerate}
%

In order to apply Theorem \ref{mainintro} to the problem of  HMS, we need to study the singularity category of a normal crossing surface $X$. We require that $X$ that may be presented as $f^{-1}(0)$ for some morphism $f : Y \to \bA^{1}$ from a smooth $3$-fold $Y$ to the line, and also that the dual intersection complex of $X$ be orientable. We require also that the irreducible components of $Z$, the singular locus of $X$, are rational curves. These rational curves correspond to the edges of a trivalent graph $G(X)$ that encodes their intersections. We shall prove that, as long as $X$ satisfies the stated requirements, the singularity category $\Dsing(X)$ depends only on the graph $G(X)$ but not on any other details of $X$. For comparison, there is a more general result of Orlov, that states that $\Dsing(X)$ (which in this paper is always taken idempotent complete) only depends on the infinitesimal neighbourhood of $Z$ in $X$. In our setting, the singularity category is also insensitive to the infinitesimal neighbourhood, and is completely determined by the combinatorics of $Z$; we must remark, however, that the requirement that $X$ arises as the fiber of a morphism from a smooth $3$-fold does entail a topological restriction on the infinitesimal neighbourhood.

The main ingredient is that $\Dsing(-)$ satisfies \'etale (and in particular Zariski) descent with respect to varieties that are presented as the zero fiber of a morphism. This allows us to write $\Dsing(X)$ as a limit that matches the limit calculating the Fukaya category in  Theorem \ref{mainintro}. Our main result is a kind of universal HMS statement for compact Riemann surfaces, which we state as Theorem \ref{hmsintro} below.

If $X$ is a normal crossing surface with rational singular locus $Z$ as before, we  associate to it a trivalent graph $G(X)$ as follows: the vertices of $G(X)$ are the singular points of $Z$; two vertices are joined by an edge if they lie on the same irreducible component of $Z$.  Pants decompositions of Riemann surfaces also give rise to trivalent graphs via their dual intersection complex. Up to homeomorphism, there is a unique Riemann surface $\Sigma_{G(X)}$ equipped with a pants decomposition corresponding to $G(X)$: if $Z$ is compact, then $\Sigma_{G(X)}$ is a compact Riemann surface without boundary; the genus of $\Sigma_{G(X)}$ is equal to the rank of $H_1(G(X), \mathbb{Z})$. 

The following is our main HMS result, see Theorem \ref{main} in the body of the paper.

\begin{theorem}
\label{hmsintro}
There is an equivalence of categories 
$$
\Dsing(X) \simeq \mathrm{Fuk}(\Sigma_{G(X)})
$$
\end{theorem}
Seidel and Efimov's results are special cases of Theorem \ref{hmsintro}, for very specific choices of the mirror superpotential $W$ (where $X=W^{-1}(0)$). But contrary to their original proofs, our approach shows that it is only the combinatorics of the singular locus  that comes into play. This makes it relatively straightforward to compare different mirror constructions: it is enough to check the shape of the singular locus of the superpotential. 
This viewpoint   has several   benefits: our proof of HMS applies to all genera, and works equally well in the compact or punctured setting; and  it clarifies that HMS depends on a simple geometric relationship between the surface and its mirror.  

As we pointed out earlier, surfaces have actually  many different geometrically meaningful mirrors. The techniques developed in this paper allow us to easily check that they indeed give rise to HMS equivalences, which are alternative to Seidel's and Efimov's. We focus on two constructions:  
\begin{enumerate}
\item Hori--Vafa mirror symmetry matches a toric CY 3-folds $X_\Sigma$,  equipped with a toric superpotential $W_\Sigma$,  to a punctured surface $\Sigma$ equipped with a pants decomposition. We show that the Hori--Vafa picture holds in complete generality for all toric 3-folds, regardless of the CY assumption, and even in the compact setting where no superpotential is available.
\item We show that the Mumford degeneration of  abelian surfaces provides a natural mirror LG model for  Riemann surfaces of genus $2$. The same construction appears, for the other direction of HMS, in beautiful recent work of Cannizzo \cite{cannizzo2020categorical}.
\end{enumerate}

\noindent \textbf{Notes on the second arXiv version.} The principal difference between the first and second arXiv versions is that discussion of the sheaf of categories in Section \ref{Grcat} has been significantly expanded and made more precise.  It seemed to us that a full exposition of this sheaf of categories would be of independent interest, especially since others might wish to adapt this sheaf for other purposes. The new version also clarifies and corrects some issues that were raised regarding the first version. One issue, raised by Ed Segal \cite{segal2021line}, was an ambiguity in the Kn\"{o}rrer periodicity equivalences used in the definition of the sheaf. In the second arXiv version, we make an auxiliary choice of what we call framing data on the graph $G$ to resolve this ambiguity. This requires the addition of a hypothesis in Theorem \ref{matrixgraph} that the dual intersection complex of $X$ be orientable. Another clarification has to do with how the area dependence of the Fukaya categories of surfaces arises from our sheaf of categories. The notion of weights introduced Section \ref{Grcat} makes this area dependence much more explicit. 

\noindent \textbf{Acknowledgments.} This project started in April 2019 when the second author visited the University of Illinois. At that time Thomas Nevins provided valuable suggestions and encouragement. We are saddened that his untimely death in February 2020 has deprived us of the opportunity to show him the results.

We thank Tommaso de Fernex for the reference \cite{camacho2003neighborhoods}. We thank Fabian Haiden and Denis Auroux for helpful correspondence about the dependence of our categories on various parameters. We thank Yanki Lekili and Kazushi Ueda for helping us uncover an unacknowledged hypothesis in an earlier version of Theorem \ref{matrixgraph}, namely that $X$ must be presented as the fiber of a morphism.

JP was partially supported by a Collaboration Grant from the Simons Foundation.

%% file: Fuk-exceptional.tex
%\title{Fukaya categories of higher-genus surfaces and pair of pants decompositions}
%\author{James Pascaleff}
%\author{Nicol\`{o} Sibilla}
%\maketitle
\section{Preliminaries}
\subsection{Notations and conventions} 
Much of the set-up of the paper will be borrowed from our previous article \cite{pascaleff2019topological}. In this section we briefly recapitulate some of the  background from \cite{pascaleff2019topological}. We will specify the categorical setting in which we will place ourselves; and then we will recall some basic notations and constructions in the (topological) Fukaya category of surfaces. The reader is advised to consult \cite{pascaleff2019topological} for additional information. % which we introduced in \cite{pascaleff2019\mathrm{top}ological} and that will play an important role in the following.

Let $\kappa$ be a ring of characteristic $0$. We remark that in the course of the article it will be important to take $\kappa$ to be equal to either \begin{itemize}
\item a fixed ground field $k$,
\item the field $\Lambda$ of universal Novikov series over $k$ with parameter $q$,
  \begin{equation*}
    \Lambda = \left\{ \sum_{i=0}^{\infty} a_{i}q^{\lambda_{i}} \mid a_{i} \in k,\ \lambda_{i} \in \R,\ \lim_{i\to\infty} \lambda_{i} = \infty \right\},
  \end{equation*}
\item or the subring $\Lambda^{fin} \subset \Lambda$ of Novikov series with finitely many nonzero terms.
\end{itemize}

\subsubsection{Preliminaries on dg categories}
%Let $\kappa$ be a ring of characteristic $0$. In the course of the article it will be important to take $\kappa$ to be equal to either \begin{itemize}
%\item a fixed ground field $k$, 
%\item the universal Novikov field $\Lambda$ over $k$
%\item or the subring $\Lambda^{fin} \subset \Lambda$ of finite Novikov series (see Section \ref{csae} for a definition of $\Lambda$ and $\Lambda^{fin}$)
%\end{itemize}
 %Since Section 2.1 of \cite{pascaleff2019\mathrm{top}ological} gives a thorough account of the categorical setting which we will need, we will refer the reader to it and limit ourselves to briefly recall the  notations we introduced there. 

By \emph{dg category} we mean one of the following equivalent notions:
\begin{enumerate}
\item a $\Z/2\Z$-graded $\kappa$-linear dg category, or
\item a $\Z$-graded $\kappa \rbu$-linear dg category, where $u$ has cohomological degree $2$ (a $2$-periodic dg category).
\end{enumerate}
The former perspective is common in the literature on Fukaya categories, while the second perspective is used in \cite{preygel2011thom}, which is a key reference for us.\footnote{Preygel denotes $u$ by $\beta$ and gives it degree $-2$, but his differentials have degree $-1$.} There is one subtle difference between these perspectives: in the latter one, the operator $u$ can be rescaled as $u \mapsto \gamma u$ for $\gamma \in \kappa^{\times}$. This changes the $\kappa\rbu$-linear structure, and this kind of twist will appear in Section \ref{Grcat}.

Throughout the paper we will work with  dg categories up to \emph{Morita equivalence}.   References for the Morita theory of $\kappa$-linear $\Z/2\Z$-graded dg categories are Section 1 of \cite{dyckerhoff2018triangulated} and Section 2 of \cite{dyckerhoff2017}. 
\begin{enumerate}
\item  We denote by $\mathrm{DGCat^{(2), non-cocmpl}}$ the $\infty$-category of (not necessarily small) dg categories
\item We denote by $\mathrm{DGCat^{(2)}_{cont}}$ the $\infty$-category of cocomplete $\kappa$-linear dg categories. If $C$ is in 
$\mathrm{DGCat^{(2)}_{cont}}$ we denote by $C^\omega$ its full subcategory of compact objects
\item We denote by $\mathrm{DGCat^{(2)}_{small}}$ the $\infty$-category of small $\kappa$-linear dg categories. 
\end{enumerate}
There are natural functors
$$
\mathrm{U}: \mathrm{DGCat^{(2)}_{cont}} \to \mathrm{DGCat^{(2), non-cocmpl}} \quad , \quad \mathrm{Ind}:  \mathrm{DGCat^{(2)}_{small}} \to  \mathrm{DGCat^{(2)}_{cont}}
$$
given respectively by the forgetful functor, and by Ind-completion.  
We recall some useful facts about limits and colimits of dg categories. 
\begin{lemma}[\cite{gaitsgory2017study}, Corollary 7.2.7]
The functor $\mathrm{Ind}:  \mathrm{DGCat^{(2)}_{small}} \to  \mathrm{DGCat^{(2)}_{cont}}$ preserves small colimits.
\end{lemma} 
\begin{lemma}[\cite{drinfeld2015compact}, Proposition 1.7.5]
\label{libas}
Let $I$ be a small $\infty$-category and let $\Psi: I \to \mathrm{DGCat^{(2)}_{cont}}$ a functor. Assume that for all morphisms $\alpha_{i,j}: i \to j$ in $I$ the functor $\Psi_{i,j}:=\Psi(\alpha_{i,j})$ admits a continuous right adjoint $\Phi_{i,j}$. Then there is a  well-defined functor $\Phi:I^{op} \to \mathrm{DGCat^{(2)}_{cont}}$ such that for all  $\alpha_{i,j}: i \to j$ in $I$ $\Phi(\alpha_{i,j}) = \Phi_{i,j}$, and a canonical equivalence
$$
\varinjlim_I \Psi \simeq \varprojlim_{I^{op}} \Phi \in \mathrm{DGCat^{(2)}_{cont}}
$$
\end{lemma}
\begin{lemma}
\label{libasi}
Let $I$ be a small $\infty$-category and let $\Phi: I \to \mathrm{DGCat^{(2)}_{small}}$ a functor. Consider the composition 
$$
\mathrm{Ind} (\Phi): I \stackrel{\Phi} \rightarrow 
\mathrm{DGCat^{(2)}_{small}} 
\stackrel{\mathrm{Ind}} \rightarrow \mathrm{DGCat^{(2)}_{cont}}
$$
Assume that for all $i$ the limit functor 
$ \, 
 \varprojlim_{I} \mathrm{Ind} ( \Phi )  \to \Phi(i)
$
 preserves compact objects. 
Then there is a fully faithful functor 
$$
\big ( \varprojlim_{I} \mathrm{Ind} ( \Phi ) \big )^\omega \to 
\varprojlim_{I}  \Phi   \text{\, $\in$ \,} \mathrm{DGCat^{(2)}_{small}}  
$$
Additionally, the functor  is an equivalence if $I$ is finite. 
 \end{lemma}
 \begin{proof}
 The existence of a functor $\big ( \varprojlim_{I} \mathrm{Ind} ( \Phi ) \big )^\omega \to  
\varprojlim_{I}  \Phi$ follows from the universal property of limits, and it is easy to check that it has to be fully-faithful. The last statement follows because a finite limit of compact objects in a triangulated dg category is also compact: this implies that when $I$ is finite we get a functor in the opposite direction  
$\varprojlim_{I}  \Phi  \to \big ( \varprojlim_{I} \mathrm{Ind} ( \Phi ) \big )^\omega$, which is an inverse.
 \end{proof}
 We conclude this section by fixing some notations for categories of matrix factorizations. The reader can find in Section 3.1.2 \cite{pascaleff2019topological} a quick primer on  matrix factorizations; we will make use of notations and results from  \cite{preygel2011thom} to which we refer the reader for more information.  Let $X$ be a separated Noetherian scheme with a flat map 
 $f: X  \to \mathbb{A}^1$. We denote by $\Qcoh^{(2)}(X)$ and $\Perf^{(2)}(X)$ the $\mathbb{Z}_2$-folding of the triangulated dg categories of quasi-coherent sheaves and perfect complexes on $X$.  
 We associate to the pair $(X,f)$ a category of matrix factorizations $\mathrm{MF}(X, f) \in \mathrm{DGCat^{(2)}_{small}}$. We denote by $\mathrm{MF}^\infty(X,f)$ its Ind-completion  
 $$\mathrm{MF}^\infty(X,f) := \mathrm{Ind}(\mathrm{MF}(X, f)) \in \mathrm{DGCat^{(2)}_{cont}}$$ Let $X_0$ the central fiber of $f$. There is an equivalence (\cite[Proposition 3.4.1]{preygel2011thom})
 $$
\mathrm{MF}(X, f) \simeq  \Dsing(X_0) := \mathrm{D^bCoh}(X_0)/\Perf(X_0)
$$
where $ \Dsing(X_0)$ is the \emph{singularity category} of $X_0$. We denote its Ind-completion by $\Dsing^\infty(X_0)$. Since the category of matrix factorizations and the singularity category encode the same data, throughout the paper we will switch freely between the two view-points. 
\subsubsection{Preliminaries on the Fukaya category}
\label{potfc}
We briefly recall some basic facts and conventions on the topological Fukaya category, following \cite{pascaleff2019topological}. Let $\Sigma$ a non-compact Riemann surface, not necessarily of finite type. It is a classical result of Behnke and Stein \cite{behnke1947entwicklung} that all non-compact Riemann surfaces are Stein. 
 Let $\Gamma \subset \Sigma$ be a skeleton: this is a ribbon graph locally of finite type.  
 \begin{enumerate}
\item The skeleton $\Gamma$ carries a canonical cosheaf of $\mathbb{Z}_2$-graded small triangulated dg categories, $\cF^{\mathrm{top}}(-)$
$$
U \subset \Gamma   \mapsto \cF^{\mathrm{top}}(U) \in \mathrm{DGCat}^{(2)}_{\mathrm{small}}
$$ 
If $U \subset V$ are open subsets, we denote the corestriction
$$
C: \cF^{\mathrm{top}}(V) \to \cF^{\mathrm{top}}(U)
$$
We call the global sections $\cF^{\mathrm{top}}( \Gamma)$ the \emph{topological Fukaya category} of $\Sigma$, and we denote it $\mathrm{\mathrm{Fuk}}^{\mathrm{top}}( \Sigma)$.  %The category $\mathrm{\mathrm{Fuk}}^{\mathrm{top}}(\widetilde{\Sigma}_i)$ is defined as the global sections of the sheaf of categories 
%$\mathcal{F}(-)$ on any compact skeleton of $\widetilde{\Sigma}_i$. 

\item Let 
$ \, 
\mathrm{Ind}: \mathrm{DGCat}^{(2)}_{\mathrm{small}} \to \mathrm{DGCat}^{(2)}_{\mathrm{cont}} \, 
$
be the Ind-completion functor. Since $\mathrm{Ind}(-)$ preserves small colimits, applying $\mathrm{Ind}$ section-wise to a cosheaf yields a cosheaf. We denote $\cF^{\mathrm{top}}_\infty$ the cosheaf of presentable $\infty$-categories
$$
U \subset \Gamma \mapsto \cF^{\mathrm{top}}_\infty(U):= \mathrm{Ind}(\cF^{\mathrm{top}}(U)) \in \mathrm{DGCat^{(2)}_{cont}} 
$$
If $U \subset V$ are open subsets, we denote the corestriction functor
$$
C_\infty = \mathrm{Ind}(C): \cF^{\mathrm{top}}_\infty(U) \to \cF^{\mathrm{top}}_\infty(V)
$$
We denote the global sections of $\cF^{\mathrm{top}}_\infty$ by $\mathrm{Fuk}_{\mathrm{top}}^\infty(\Sigma)$. Note that 
$$
\mathrm{Fuk}_{\mathrm{top}}^\infty(\Sigma) \simeq \mathrm{Ind}(\mathrm{\mathrm{Fuk}}^{\mathrm{top}}( \Sigma))
$$
\item The corestriction functors 
have \emph{continuous} right adjoints $ C_\infty \dashv R_\infty$
%\textcolor{red}{should R be the left adjoint...? just to fit with the B-model...?}
$$
U \subset V \subset \Gamma, \quad 
R_\infty: \cF^{\mathrm{top}}_\infty(V) \to \cF^{\mathrm{top}}_\infty(U).
$$
 It follows from Lemma \ref{libas} that $\cF^{\mathrm{top}}_\infty(-)$, equipped with $R_\infty$, defines a sheaf on $\Gamma$ with values in $\mathrm{DGCat}^{(2)}_{\mathrm{cont}}
$.  
\item The Ind-completed topological Fukaya category also carries another type or restriction functors, called \emph{exotic restrictions}  in \cite{pascaleff2019topological}, and considered also in \cite{dyckerhoff2017}. Namely, let $U \subset \widetilde{\Gamma}$ be an open subgraph without $1$-valent vertices, and let $Z \subset  \Gamma$ be its complement.

Since $Z$ is also a locally finite ribbon graph, we can meaningfully evaluate on it the topological Fukaya category $\cF^{\mathrm{top}}(-)$. There are restriction functors
$$
S: \cF^{\mathrm{top}}(\Gamma) \to \cF^{\mathrm{top}}(Z), \quad 
S_\infty=\mathrm{Ind}(S): \cF^{\mathrm{top}}_\infty(\Gamma) \to \cF^{\mathrm{top}}_\infty(Z)
$$
fitting into cofiber sequences 
$$
\cF^{\mathrm{top}}(U) \stackrel{C} \longrightarrow \cF^{\mathrm{top}}(\Gamma) 
\stackrel{S} \longrightarrow \cF^{\mathrm{top}}(Z), \quad 
\cF^{\mathrm{top}}_\infty(U) \stackrel{C_\infty} \longrightarrow \cF^{\mathrm{top}}_\infty(\Gamma) 
\stackrel{S_\infty} \longrightarrow \cF^{\mathrm{top}}_\infty(Z)
$$
We denote by $T_\infty: \cF^{\mathrm{top}}_\infty(Z) \to \cF^{\mathrm{top}}_\infty(\Gamma)$ the right adjoint of $S_\infty$, and refer to it as the exceptional corestriction functor. Since $S_\infty$ is a Verdier localization, its right adjoint $T_\infty$ is fully-faithful. The functor $T_\infty: \cF^{\mathrm{top}}_\infty(Z) \to \cF^{\mathrm{top}}_\infty(\Gamma)$ preserves colimits, and  is therefore an arrow in $\mathrm{DGCat}^{(2)}_{\mathrm{cont}}$.
\end{enumerate}
We also record the following simple observation.
\begin{remark} 
\label{luvg}
Let $U \subset V \subset \Gamma$. The symplectic collars of $U$ and $V$ are open sub-surfaces of $\Sigma$, let us denote them $\Sigma_U$ and $\Sigma_V$. Assume that $\Sigma_U$ and $\Sigma_V$ are \emph{fully stopped}: that is, the intersection of $U$ and $V$ with each connected component of the boundaries of $\Sigma_U$ and $\Sigma_V$, respectively, is non-empty. Then 
$R_\infty: \cF^{\mathrm{top}}_\infty(V) \to \cF^{\mathrm{top}}_\infty(U)$ restricts to compact objects, and in fact  $R_\infty= \mathrm{Ind}(R)$ where 
$R: \cF^{\mathrm{top}}(V) \to \cF^{\mathrm{top}}(U)$.

Let us briefly explain why this is the case. In the fully stopped setting, the topological Fukaya category is smooth and proper (see \cite{dyckerhoff2017}). Since  the corestriction
$C: \cF^{\mathrm{top}}(V) \to \cF^{\mathrm{top}}(U)$  is a functor between smooth and proper categories, it admits a right adjoint 
$$
R: \cF^{\mathrm{top}}(U) \to \cF^{\mathrm{top}}(V).
$$ 
Ind-completion preserves adjoints: that is, an adjunction $C \dashv R$ can be Ind-completed to an adjunction $\mathrm{Ind}(C) \dashv \mathrm{Ind}(R)$. The uniqueness of adjoints then implies that $\mathrm{Ind}(R) \simeq R_\infty$. 
\end{remark}
By \cite{haiden2017flat} there is an equivalence between the topological Fukaya category $\mathrm{\mathrm{Fuk}}^{\mathrm{top}}(\Sigma)$  and the wrapped Fukaya category of $\Sigma$
\begin{equation}
\label{GPS}
\mathrm{\mathrm{Fuk}}^{\mathrm{top}}(\Sigma) \simeq \mathrm{\mathrm{Fuk}}^{w}(\Sigma). 
\end{equation}
Ganatra--Pardon--Shende,  starting from \cite{ganatra2020covariantly}, establish this for more general Weinstein manifolds. 
Equivalence (\ref{GPS}) is key to our approach in this paper.  It  extends more generally to the partially wrapped setting. 
Let $\Gamma \subset \Sigma$ be a skeleton with non-compact edges, that is such $S=\Gamma \cap \partial \Sigma$ is non-empty. Let $\mathrm{\mathrm{Fuk}}^{w}(\Sigma, S)$ the partially wrapped Fukaya category relative to the set of stops given by $S$. Then their work gives an equivalence
\begin{equation}
\cF^{\mathrm{top}}(\Gamma) \simeq \mathrm{\mathrm{Fuk}}^{w}(\Sigma, S)
\end{equation}

 %: in particular, up to equivalence, $\mathrm{\mathrm{Fuk}}^{\mathrm{top}}(\widetilde{\Sigma}_i)$ does not depend  on the choice of a skeleton. 

 \section{Graphs and categories}
 \label{Grcat}
Let $G$ be a  graph. We assume throughout this section that $G$ has no loops and that all vertices of 
$G$ have valency $1$ or $3$.  We denote by
\begin{itemize}
\item $V_G$, the set of vertices of $G$
\item $H_G$, the set of half-edges of $G$
\item $E_G$, the set of edges of $G$
\end{itemize}
If $v$ is a trivalent vertex in $V_G$ we denote by $ x_{v,1}, x_{v,2}$ and $x_{v,3}$ the  half-edges incident to $v$; similarly, if $v$ has valency one we denote $x_v$ the half-edge incident to $v$. For every $v \in V_G$ we denote $H_v$ the set of half-edges incident to 
$v$. There is a map $\alpha: H_G \to E_G$ that associates to an half-edge the corresponding edge. We allow $G$ to have non-compact edges, i.e. edges that are  incident to a vertex only at one of their ends but not the other: $t \in E_G$ is  non-compact if and only if the preimage $\alpha^{-1}(t)$ is a singleton. We say that a half-edge $x$ is \emph{incident} to an edge $t$ if $x \in \alpha^{-1}(t)$.

Our graphs will also be decorated with the following data:
\begin{itemize}
\item \emph{Weights,} consisting of functions $\alpha: V_{G} \to \kappa^{\times}$ and $\beta: E_{G} \to \kappa^{\times}$ that assign invertible scalars to each vertex and edge of $G$.
\item \emph{Framings,} consisting of a cyclic ordering of the edges at each 3-valent vertex, and an orientation of each edge.
\end{itemize}

\begin{remark}
  The weights are present for two reasons that are intertwined.
  \begin{itemize}
  \item On the A-side, the weights take into account the fact that the Fukaya category depends on a parameter $q$ in the Novikov field $\Lambda$.\footnote{When considering noncompact surfaces, the dependence on $q$ may be trivialized, and when considering compact surfaces of genus $g \ge 2$, there are constructions that allow one to set $q = 1$; in these situations the weights will disappear, see Section \ref{csae}. However, the covering space argument in Section \ref{csae} requires one to work with categories where $q$ is present as a parameter since the $q$-dependence cannot be equivariantly trivialized.} When comparing the sheaf of categories we construct to the Fukaya category of a surface with a pants decomposition, we will take $\kappa = \Lambda$ and set $\alpha(v) = q^{-A(v)}$ and $\beta(t) = q^{B(t)}$ for some functions $A : V_{G}\to \R$ and $B: E_{G}\to \R$ that encode the areas of various parts of the pants decomposition.
  \item On the B-side, the weights take into account the fact that categories such as $\mathrm{MF}(X,f)$ and $\Dsing(X_{0})$ are $2$-periodic dg categories, meaning that there is a prescribed isomorphism from the double shift to the identity. Equivalently, one may regard these categories as being linear over $\kappa\rbu$ where $u$ has cohomological degree 2.\footnote{The Fukaya category may also be regarded this way, but we usually treat the Fukaya category as a $\Z/2\Z$-graded category, meaning that the double shift is strictly equal to identity.} Rescaling the potential $f$ leads to an equivalent category, but with a rescaled 2-periodicity structure.
  \end{itemize}
  We shall eventually show that, up to equivalence, the category of global sections of our sheaf depends only on the product $\prod_{v \in V_{G}} \alpha(v) \prod_{t\in E_{G}}\beta(t)$, and in many cases it does not even depend on that.
\end{remark}

\begin{remark}
  The framings are necessary in order to address a certain $\Z/2\Z$-grading ambiguity pointed out in \cite{segal2021line}. The fact that this ambiguity arises in the present construction is due to our choice to present the sheaf in terms of B-side categories. On the A-side there is a canonical way of resolving the ambiguity: in brief, this is because the Viterbo or Lee restriction functor preserves the orientations of Lagrangians (rather than reversing them). The framings are the data that allow us to transfer this preferred choice to the B-side. Up to equivalence, the category of global sections does not depend on these data; they are purely auxiliary.

From the HMS perspective, the framings are necessary for the following reason. Let $P$ be the pair of pants. Then there exist HMS equivalences of categories
  \begin{equation*}
    \mathrm{Fuk}(P) \cong \mathrm{MF}(\mathbb{A}^{3},x_{1}x_{2}x_{3}).
  \end{equation*}
  There are several such equivalences. The symmetric group $S_{3}$ acts on $P$ by symplectomorphisms that permute the three holes of $P$, and it acts on $\mathbb{A}^{3}$ by permuting the variables $x_{1},x_{2},x_{3}$. Thus $S_{3}$ acts on both $\mathrm{Fuk}(P)$ and $\mathrm{MF}(\mathbb{A}^{3},x_{1}x_{2}x_{3})$, and this induces an action of $S_{3}$ on the set of (homotopy classes of) equivalences between these categories. The key point is that this action is not trivial: even permutations act trivially, but odd permutations act by composing a given equivalence with the shift functor.

  Let $C$ be a cylinder. Then we have $\mathrm{Fuk}(C) \cong \Perf^{(2)}(\mathbb{G}_{m})$. Now $S_{2}$ acts on $C$ by permuting the holes, and it acts on $\mathbb{G}_{m}$ by the inversion morphism. Again, the action of $S_{2}$ on the set of such equivalences is not trivial: it acts by composition with a shift.
\end{remark}

%  \begin{remark}
%Our assumption on the valency of the graph $G$ might seem unnatural. %, and in fact in our earlier treatment in Section 3.1.2 of \cite{pascaleff2019\mathrm{top}ological} we restricted attention to 3-valent graphs. 
%%The graphs we will work with arise from two different contexts. 
%Dual intersection complexes of pants decomposition of a Riemann surfaces play a key role in the paper, and these will always be trivalent. However it will be useful to allow also for vertices of valency one. In the context of pants decompositions, 1-valent vertices correspond to pants that are \emph{capped} on one end: in Section \ref{hms} we will interpret these as encoding s\mathrm{top}s in the Fukaya category. Additionally, in Section \ref{mfat} below, we will  work with graphs arising from singular loci of normal crossing divisors. In this context $1$-valent vertices will correspond to irreducible components of the singular locus which are isomorphic to $\mathbb{P}^1$ and meet the other components only in one point.
%\end{remark}

\subsection{Local model for the restriction functors}
\label{sec:local-model-res}

Before setting up the sheaf of categories on graphs, we will present a detailed construction of the restriction functors in a local model consisting of a single trivalent vertex. This allows us to suppress some indices and make the essential points clearer.

So let $G$ consist of a single trivalent vertex $v$, and three edges. Equip $G$ with a cyclic ordering of the edges, and an orientation of each edge. Let us name the edges $t_{1},t_{2},t_{3}$ so that the chosen cyclic order of the edges is compatible with the standard cyclic order of the indices $1,2,3$. Then we may choose weights consisting of invertible scalars $\alpha = \alpha(v) \in \kappa^{\times}$ at the vertex $v$ and $\beta_{i} = \beta(t_{i}) \in \kappa^{\times}$ at the edges $t_{i}$.

The stalk at the vertex $v$ is
\begin{equation*}
  \mathcal{B}_{v} = \mathrm{MF}^{\infty}(\mathbb{A}^{3},\alpha x_{1}x_{2}x_{3}).
\end{equation*}
Note that the weight appears in the potential; the indexing of the variables is meant to reflect a bijection between these variables and the set of edges. For each edge $t_{i}$, construct the ring $\kappa[x_{i},\tilde{x}_{i}]/(x_{i}\tilde{x}_{i} - \beta_{i})$. Note that the weight appears in the relation. We write $\mathbb{G}_{m,x_{i},\tilde{x}_{i},\beta_{i}}$ for the spectrum of this ring. The stalk at the edge $t_{i}$ is
\begin{equation*}
  \mathcal{B}_{t_{i}} = \Qcoh^{(2)}(\mathbb{G}_{m,x_{i},\tilde{x}_{i},\beta_{i}}).
\end{equation*}
What we need to do is specify precisely a collection of restriction functors $\mathcal{B}_{v} \to \mathcal{B}_{t_{i}}$.
For each $i = 1,2,3$, we can consider a composition
\begin{equation*}
  R_{i}: \mathcal{B}_{v} = \mathrm{MF}^{\infty}(\mathbb{A}^{3},\alpha x_{1}x_{2}x_{3}) \to \mathrm{MF}^{\infty}(\mathbb{A}^{3} \setminus \{x_{i} = 0\},\alpha x_{1}x_{2}x_{3}) \to \Qcoh^{(2)}(\mathbb{G}_{m,x_{i},\tilde{x}_{i},\beta_{i}}) = \mathcal{B}_{t_{i}},
\end{equation*}
where the first functor is restriction to a Zariski open subset, and the second functor is a Kn\"{o}rrer periodicity equivalence. Two points deserve emphasis:
\begin{enumerate}
\item This functor is to be chosen so that the variable $x_{i}$ appearing in the notation for the source category matches the variable $x_{i}$ appearing in the notation for the target category. In particular, these two variables cannot be rescaled independently.
\item There are two choices for the Kn\"{o}rrer periodicity equivalence that differ from each other by a shift (note that our categories are $\Z/2\Z$-graded so the double shift is the identity). 
\end{enumerate}

The latter point means that we need to specify a preferred choice of Kn\"{o}rrer periodicity equivalence for each half-edge. We do this using the framing data. Recall that these data determine a cyclic ordering of the edges at $v$, and an orientation of each edge. For each \emph{unordered} pair of distinct indices $i,j \in \{1,2,3\}$, let $k \in \{1,2,3\}$ be the third index (so that $\{i,j\} \cup \{k\} = \{1,2,3\}$), and construct the matrix factorization
\begin{equation}
  \label{basic-factorizations}
  F_{ij} = \left(\xymatrix{ \cO_{\mathbb{A}^{3}} \ar@<2pt>[r]^{x_{i}x_{j}} & \ar@<2pt>[l]^{\alpha x_{k}}\cO_{\mathbb{A}^{3}} }\right)
\end{equation}
where the left-hand $\cO_{\mathbb{A}^{3}}$ is placed in even degree. When restricting to the $i$-th or $j$-th edge, the object $F_{ij}$ is mapped by Kn\"{o}rrer periodicity to an object isomorphic to either $\cO_{\mathbb{G}_{m}}$ or $\cO_{\mathbb{G}_{m}}[1]$. We stipulate the following rule:
\begin{itemize}
\item Suppose $j$ follows $i$ with respect to the given cyclic order at $v$, then:
\item If $t_{i}$ is oriented into $v$, we require $R_{i}(F_{ij}) \cong \cO_{\mathbb{G}_{m}}$.
\item If $t_{i}$ is oriented out of $v$, we require $R_{i}(F_{ij}) \cong \cO_{\mathbb{G}_{m}}[1]$.
\item If $t_{j}$ is oriented into $v$, we require $R_{j}(F_{ij}) \cong \cO_{\mathbb{G}_{m}}[1]$.
\item If $t_{j}$ is oriented out of $v$, we require $R_{j}(F_{ij}) \cong \cO_{\mathbb{G}_{m}}$.
\end{itemize}
For instance, the condition $R_{i}(F_{ij}) \cong \cO_{\mathbb{G}_{m}}$ obtains when the inverse functor
\begin{equation*}
   \Qcoh^{(2)}(\mathbb{G}_{m,x_{i},\tilde{x}_{i},\beta_{i}}) \to \mathrm{MF}^{\infty}(\mathbb{A}^{3} \setminus \{x_{i} = 0\},\alpha x_{1}x_{2}x_{3})
 \end{equation*}
 is given by tensoring with the restriction of $F_{ij}$. The reason why this convention is consistent is that if $\{i,j,k\}$ is an unordered triple of pairwise distinct indices, then $R_{j}(F_{ij}) = R_{j}(F_{jk})[1]$ with respect to \emph{any} choice of Kn\"{o}rrer periodicity equivalance at the edge $j$. It is evident that reversing the orientation of an edge modifies the corresponding restriction by a shift, and reversing the cyclic ordering at the vertex modifies all of the restrictions at that vertex by a shift.

\subsection{Dependence of the local model on the weights}
\label{sec:weight-dependence}

We now turn to the question: In what sense do the local models depend on the weights? One may argue that the categories $\mathrm{MF}(\mathbb{A}^{3},\alpha x_{1}x_{2}x_{3})$ are all equivalent as $\alpha$ varies, and the rings $\kappa[x,\tilde{x}]/(x\tilde{x} - \beta)$ are all isomorphic as $\beta$ varies, so it would seem that these weights could all be set to $1$ once and for all. However, this is not appropriate in our setting because, as mentioned above, we regard the coordinates $x_{1},x_{2},x_{3}$ as fixed; these coordinates serve to ``mark'' the categories and are used in the construction of the restriction functors.

The case of $\kappa[x,\tilde{x}]/(x\tilde{x}-\beta)$ is the more transparent one: it is clear that any automorphism of this ring that fixes the coordinates $x,\tilde{x}$ cannot change the value of $\beta$: indeed, the value of $\beta$ may be extracted as the product of the coordinates.

The case of $\mathrm{MF}(\mathbb{A}^{3},\alpha x_{1}x_{2}x_{3})$ is more subtle. A famous theorem of Orlov \cite{orlov2003triangulated} implies that there is an equivalence of categories\footnote{All our categories are idempotent complete.}
\begin{equation*}
  \mathrm{MF}(\mathbb{A}^{3},\alpha x_{1}x_{2}x_{3}) \cong \Dsing(\{x_{1}x_{2}x_{3}=0\}).
\end{equation*}
The right-hand side appears not to depend on $\alpha$, but in fact there is a finer structure that does depend on $\alpha$, namely the $2$-periodic or $\kappa\rbu$-linear structure.

The way this works can be seen using Preygel's approach to $\mathrm{MF}(X,f)$ \cite{preygel2011thom}. Preygel begins by considering the category of tuples $(\cE, d, B)$ where $(\cE, d)$ is a perfect complex on the total space $X$ and $B \in \hom^{-1}(\cE,\cE)$ is a nullhomotopy for the multiplication by $f$ such that $B^{2} = 0$; this is a dg model for $\Coh(X_{0})$. Next, tensor everything with $\kappa \sbu$, where $u$ has cohomological degree $2$, and replace $(\cE,d,B)$ with $(\cE, d+uB)$. Note that $(d+uB)^{2} = uf$, so the term ``matrix factorization'' is justified. Lastly, localize by inverting $u$ to obtain the $\kappa\rbu$-linear category $\mathrm{MF}(X,f)$. This category is $2$-periodic, with the periodicity given by multiplication by $u$. The formulae make clear the effect of rescaling $f$: if we rescale $f$ to $\alpha f$ for some $\alpha \in \kappa^{\times}$, then to get an equivalent category we must rescale $u$ to $\alpha^{-1}u$.

There is also a way to view this in terms Orlov-style $\Z/2\Z$-graded matrix factorizations. In this construction, we start with pairs
\begin{equation}
  F = \left(\xymatrix{ \cE_{0} \ar@<2pt>[r]^{t_{0}} & \ar@<2pt>[l]^{t_{1}}\cE_{1} }\right)
\end{equation}
such that $t_{1}t_{0} = f \cdot 1_{\cE_{0}}$ and $t_{0}t_{1} = f \cdot 1_{\cE_{1}}$. For two such objects $F,F'$, the differential $d$ on the complex $\hom(F,F')$ is the graded commutator with the internal differentials $t, t'$. Now suppose that $f$ is rescaled by the factor $c^{2}$ for some $c \in \kappa^{\times}$. There are several ways to convert a matrix factorization for $f$ into one for $c^{2}f$. The simplest is to rescale all internal differentials such as $t_{0},t_{1}$ by $c$. This has the effect of rescaling the differential $d$ on $\hom(F,F')$ by $c$, but it does not change the chain-level composition. Since the cocycles and coboundaries for $d$ and $cd$ are the same, there is a natural identification of the cohomology categories.

To see the effect of rescaling the differential in terms of $A_{\infty}$-structures, we use the homological perturbation method. This involves, for each hom complex, choosing a projection $\Pi$ onto its cohomology as well as a homotopy operator $P$ such that
\begin{equation*}
  Pd + dP = I - \Pi.
\end{equation*}
Because this equation is inhomogeneous, when we rescale $d$ by $c$ we should rescale $P$ by $c^{-1}$. The homological perturbation method then produces an $A_{\infty}$-structure on the cohomology category with vanishing differential. For each $d \geq 2$, the formula for the $d$-th operation $m_{d}$ uses the homotopy operator $(d-2)$ times, and hence the operator $m_{d}$ is rescaled by $c^{2-d}$. There is no canonical way to trivialize this dependence on $c$ in the $\Z/2\Z$-graded setting.

This shows that $\mathrm{MF}(X,f)$ and $\mathrm{MF}(X,\alpha f)$ are not canonically equivalent with their $2$-periodic structures, but of course there can be other ways to construct an equivalence between these categories. For instance, if $X$ carries a $\mathbb{G}_{m}$-action such that $f$ is homogeneous with nonzero weight, we can rescale the coordinates on $X$ to absorb the scaling factor $c^{2}$. This already applies to the case $(X,f)= (\mathbb{A}^{3},x_{1}x_{2}x_{3})$ (in several ways).

  Using the classification of $A_{\infty}$-structures presented in \cite{aaeko}, it is possible to show that any equivalence $\mathrm{MF}(\mathbb{A}^{3},\alpha x_{1}x_{2}x_{3}) \cong \mathrm{MF}(\mathbb{A}^{3},\alpha' x_{1}x_{2}x_{3})$ for $\alpha \neq \alpha'$ that is compatible with the $2$-periodic structure \emph{must} involve rescaling the coordinates $x_{1},x_{2},x_{3}$ in some way.

% Now consider the categories $\mathrm{MF}(\mathbb{A}^{3},\alpha x_{1}x_{2}x_{3})$ for various $\alpha$. We shall assume that $\alpha$ has a square root in $\kappa$ that we denote $\sqrt{\alpha}$. Although these categories are mutually $A_{\infty}$-equivalent for all values of $\alpha$, we shall argue that such equivalences \emph{must} involve rescaling the coordinates $x_{1},x_{2},x_{3}$ in some way.

To make a precise statement, consider the matrix factorizations
\begin{equation}
  \label{basic-factorizations-alpha}
  F_{ij}(\sqrt{\alpha}) = \left(\xymatrix{ \cO_{\mathbb{A}^{3}} \ar@<2pt>[r]^{\sqrt{\alpha}x_{i}x_{j}} & \ar@<2pt>[l]^{\sqrt{\alpha}x_{k}}\cO_{\mathbb{A}^{3}} }\right)
\end{equation}
These objects generate $\mathrm{MF}(\bA^{3},\alpha x_{1}x_{2}x_{3})$, and they correspond to each other for various values of $\alpha$ according to our identifications. Furthermore, the coordinates $x_{i}$ actually appear in the cohomology category as generators of $H^{0}(\hom(F_{ij}(\sqrt{\alpha}),F_{ij}(\sqrt{\alpha})) \cong \kappa[x_{i},x_{j}]/(x_{i}x_{j})$. Recall that a \emph{strict homotopy equivalence} between to $A_{\infty}$-categories whose cohomologies have been identified is an $A_{\infty}$-equivalence that induces identity on the cohomology category. We use the term \emph{semi-strict homotopy equivalence} to refer to an $A_{\infty}$-equivalence that preserves the objects $F_{ij}(\sqrt{\alpha})$ and the morphisms denoted $x_{i}$, but which may behave arbitrarily on other morphisms. 

%namely those that correspond under the Orlov equivalence $\mathrm{MF}(\mathbb{A}^{3},\alpha x_{1}x_{2}x_{3}) \cong D_{sg}(\{x_{1}x_{2}x_{3} = 0\})$ to structure sheaves of lines parallel to the axes that are contained in $\{x_{1}x_{2}x_{3} = 0\}$.

  % From an algebraic perspective, one can ask why it is necessary to introduce the parameter $q$ at all. One may deduce from CITE AAEKO that the categories
  % \begin{equation}
  %   \mathrm{MF}(\mathbb{A}^{3},q^{A}x_{1}x_{2}x_{3})
  % \end{equation}
  % are $A_{\infty}$-equivalent for all values of $A$. However, it is not possible to choose such equivalences to be compatible with the other choices that we make in our construction. The key point is that, in our construction, we need to regard the coordinates $x_{1},x_{2},x_{3}$ as something fixed; in particular we are not allowed to rescale these coordinates freely.

\begin{proposition}
  Suppose that $\alpha,\alpha'\in \kappa^{\times}$ are such that $\mathrm{MF}(\mathbb{A}^{3},\alpha x_{1}x_{2}x_{3})$ and $\mathrm{MF}(\mathbb{A}^{3},\alpha' x_{1}x_{2}x_{3})$ are semi-strictly homotopy equivalent as $\kappa\rbu$-linear categories. Then $\alpha = \alpha'$.

  Suppose that $\alpha'/\alpha = c^{2}$ for some $c \in \kappa^{\times}$. Let $\cC$ be the $A_{\infty}$-category obtained from the minimal model of $\mathrm{MF}(\bA^{3},\alpha x_{1}x_{2}x_{3})$ by rescaling $m_{d}$ by $c^{2-d}$. Then $\cC$ is semi-strictly homotopy equivalent to $\mathrm{MF}(\bA^{3},\alpha' x_{1}x_{2}x_{3})$.
\end{proposition}

\begin{proof}
Our argument is based on the results of \cite{aaeko}.  Consider the generating objects $F_{ij}(\sqrt{\alpha})$ defined in \eqref{basic-factorizations-alpha}, which are preserved by a semi-strict homotopy equivalence. The cohomology category of these objects is a certain quiver algebra with relations which the authors of \cite{aaeko} denote by $A_{(p,q)}$, where $p$ and $q$ are certain grading data that are irrelevant in the $\Z/2\Z$-graded setting. The variables $x_{1},x_{2},x_{3}$ then serve as names for certain loops of length two in the quiver $A_{(p,q)}$. The $A_{\infty}$-structures on $A_{(p,q)}$ are classified up to strict homotopy equivalence by two scalars $a,b$, which are extracted from certain $m_{3}$ operations. These structures are denoted $m^{a,b}$. Rescaling $m_{d}$ by $c^{2-d}$ thus changes $m^{a,b}$ to $m^{c^{-1}a,c^{-1}b}$. It is demonstrated that the categories $\mathrm{MF}(\mathbb{A}^{3},\alpha x_{1}x_{2}x_{3})$ correspond to nonzero values of $a$ and $b$. By rescaling the arrows in $A_{p,q}$, all of the structures $m^{a,b}$ with $a \neq 0$ and $b \neq 0$ are equivalent to $m^{1,1}$.  However, such rescalings may modify the morphisms named $x_{1},x_{2},x_{3}$. If we only allow rescalings that do not modify these morphisms, then $m^{a,b}$ is only equivalent to $m^{\lambda a, \lambda^{-1}b}$, and so the product $ab$ of the two scalars is invariant with respect to this class of rescalings. Since $m^{a,b}$ and $m^{c^{-1}a,c^{-1}b}$ give different values for this product, they are not equivalent with respect to the class of rescalings that fix $x_{1},x_{2},x_{3}$.  
\end{proof}

\subsection{Relationship of the local model to the Fukaya category}
\label{sec:local-fukaya}

We shall now describe how the local model for the restriction defined in Section \ref{sec:local-model-res} is compatible with the theory of Fukaya categories. 

Let the notation be as in Section \ref{sec:local-model-res}: let $v$ be a trivalent vertex with incident edges $t_{1},t_{2},t_{3}$, with weights $\alpha$, $\beta_{1},\beta_{2},\beta_{3}$, and with orientations on the edges and a cyclic ordering at the vertex that is compatible with our indexing of the edges. We take $\kappa = \Lambda$, the Novikov field, and we assume that the weights are of the form $\alpha = q^{-A}$ and $\beta_{i}= q^{B_{i}}$ ($i=1,2,3$) for some real numbers $A,B_{1},B_{2},B_{3}$ to be determined. 

To the vertex $v$ we associate a pair of pants (sphere with three holes) $P$, whose boundary components $(\partial P)_{i}$ are in bijection with the edges $t_{i}$, and to each edge $t_{i}$ we associate a cylindrical neighborhood or collar $C_{i}$ around the corresponding boundary component of $P$. The weights $A,B_{1},B_{2},B_{3}$ will be related to the the size of the symplectic form on the pieces $P_{v}, C_{1}, C_{2}, C_{3}$ respectively, and the framing data are used to single out certain collections of generating objects for the associated Fukaya categories.

For each pair of boundary components of $P$, there is a Lagrangian arc that connects them. The cyclic ordering of the edges (part of the framing) induces a cyclic ordering of the boundary components, and we orient these arcs so as to be compatible with this cyclic ordering. We notate these objects as $L_{12}, L_{23}, L_{31}$, where $L_{ij}$ is an arc connecting boundary component $i$ to boundary component $j$ and oriented from $i$ to $j$. (Reversing the cyclic ordering of the edges at $v$ reverses the orientations on these objects, which categorically speaking is the same as applying the shift functor.)

For each $C_{i}$, we take an arc $K_{i}$ joining the two boundary components of $K_{i}$. We choose the orientation on $K_{i}$ to be induced from the orientation on the edge $t_{i}$: if the orientation on $t_{i}$ points into the vertex $v$, then the orientation on $K_{i}$ points toward the interior of the pair of pants, and vice versa. Note that this orientation may differ from the orientation of $L_{ij}\cap C_{i}$.

% We can set things up so that $L_{ij}$ intersects the collars $C_{i}$ and $C_{j}$ in single arcs (as opposed to some disjoint union of arcs). The other part of the HMS framing, the orientation on the edges $t_{i}$, then determines an orientation on $L_{ij} \cap C_{i}$: if $t_{i}$ is oriented ``into'' $v$ then $L_{ij} \cap C_{i}$ is oriented ``into'' $P$, and vice versa. This orientation may or may not agree with the one induced on $L_{ij}$ via the cyclic ordering at the vertex $v$.

Next, form the (wrapped) Fukaya categories $\mathrm{Fuk}(P)$ and $\mathrm{Fuk}(C_{i})$ over $\Lambda$, and apply HMS \cite{aaeko}. There is an equivalence of categories
\begin{equation*}
  \Phi : \mathrm{Fuk}(P) \cong \mathrm{MF}(\mathbb{A}^{3},q^{-A}x_{1}x_{2}x_{3})
\end{equation*}
That sends generating objects $L_{ij}$ to $F_{ij}(q^{-A/2})$ as in \eqref{basic-factorizations-alpha}, where $A$ is the area of the pair of pants. Note that we have used the cyclic ordering. In terms of \cite{aaeko}, this is the model for the Fukaya category that is generated by odd morphisms $u_{ij}$ and $v_{ji}$ for cyclically ordered $i,j$ such that $m_{3}(u_{12},u_{23},u_{31})$ and $m_{3}(v_{13},v_{32},v_{21})$ have constant term $q^{A/2}$.

Now consider $\mathrm{Fuk}(C_{i})$. There is an equivalence of categories
\begin{equation*}
  \Psi_{i} : \mathrm{Fuk}(C_{i}) \cong \Perf^{(2)}(\mathbb{G}_{m,x_{i},\tilde{x}_{i},B_{i}}) 
\end{equation*}
where $B_{i}$ is the area of the annulus. $\mathrm{Fuk}(C_{i})$ is generated by the arc $K_{i}$ with a specific orientation determined by the edge orientation, and we require that it corresponds to the structure sheaf under $\Psi_{i}$. The variables $x_{i}$ and $\tilde{x}_{i}$ are to be interpreted as the basic Reeb chords that generate the endomorphism of this object.

We now explain the functor
\begin{equation*}
  R_{i} : \mathrm{Fuk}(P) \to \mathrm{Fuk}(C_{i})
\end{equation*}
which we will define as a variant of the Viterbo restriction functor. To define it, we need to make use of a Liouville one-form $\theta$ on $P$, which is a one form such that $d\theta = \omega$ is the symplectic form, and such that the corresponding Liouville vector field points outward along the boundary. The choice of $\theta$ is not unique, since two choices $\theta, \theta'$ differ by a closed one form that represents a class $[\theta' - \theta] \in H^{1}(P,\R)$. For each $i$, we choose a form $\theta_{i}$ such that $C_{i} \subset P$ forms a Liouville subdomain. The core or skeleton associated to $\theta_{i}$ will then consist of a circle contained in $C_{i}$ with some some other $1$-cells attached. It is not possible to choose a single form $\theta$ such that all three annuli $C_{1},C_{2},C_{3}$ form Liouville subdomains simultaneously, because $P$ does not retract onto its boundary plus a collection of $1$-cells; a $2$-cell is required. However, it is possible to choose a form such that two of the three annuli are Liouville subdomains (the skeleton is then a dumbbell graph).

Algebraically, the choice of $\theta_{i}$ trivializes the categories $\mathrm{Fuk}(P)$ and $\mathrm{Fuk}(C_{i})$ with respect to the Novikov parameter $q$. This is done be rescaling all morphisms by certain $q$-powers. Denote by $\Fex(P,\theta_{i})$ and $\Fex(C_{i},\theta_{i})$ the resulting categories over the ground field $k$. Geometrically, the existence of a Liouville embedding $C_{i} \to P$ gives a Viterbo restriction functor
\begin{equation*}
  R^{ex}_{i} : \Fex(P,\theta_{i}) \to \Fex(C_{i},\theta_{i}).
\end{equation*}
that induces a functor $R_{i} : \mathrm{Fuk}(P) \to \mathrm{Fuk}(C_{i})$ once we restore the trivialized $q$ parameter.

\begin{remark}
Even though $R_{i}$ is obtained from a functor over $k$ by base change to $\Lambda$, the diagram consisting of all three restrictions $R_{1},R_{2},R_{3}$ is \emph{not} the base change to $\Lambda$ of a diagram over $k$, precisely because we cannot choose a single Liouville form that works for all three annuli.
\end{remark}

It remains to verify that the corresponding functor between B-side categories does have the properties required in Section \ref{sec:local-model-res}. The first requirement is that the coordinate $x_{i}$ should correspond under $R_{i}$ to the coordinate of the same name. This holds because, when we trivialize the $q$ parameter using $\theta_{i}$, the morphisms in $\mathrm{Fuk}(P)$ and $\mathrm{Fuk}(C_{i})$ named by $x_{i}$ are rescaled in the same way.

The second requirement has to do with the shift ambiguity in the Kn\"{o}rrer periodicity equivalence. The key observation is that Viterbo restriction does not reverse the orientation of the Lagrangians. We spell out the details: Choose a pair of indices $ij$ so that $j$ follows $i$ in the cyclic order at $v$. Then the matrix factorization $F_{ij}$ is mapped by $\Phi^{-1}$ to $L_{ij}$, where $L_{ij}$ is the arc on $P$ that connects boundary $i$ to boundary $j$, and that is oriented into $P$ at $i$ and out of $P$ at $j$.
\begin{itemize}
\item If $t_{i}$ is oriented into $v$, the orientations on $K_{i}$ and $L_{ij}$ match, and $R_{i}(F_{ij}) \cong \cO_{\mathbb{G}_{m}}$.
\item If $t_{i}$ is oriented out of $v$, the orientations on $K_{i}$ and $L_{ij}$ differ, and $R_{i}(F_{ij}) \cong \cO_{\mathbb{G}_{m}}[1]$.
\item If $t_{j}$ is oriented into $v$, the orientations on $K_{j}$ and $L_{ij}$ differ, and $R_{j}(F_{ij}) \cong \cO_{\mathbb{G}_{m}}[1]$.
\item If $t_{j}$ is oriented out of $v$, the orientations on $K_{j}$ and $L_{ij}$ match, and $R_{j}(F_{ij}) \cong \cO_{\mathbb{G}_{m}}$.
\end{itemize}

 \subsection{Sheaves of categories on graphs}
 \label{pdac}
In Section 3.1.3 of \cite{pascaleff2019topological} we explained how to associate to a graph $G$ a locally constant sheaf of categories $\mathcal{B}(-)$, with values in $\mathrm{DGCat}^{(2)}_{\mathrm{cont}}$. In this section we enlarge slightly the set-up of   \cite{pascaleff2019topological} : %broaden the scope of our construction to general graphs with vertices of valency $1$ or $3$ and no loops: 
in \cite{pascaleff2019topological}  we  required $G$ to be trivalent and planar, but in fact this is not necessary for the construction, and it will be useful to consider more general graphs. 

Note that specifying $\mathcal{B}(-)$ means more concretely giving the following  data:
 \begin{itemize}
 \item for each vertex $v$ of $G$, a category 
 $\mathcal{B}_v$, which is the stalk of $\mathcal{B}(-)$ at $v$
 \item for each edge $t$ of $G$, a category 
 $\mathcal{B}_t$, which is the stalk of $\mathcal{B}(-)$ at an interior point of $t$
\item if $t$ is incident to $v$, a restriction functor $\mathcal{B}_v \to \mathcal{B}_t$
 \end{itemize}
 The information can be encoded equivalently in terms of   the \emph{exit-path quiver} of $G$: the exit-path quiver has the as vertices the strata of $G$ (that is, vertices and edges), while the arrows keep track of inclusions of strata.  We find it convenient to work in fact with a small variation of the exit-path quiver of $G$, which is bicolored. Namely, we associate to $G$ a bipartite quiver $Q_G$ with black and white vertices:
\begin{enumerate}
\item The set of black vertices of $Q_G$ is given by $V_G$; the set of white vertices   is given by $E_G$
\item All the arrows in $Q_G$  go from black  to white vertices: if $v$ is in $V_G$ and $t$ is in $E_G$, the set of arrows between $v$ and $t$ is the set of half-edges $x$ incident to $e$ such that $\alpha(x) = t$
\end{enumerate}
Note that there are either one or no arrows between a black vertex $\bullet_v$ and a white vertex $\circ_t$; an arrow $\bullet_v \to \circ_t$ is naturally labeled   by an half-edge  incident to $v$. We will write $q \in Q_G$ to refer to any vertex $q$  in $Q_G$, without specifying the color. As we said, if we forget the coloring of the vertices of $Q_G$, we get back the usual exit path quiver of $G$.

The data defining the constructible sheaf of categories $\cB$ can be phrased in terms of a functor between $Q_G$ and $\mathrm{DGCat}^{(2)}_{\mathrm{cont}}$, which with small abuse of notation we keep denoting $\mathcal{B}(-)$. For clarity let us make a concrete example: if $t$ is an edge of $G$, then this corresponds to a white vertex $\circ_t$ of $Q_G$; the value of the  functor $\mathcal{B}(-)$ on the vertex $\circ_t$ coincides with the stalk $\cB_t$ at an interior point of $t$.   % of $Q_G$. 
   In the following we shall switch between these two perspectives according to what is more convenient in a given context.

 Before defining $\cB$ we need to fix some notations. Let $T$ be a finite set with $n$ elements: we will regard the   elements of $T$ as variables by working 
 with the free commutative algebra over $T$, denoted  $\kappa[t, t \in T]$, and  subalgebras spanned by subsets $T' \subset T$. %Let $\kappa[t, t \in T]$ be the free commutative algebra over $T$. 
 We set 
$$\mathbb{A}^n_{T}:= \Spec(\kappa[t, t \in T]), \quad \text{and} \quad   W_T := \, \times_{t \in T} t \, : \mathbb{A}^n_{T} \longrightarrow \mathbb{A}^1.$$ 
If $t$ is in $T$ we denote $D(t) = \mathbb{A}^n_{T} - \{t=0\}$  the corresponding principal open subset. We set $\mathbb{G}_{m,t} :=\Spec(\kappa[t, t^{-1}])$, and 
$\mathbb{A}^1_t =\Spec(\kappa[t])$. 
%we denote 
%$\mathbb{P}^1_t$ the projective line compactifying $\mathbb{G}_{m,t}$. 
We denote $j: \mathbb{G}_{m,t} \to \mathbb{A}^1_t $ the inclusion.   %and consider the pull-back
%$$
%j^*: \Qcoh^{(2)}(\mathbb{A}^1_t) \longrightarrow  \Qcoh^{(2)}(\mathbb{G}_{m,t}). 
%$$
If $t, t' \in T$ and $\beta\in \kappa^{\times}$, we set $\mathbb{G}_{m,t, t',\beta} :=\Spec(\kappa[t, t']/tt'-\beta)$. We have inclusions 
$
\mathbb{A}^1_t  \stackrel{j}\leftarrow \mathbb{G}_{m,t, t',\beta} \stackrel{j}\rightarrow \mathbb{A}^1_{t'}. 
$
%and pull-backs 
%$$
 %\Qcoh^{(2)}(\mathbb{A}^1_t) \stackrel{j^*}\longrightarrow  \Qcoh^{(2)}(\mathbb{G}_{m,t,t'}) \stackrel{j^*} \longleftarrow \Qcoh^{(2)}(\mathbb{A}^1_{t'})
%$$
The canonical isomorphisms $\mathbb{G}_{m,t, t',\beta} \cong \mathbb{G}_{m,t'}$ and $\mathbb{G}_{m,t,t',\beta} \cong \mathbb{G}_{m,t}$ determine equivalences $$
\Qcoh^{(2)}(\mathbb{G}_{m,t'})  \simeq \Qcoh^{(2)}(\mathbb{G}_{m,t,t',\beta}) \quad \Qcoh^{(2)}(\mathbb{G}_{m,t})  \simeq \Qcoh^{(2)}(\mathbb{G}_{m,t,t',\beta})  
$$

% Assume that $n=|T| =3$.  We denote by $\iota^*$ the functor defined as the composition
% $$
% \iota^*: \mathrm{MF}^\infty(\mathbb{A}^3_T, W_T) \longrightarrow \mathrm{MF}^\infty(D(t), W_T|_{D(t)}) \stackrel{\simeq} \longrightarrow \Qcoh^{(2)}(\mathbb{G}_{m,t})
% $$
% \textcolor{red}{where the second arrow is the preferred Kn\"orrer periodicity equivalence determined by the framing data as in Section \ref{sec:local-model-res}.}

Now we define the sheaf of categories $\cB$:

\begin{enumerate}
\item We define $\mathcal{B}(-)$ on the vertices of $Q_G$ as follows: 
\begin{itemize}
\item If $v$ is a trivalent vertex of $G$, 
$$
\cB_v=\mathcal{B}(\bullet_v) := \mathrm{MF}^\infty(\mathbb{A}^3_{H_v}, \alpha(v)W_{H_{v}})
$$ 
\item If $v$ has valency one, 
$$\mathcal{B}_v=\mathcal{B}(\bullet_v) := \Qcoh^{(2)}(\mathbb{A}^1_{x_v})$$ 
\item If $t$ is a non-compact edge of $G$, and $x$ is the unique half-edge incident to $t$, 
$$
\cB_t=\mathcal{B}(\circ_t) := \Qcoh^{(2)}(\mathbb{G}_{m,x})
$$
\item if $t$ is a compact edge of $G$, and $x$ and $y$ are the two half-edges incident to $t$, we set $\mathcal{B}(\circ_t) := \Qcoh^{(2)}(\mathbb{G}_{m,x,y,\beta(t)})$
\end{itemize}
\item We define  $\mathcal{B}(-)$ on the arrows of $Q_G$ as follows:
\begin{itemize}
\item If $v$ is trivalent and $\bullet_v \stackrel{x} \to \circ_t$ is an arrow,
  $$\cB(x) : \cB(\bullet_v) = \mathrm{MF}^\infty(\mathbb{A}^3_{H_v}, \alpha(v)W_{H_v})  \stackrel{\iota^* } \longrightarrow  \Qcoh^{(2)}(\mathbb{G}_{m,x}) \simeq  \cB(\circ_t)$$
  is the functor defined in Section \ref{sec:local-model-res}, which uses the framings to determine the Kn\"{o}rrer periodicity equivalence.
\item If $v$ has valency one and $\bullet_v \stackrel{x} \to \circ_t$ is an arrow,
$$\cB(x) : \cB(\bullet_v) = \Qcoh^{(2)}(\mathbb{A}^1_x)  \stackrel{j^*} \longrightarrow  \Qcoh^{(2)}(\mathbb{G}_{m,x}) \simeq   \cB(\circ_t)$$
%where we identify $\mathbb{A}^1_{x}$ and $\mathbb{A}^1_{t}$ via $t \mapsto x$. 
\end{itemize}
\end{enumerate}

\begin{remark}
The quiver $Q_G$ is not actually a category. Formally speaking the source of $\mathcal{B}(-)$ is    %rather %the source of this functor is in fact 
the $\infty$-category which the nerve of the 1-category associated with $Q_G$. However since as are no composable arrows in $Q_G$, to define such a functor it is enough to specify what it does on  vertices and arrows. 
\end{remark}

The global sections of $\mathcal{B}$ over $G$ can be computed in the usual way by taking a limit over the quiver $Q_G$
\begin{equation}
\label{dlgbag}
\mathcal{B}(G) \simeq \varprojlim_{q \in Q_G} \cB(q) \, \in \,  \mathrm{DGCat}^{(2)}_{\mathrm{cont}}
\end{equation}

\begin{remark}
\label{atavo}
Let us explain more concretely how to interpret (\ref{dlgbag}).  Assume for simplicity that all vertices of $G$ have valency 3. If $v_1, v_2$ are vertices of $G$  joined by an edge $t$ we have a diagram
$$
\xymatrix{
\cB_{v_1} \times  \cB_{v_2}
\ar@<-.5ex>[r]_-*!/d0.5pt/{\labelstyle}
\ar@<.5ex>[r]^ -*!/u0.7pt/{\labelstyle } 
& \cB_{t} 
}
$$
where the arrows are given by projection followed by restriction. Running over the vertices and edges of $G$ yields  a 
\v{C}ech diagram in $\mathrm{DGCat}^{(2)}_{\mathrm{cont}}$ and $\cB(G)$ is the following equalizer 
\begin{equation}
\label{vvv}
\xymatrix{
\mathcal{B}(G) \to[ \underset{v \in V_G} {\bigoplus} \mathcal{B}_v
\ar@<-.5ex>[r]_-*!/d0.5pt/{\labelstyle }
\ar@<.5ex>[r]^ -*!/u0.7pt/{\labelstyle } 
& \underset{t \in E_G}{\bigoplus} \mathcal{B}_t]
}
\end{equation}
Formula (\ref{dlgbag}) encodes the same information: it states that 
$\mathcal{B}(G)$ is equivalent to the equalizer  
\begin{equation}
\label{bbb}
\xymatrix{
\mathcal{B}(G) \to[ \underset{\bullet_v \in Q_G} {\bigoplus} \mathcal{B}(\bullet_v) 
\ar@<-.5ex>[r]_-*!/d0.5pt/{\labelstyle }
\ar@<.5ex>[r]^ -*!/u0.7pt/{\labelstyle } 
& \underset{\circ_t \in Q_G}{\bigoplus} \mathcal{B}(\circ_t) ]
}
\end{equation}
%where $\bullet_v$ and $\circ_t$ run over all black and white vertices of $Q_G$, 
which is the same as (\ref{vvv}).  Throughout the paper we will switch freely  between the formalisms (\ref{vvv}) and (\ref{bbb}) to compute the global sections of $\cB$, depending on what is more convenient in a given situation. %as sometimes one perspective  might be more convenient than the other. 
\end{remark}
% \begin{example}
% \textcolor{red}{It 
% would be helpful to include a picture of a simple graph and of  the resulting diagram of categories}
% \end{example}

The assignment above allow us to uniquely define the sheaf $\mathcal{B}$ over any graph with 1 or 3-valent vertices and no loops. In more technical terms, $\mathcal{B}$ defines a sheaf over the site made up by these graphs and open inclusions between them.   If $U \subset V$ is an open inclusion of graphs we denote the restrictions functor  
 $ 
 S_\cB: \mathcal{B}(V) \to \mathcal{B}(U). 
 $  
 
We record the following elementary fact.
\begin{lemma}
\label{ambi}
Let $U \subset V$ be an open inclusion of  graphs with vertices of valency one or three, and let $
S_\cB: \cB(V) \to \cB(U)
$ be the restriction functor. Then the left and the right adjoint of $S_\cB$ are naturally equivalent: that is, we have ambidexterous adjunctions  
%Then $S_\cB$ is also equivalent to the right adjoint of $T_\cB$ 
$$
 T_\cB \dashv S_\cB \dashv T_\cB, \quad S_\cB \dashv T_\cB \dashv S_\cB
$$
%Let $U \subset V$ be an open inclusion of trivalent graphs, and let $
%T_\cB: \cB(U) \to \cB(V)
%$ be the right adjoint of the restriction functor $S_\cB$.  
%Then $S_\cB$ is also equivalent to the right adjoint of $T_\cB$ 
%$$
%S_\cB \dashv T_\cB \dashv S_\cB
%$$
\end{lemma}
\begin{proof}
Clearly, it is sufficient to prove that we have  
$$
S_\cB \dashv T_\cB \dashv S_\cB. 
$$
By construction, $\cB(U)$ and $\cB(V)$ can be realized as limits of the categories associated to the vertices $\cB_v$ and the edges $\cB_t$. The corestriction functor $T_\cB$ can be naturally defined in terms  of this presentation. Indeed we can write $T_\cB$ as a limit of basic functors only involving the categories $\cB_v$ and $\cB_t$:  namely $\cB_v \stackrel{=} \to \cB_v$ if $v$ belongs to both $U$ and $V$;  the zero functor $0 \to \cB_v$, or the right adjoint $\cB_t \to \cB_v$  to the restriction  $\cB_v \to \cB_t$, when $v$ is a vertex which does not belong to $U$. 

 It is therefore sufficient to show that these three basic functors admit ambidexterous adjoints. This clearly holds for the identity and the zero functor. So we only need to check  the corestriction $\cB_t  \to \cB_v$. We will explain the case when $v$ is 3-valent, the 1-valent case being similar. Let $X = Spec(\kappa[x, y]/xy)$. Note that we can rewrite this as $$
 \xymatrix{
 \cB_t \ar[r] \ar[d]_-\simeq & \cB_v \ar[d]^-\simeq \\
   \mathrm{IndCoh^{(2)}(\mathbb{G}_m)} 
 \ar[r]^-{i_*} &  \mathrm{IndCoh^{(2)}(X)}}
 $$
 where $i:\mathbb{G}_m \to X$ is the open inclusion of an axis minus the origin. Then the fact that we have an ambidexterous adjunction 
 $ \, 
 i^* \dashv i_* \dashv i^! \simeq i^*
 $ 
 is well-known; a reference is for instance Appendix A.2 of \cite{preygel2011thom} (see the proof of Proposition A.2.3).
 \end{proof}
 
 \begin{remark}
 \label{iwab}
 It will also be useful for us to work with small categories. We will  consider the constructible pre-sheaf of small categories 
 $\mathcal{B}^\omega$ defined as follows  \begin{itemize}
 \item the stalk of $\mathcal{B}^\omega$ on a vertex $v$ of $G$ is given by $\mathcal{B}_v^\omega$ 
 \item the stalk of $\mathcal{B}^\omega$  on an edge $t$ of $G$ is given by $\mathcal{B}_t^\omega$
 \item the restrictions are given by $S_\mathcal{B}^\omega: \mathcal{B}_v^\omega \to \mathcal{B}_t^\omega$ 
 \end{itemize}
 The pre-sheaf $\mathcal{B}^\omega$  is well defined,  since the restrictions functors preserve compact objects: e.g. if $v$ is trivalent  we can rewrite $S_\mathcal{B}^\omega$ in terms of   matrix factorization categories as 
 $$
 \xymatrix{
 \mathcal{B}_{v}^\omega  \ar[d]_\simeq \ar[r] ^-{S_\cB^\omega}  &  \mathcal{B}_{e}^\omega  \ar[d]^\simeq   \\
 \mathrm{MF}   (\mathbb{A}^3_{x,y,z}, W=xyz  )  \ar[r]^-{i^*}   &  
 \Perf^{(2)} (\mathbb{G}_{m}) % &  \mathrm{QCoh}^{(2)}\big (Spec(\kappa[t_1^{\pm 1}]) \big ) &
}
 $$
 %where $i^*$ is clearly well defined.
 If $G$ is a graph, we will denote by  $\mathcal{B}^\omega(G)$ the sections of $\mathcal{B}^\omega$ on $G$. Note that, as we did with $\cB$, we can encode $\cB^\omega$  as a functor $Q_G \to \mathrm{DGCat^{(2)}_{small}} $.

Let us make two  observations on the differences between $\mathcal{B}$ and 
$\mathcal{B}^\omega$. First, $\mathcal{B}^\omega$ is actually a sheaf as long as 
$G$ is finite, but it is only a presheaf if $G$ is not finite. Second, the corestrictions $T_\mathcal{B}$ do not preserve compact objects and therefore   are not well defined in the setting of $\mathcal{B}^\omega$.
\end{remark}

\subsection{Dependence of $\cB(G)$ on the weights and framings}
\label{sec:overall-dependence}
Let $G$ be a graph as above with weights $\alpha : V_{G} \to \kappa^{\times}$ and $\beta : E_{G}\to \kappa^{\times}$ and framings $f$.
Denote by $\cB(G,\alpha,\beta,f)$ the category of global sections constructed above.
\begin{proposition}
  Suppose $G$ is connected. Up to $\kappa\rbu$-linear equivalence, the category $\cB(G,\alpha,\beta,f)$ depends only on the underlying graph $G$ and the product of the weights $\gamma = \prod_{v \in V_{G}} \alpha(v) \prod_{t \in E_{G}}\beta(t)$. If the graph $G$ has a noncompact edge or a $1$-valent vertex, this category does not even depend on the weights.
\end{proposition}

\begin{proof}
  First we show that $\cB(G,\alpha,\beta,f)$ does not depend on the framings $f$. Any two framings $f$, $f'$ may be obtained from one another by a sequence of moves where the orientation of an edge is reversed, or the cyclic ordering at a trivalent vertex is reversed.
  \begin{itemize}
  \item If $f$ and $f'$ differ only in the orientation of the edge $t$, then the diagrams computing $\cB(G,\alpha,\beta,f)$ and $\cB(G,\alpha,\beta,f')$ are equivalent via the functor that acts on $\cB(\circ_{t})$ by the shift $[1]$ and the identity elsewhere.
  \item If $f$ and $f'$ differ only in the cyclic ordering at the vertex $v$, then the diagrams computing $\cB(G,\alpha,\beta,f)$ and $\cB(G,\alpha,\beta,f')$ are equivalent via the functor that acts on $\cB(\bullet_{v})$ by the shift $[1]$ and the identity elsewhere.
  \end{itemize}

  Next we address the dependence on the weights. This has to do with rescaling the coordinates in the presentation. Recall the local model for the restriction
  \begin{equation*}
    R_{i}: \mathrm{MF}^{\infty}(\bA^{3},\alpha(v)x_{1}x_{2}x_{3}) \to \Qcoh^{(2)}(\Spec \kappa[x_{i},\tilde{x}_{i}]/(x_{i}\tilde{x}_{i} - \beta(t_{i}))
  \end{equation*}
  If the variable $x_{i}$ is rescaled to $x_{i}' = \lambda x_{i}$, then $\alpha(v)$ changes to $\alpha'(v) = \lambda^{-1}\alpha(v)$, while $\beta(t_{i})$ becomes $\beta'(t_{i}) = \lambda \beta(t_{i})$. Thus the product of all weights is unchanged. In essence, the weights define a $0$-chain on the quiver $Q_{G}$ with values in $\kappa^{\times}$, $(\alpha,\beta) \in C_{0}(Q_{G},\kappa^{\times})$, while the set of possible rescalings $\lambda$ is the set of $1$-chains on $Q_{G}$, $\lambda \in C_{1}(Q_{G},\kappa^{\times})$. The effect of performing a rescaling is to add the boundary of $\lambda$: $(\alpha',\beta') = (\alpha,\beta) + \partial \lambda$. Thus any two sets of weights that define the same class in $H_{0}(Q_{G},\kappa^{\times}) = \kappa^{\times}$ lead to equivalent categories.

  Moreover, this argument shows that by rescaling, all weights can be set to $1$ except for a single one, which may be at a vertex or an edge. In the case where $G$ has a noncompact edge or a $1$-valent vertex, the weights may be concentrated at the corresponding edge, and it amounts to a specific choice of coordinate along that edge, which may be disregarded after the limit has been taken.
\end{proof}

\begin{definition}
  \label{def:graphcat}
  For a graph $G$ and a weight $\gamma \in \kappa^{\times}$, we shall use the abbreviated notation $\cB(G,\gamma)$ to denote $\cB(G,\alpha,\beta,f)$ where $\alpha$ and $\beta$ are a set of weights whose product is $\gamma$, and $f$ is any choice of framings. The previous proposition implies that this category is well-defined up to $\kappa\rbu$-linear equivalence.
\end{definition}

When comparing $\cB(G,\gamma)$ with the Fukaya category of surface, we take $\alpha(v) = q^{-A(v)}$ and $\beta(t) = q^{B(t)}$ where $A(v)$ is the area of the pair of pants at $v$ and $B(t)$ is the area of the annulus at $t$. We then have
\begin{equation*}
  -\val_{q}(\gamma) = - \val_{q}\left(\prod_{v \in V_{G}} \alpha(v) \prod_{t \in E_{G}}\beta(t)\right) = \sum_{v \in V_{G}} A(v) - \sum_{t\in E_{G}}B(t),
\end{equation*}
which is the total area of the surface: the sum of the areas of the pairs of pants double-counts the area of each annulus.

When comparing $\cB(G,\gamma)$ with categories of matrix factorizations, it is useful to think of the parameter $\gamma$ as twisting the $2$-periodicity structure. That is to say, the categories $\cB(G,\gamma)$ are not necessarily equivalent as $\kappa\rbu$-linear categories ($\deg u = 2$), but they are equivalent as $\kappa$-linear categories. More precisely, $\cB(G,\gamma)$ is equivalent to $\cB(G,1)$ with the substitution $u \mapsto \gamma u$.

\begin{remark}
  \label{rem:segal-twist}
  As shown in \cite{segal2021line}, it is possible to twist the construction of $\cB(G)$ using a $\Z/2\Z$-valued cocycle $\gamma \in C^{1}(G,\Z/2\Z)$. The effect is to modify the gluing along the edge $t$ to incorporate a shift $[\gamma(t)]$. The result depends up to equivalence on the cohomology class $[\gamma] \in H^{1}(G,\Z/2\Z)$. We refer the reader to \cite{segal2021line} for further discussion of these twists, as well as the $\Z$-graded case where these twists have a more fundamental importance.
\end{remark}

\section{Topological Fukaya category  and pants decomposition}
The main technical result of \cite{pascaleff2019topological} was a description of the topological Fukaya category of a punctured Riemann surface of finite type in terms of pants decomposition. In this section we revisit that result, and extend it to some non-finite type Riemann surfaces arising as covers of compact Riemann surfaces. In the non-finite type setting the kind of Fukaya category which will be most relevant for us is the \emph{locally finite} Fukaya category, which will be introduced in Section \ref{lffc}. We remark that not all the results in this section will actually be needed for the later sections. However we believe that giving an account of the theory of \cite{pascaleff2019topological} in the more general setting of non-finite type Riemann surfaces might be of some independent interest. 

In this section all of the weights are set to $1$. When dealing with topological Fukaya categories of open Riemann surfaces, this does not entail a loss of generality.

\subsection{Riemann surfaces and maximal tropical covers}
\label{rsamtc}
Let $\Sigma$ be a compact Riemann surface of genus $g \geq 2$, equipped with a pants decomposition $\cP$. We can encode the pants decomposition $P$ into a trivalent  graph $G$. The graph $G$ is the dual intersection complex of the pants decomposition: that is, vertices of $G$ are in bijection with the pants making up the decomposition, and the half-edges incident to a vertex 
$v$ correspond to  the boundary components of the corresponding pair of pants $\cP_v$; two vertices $v_1$, $v_2$ are connected by an edge if the pants $\cP_{v_1}$ and $\cP_{v_2}$ are glued along the corresponding boundary components.  For simplicity, we will assume that no pairs-of-pant in  $\mathcal{P}$ self-intersects, i.e. $G$ contains no loops.  The graph $G$ is a kind of combinatorial tropicalization of $\Sigma$. In particular we can define a map of topological spaces $\Sigma \to G$, such that the preimages of the vertices of $G$ are homeomorphic to pairs-of-pants, and the preimages of points in the interior of the edges of $G$ are $S^1$.

We will work with an infinite-sheeted unramified cover $\widetilde{\Sigma}$ of $\Sigma$, which we define next.  Let $\Sigma \to G$ be the tropicalization map. Let 
$\widetilde{G}$ be the maximal abelian cover of $G$, and let $q:\widetilde{G} \to G$ be the covering map. Then $\widetilde{\Sigma}$ can be defined as the fiber product in spaces 
$$
\xymatrix{
\widetilde{\Sigma} \ar[r]^-\pi \ar[d] & \Sigma \ar[d] \\
\widetilde{G} \ar[r]^-{q} & G}
$$
Equivalently, we can consider the natural homomorphism $\pi_1(\Sigma) \to 
\pi_1(G)$: then the regular covering space $\widetilde{\Sigma}$ corresponds to quotient of $\pi_1(\Sigma)$ given by 
$$
\pi_1(\Sigma) \to 
\pi_1(G) \to \pi_1(G)/[\pi_1(G), \pi_1(G)] \cong \mathbb{Z}^g
$$

\begin{definition}
We call $\pi: \widetilde{\Sigma} \to \Sigma$ a \emph{maximal tropical  cover} of $\Sigma$.
\end{definition}

By construction $\widetilde{\Sigma}$ is an infinite type Riemann surface that comes equipped with a regular covering map
$
\pi: \widetilde{\Sigma} \to \Sigma
$ 
with group of deck transformation  given by  $O_\pi:=\mathbb{Z}^{g}.$ A pants 
decomposition of $\Sigma$ induces a pants decomposition $\widetilde{\mathcal{P}}$ of 
$\widetilde{\Sigma}$. The trivalent graph $\widetilde{G}$ encoding the pants decomposition $\widetilde{\mathcal{P}}$ is   $\widetilde{G}$. The action by the group of deck transformations  $\mathbb{Z}^{g}$ preserves the pants decomposition.

Since 
$\widetilde{\Sigma}$ is an infinite-genus Riemann surface, it is in particular non-compact. By a classical result of Behnke and Stein \cite{behnke1947entwicklung}  all non-compact Riemann surfaces are Stein. As a consequence, we can study the Fukaya category of $\widetilde{\Sigma}$ via the topological models which become available in the Stein setting. In particular $\widetilde{\Sigma}$ can be equipped with a skeleton $\widetilde{\Gamma}$, which is an infinite but locally finite ribbon graph:  $\widetilde{\Gamma}$ is homotopy equivalent to an enumerable wedge of circles. 

% \textcolor{red}{Since we will be ultimately interested in a more intrinsic description of the \mathrm{top}ological Fukaya category of $\widetilde{\Sigma}$ which is not directly tied to the actual shape of the skeleton, we will limit ourselves to describe as an example a  skeleton for $\widetilde{\Sigma}$, when $\Sigma$ is genus 2. }

% \begin{example}
% \textcolor{red}{example on genus $2$ curve}
% \end{example}

\subsubsection{Covers by finite type subsurfaces}
\label{cbfts}
 Let $\mathcal{U} = \{ \widetilde{\Sigma}_i\}_{i \in \mathbb{N}}$ be an cover of $\widetilde{\Sigma}$ by open subsurfaces of finite type having the following properties:
\begin{enumerate}
\item for all $i$, $\widetilde{\Sigma}_i \subset \widetilde{\Sigma}_{i+1}$
\item $\widetilde{\Sigma} = \bigcup_{i \in \mathbb{N}} \widetilde{\Sigma}_i$
\item for all $i$, $\widetilde{\Sigma}_i$ is a connected open Riemann surface of finite genus, which is the interior of a union of pants $P$ in 
$\widetilde{\mathcal{P}}$.  
\end{enumerate}
We say that $\mathcal{U}$ is \emph{compatible} with $\widetilde{\mathcal{P}}$. 

\begin{definition}
Let $\widetilde{\Gamma}$ be a skeleton of $\widetilde{\Sigma}$. We say that $\widetilde{\Gamma}$ is \emph{adapted to $\mathcal{U}$} if for all $i \in \mathbb{N}$, $\widetilde{\Gamma}$ intersects transversely the boundary of $\widetilde{\Sigma}_i$ and 
$$
\Gamma_i:= \widetilde{\Sigma}_i \cap \widetilde{\Gamma}
$$ 
is a skeleton for $\widetilde{\Sigma}_i$, possibly with \emph{stops}: i.e. edges  incident to  the boundary $\partial \widetilde{\Sigma}_i$. 
\end{definition}

\begin{lemma}
Given an open cover $\mathcal{U}$ of $\widetilde{\Sigma}$ satisfying properties $(1), (2)$ and $(3)$ above there exist skeleta adapted to $\mathcal{U}$. 
\end{lemma}
\begin{proof}
For concreteness we will briefly explain the case $g=2$, where we work with the minimal pants decomposition. The general case can be dealt with similarly. We make first a preliminary observation. Let $\widetilde{\Sigma}$ be a non-compact Riemann surface such that can be written as 
$ 
\widetilde{\Sigma} = \widetilde{\Sigma}_1 \coprod_Z \widetilde{\Sigma}_2
$ 
where 
\begin{enumerate}
\item $\widetilde{\Sigma}_1$ and $\widetilde{\Sigma}_2$ are Riemann surfaces with boundary, and $Z \subset \partial \widetilde{\Sigma}_1 \cap \partial \widetilde{\Sigma}_2$
\item $\widetilde{\Sigma}_1$ and $\widetilde{\Sigma}_2$ are equipped with skeleta 
$\Gamma_1$ and $\Gamma_2$ such that $Z \subset \Gamma_1 \cap \Gamma_2$. \end{enumerate}
Then 
$ 
\Gamma:=\Gamma_1 \coprod_Z \Gamma_2
$ 
is a is a skeleton for $\widetilde{\Sigma}$.

Now let $\Sigma$ be a compact 
Riemann surface of genus $2$. Let $\widetilde{\Sigma} \to \Sigma$ be  its maximal tropical cover. There is unique pants decomposition of $\Sigma$ made up of two pairs of pants which do not self-intersect: we denote these pairs-of-pants $\Sigma_1$ and $\Sigma_2$ and write $\Sigma=\Sigma_1 \cup_D  \Sigma_2$, where $D \cong S^1 \amalg S^1 \amalg S^1$. 
The maximal abelian cover of $\Sigma_i$ can be described explicitly as follows.

Consider the equilateral simplicial tessellation of $\mathbb{R}^2$. We can view its boundary as a locally finite ribbon graph $\Delta$ embedded in $\mathbb{R}^2$: let $T$ be a spanning tree for this graph. Let $B$ be the union of the open balls $B(x, \epsilon)$ with radius $\epsilon << 1$, where $x$ ranges over the vertices of $\Delta$. Let $\partial B$ be the boundary of $B$: this is a enumerable union of circles with centers in the vertices of $\Delta$.  Then $\widetilde{\Sigma}_i$ can be realized as 
$\mathbb{R}^2 - B$. A skeleton for $\widetilde{\Sigma}_i$ is given by 
$$
\Gamma_i=(T-D) \bigcup \partial B. 
$$
In words, $\Gamma_i$ is a connected ribbon graph obtained by drawing edges between the components of $\partial B$, but in such a way that no new loops are created. 

The covering $\widetilde{\Sigma} \to \Sigma$ can be realized via the identification 
$$
\widetilde{\Sigma} \cong \widetilde{\Sigma}_1 \coprod_{\partial B} \widetilde{\Sigma}_1 
\longrightarrow \Sigma_1 \coprod_{D} \Sigma_2 \cong \Sigma. 
$$
Note that $\widetilde{\Sigma}$ carries a natural pants decomposition 
$\widetilde{\mathcal{P}}$. Also $\widetilde{\Sigma}$ inherits a skeleton $\widetilde{\Gamma} = \Gamma_1 \coprod_{\partial B} \Gamma_2$. Then the skeleton $\widetilde{\Gamma} $ is adapted to $\mathcal{U}$, for any open cover $\mathcal{U}$ compatible with $\widetilde{\mathcal{P}}$. 

Indeed, let $U$ be an open subset of $\widetilde{\Sigma}$  which is the interior of a union of pants. Then $U \cap \Gamma_i$ is a skeleton (possibly with non-compact edges) of $U \cap \widetilde{\Sigma_i}$: indeed $U \cap \widetilde{\Sigma_i}$ retracts onto a bouquet of $N$ circles for some $N$, and  $\Gamma_i \cap \widetilde{\Sigma_i}$ has the same homotopy type, which can be easily checked by inspection. Then it follows that 
$$
U \cong U \cap \widetilde{\Sigma_1} \coprod_{U \cap \partial B} U \cap \widetilde{\Sigma_2}
 \quad \text{retracts onto} \quad 
U \cap \Gamma = U \cap \Gamma_1 \coprod_{U \cap \partial B}  U \cap \Gamma_2$$ which is what we needed to prove. 
\end{proof}
\begin{remark}
\label{ntim}
If $\mathcal{U} = \{\widetilde{\Sigma}_i\}_{i \in \mathbb{N}}$ is an open cover as above, and $\widetilde{\Gamma}$ is a skeleton adapted to $\mathcal{U}$, then $\widetilde{\Gamma}$ will have at least one stop on each  boundary component of $\widetilde{\Sigma}_i$, for all $i$. 
That is,  for each connected component 
$B$ of $\partial \widetilde{\Sigma}_i$, the intersection $\Gamma_i \cap B$ is non-empty.
\end{remark}

 \subsection{The topological Fukaya category of $\widetilde{\Sigma}$}
 \label{tfc}
 In this Section we explain how to compute the Fukaya category of $\widetilde{\Sigma}$ via the sheaf of categories on graphs introduced in Section \ref{pdac}. This follows closely \cite{pascaleff2019topological} except for the fact that, since $\widetilde{\Sigma}$ is not of finite type, some extra argument is needed. 
 
 We fix once and for all a covering $\mathcal{U} = \{\widetilde{\Sigma}_i\}_{i \in \mathbb{N}}$ compatible with $\cP$, and a skeleton $\Gamma$ adapted to $\mathcal{U}$. 
It follows  immediately  from the definition that 
$$
\mathrm{\mathrm{Fuk}}^{\mathrm{top}}(\widetilde{\Sigma}) = \varinjlim_i \cF^{\mathrm{top}}(\Gamma_i), \quad 
\mathrm{\mathrm{Fuk}}^{\mathrm{top}}_\infty(\widetilde{\Sigma}) = \varinjlim_i \cF^{\mathrm{top}}_\infty(\Gamma_i). 
$$

By definition $\Gamma_i$ is a skeleton for $\widetilde{\Sigma}_i$, possibly with stops. The closure of $\widetilde{\Sigma}_i$ is a Riemann surface with boundary equipped with a pants decomposition. Let $G_i \subset \widetilde{G}$ be the open sub-graph encoding the pants decomposition of $\widetilde{\Sigma}_i$. We can evaluate the sheaf $\mathcal{B}(-)$ on $G_i$. By Theorem 8.3 of \cite{pascaleff2019topological} we have that $\mathcal{B}(G_i)$ is equivalent to the Ind-completion of the topological Fukaya category of $\widetilde{\Sigma}_i$
$$
\mathcal{B}(G_i) \simeq \mathrm{Ind}(\mathrm{\mathrm{Fuk}}^{\mathrm{top}}(\widetilde{\Sigma}_i))
$$
Now let  
$\Gamma_i'$ be the compact skeleton of $\widetilde{\Sigma}_i$ obtained by removing the non-compact edges of $\Gamma_i$. Since the topological Fukaya category of a surface can be computed by evaluating $\cF^{\mathrm{top}}$ on a compact skeleton we have that $\cB(G_i) \simeq \cF^{\mathrm{top}}_{\infty}(\Gamma_i')$.

We have inclusions 
$$
\Gamma_i' \subset \Gamma_i, \quad \Gamma_i' \subset \Gamma_{i+1}', \quad \Gamma_i \subset \Gamma_{i+1}' 
$$ 
where $\Gamma_i'$ is a closed subgraph of $\Gamma_i$ and of $\Gamma_{i+1}'$, while $\Gamma_i$ is an open subgraph of $\Gamma_{i+1}'$. This yields functors 
$$
T_\infty: \cB(G_i) \simeq \mathrm{\mathrm{Fuk}}^{\mathrm{top}}_\infty(\widetilde{\Sigma}_i) \longrightarrow \cF^{\mathrm{top}}_\infty(\Gamma_i), \quad T_\infty: \cB(G_i) \simeq \mathrm{\mathrm{Fuk}}^{\mathrm{top}}_\infty(\widetilde{\Sigma}_i) \longrightarrow \cB(G_{i+1}) \simeq \mathrm{\mathrm{Fuk}}^{\mathrm{top}}_\infty(\widetilde{\Sigma}_{i+1}), 
$$
$$
C_\infty: \cF^{\mathrm{top}}_\infty(\Gamma_i) \longrightarrow \cB(G_{i+1}) \simeq \mathrm{\mathrm{Fuk}}^{\mathrm{top}}_\infty(\widetilde{\Sigma}_{i+1})
$$
Note that $S_\infty: \cB(G_{i+1}) \to \cB(G_i)$ coincides with the restriction functor $S_\cB$ of the sheaf $\cB$ discussed in Section \ref{pdac}. Thus $T_\infty: \cB(G_i) \to \cB(G_{i+1})$ is the same as the corestriction functor $T_\cB$ introduced in Lemma \ref{ambi}.  
For all $i$, we have factorizations
% \footnote{\textcolor{red}{this should be obvious, at least looking at the B-side version of the statement where we are just saying that the composition of push-forwards is push-forward of the composition}}
$$
\xymatrix{
\cF^{\mathrm{top}}_\infty(\Gamma_i) \ar[r]^-{C_\infty}   \ar@/^15pt/@{->}[rr]^{C_\infty}   & \cB(G_{i+1}) \ar[r]^-{T_\infty} & \cF^{\mathrm{top}}_\infty(\Gamma_{i+1})
}
\quad 
\xymatrix{
\cB(G_i) \ar[r]^-{T_\infty}   \ar@/^15pt/@{->}[rr]^{T_\infty}   & \cF^{\mathrm{top}}_\infty(\Gamma_{i}) \ar[r]^-{C_\infty} & \cB(G_{i+1})
}
$$
 \begin{proposition}
 \label{tiae}
There is an equivalence 
$$ 
\mathrm{\mathrm{Fuk}}^{\mathrm{top}}_\infty(\widetilde{\Sigma}) \simeq \varinjlim_i \cF^{\mathrm{top}}_\infty(\Gamma_i) \simeq \varinjlim_i \cB(G_{i})
$$ 
\end{proposition}
\begin{proof}
Consider the directed system in $\mathrm{DGCat}^{(2)}_{\mathrm{cont}}$ 
$$
C_1:=\mathcal{B}(G_1) \stackrel{T_\infty} \to C_2 :=  \cF^{\mathrm{top}}_\infty(\Gamma_1) \stackrel{C_\infty} \to 
C_3:=\mathcal{B}(G_2) \stackrel{T_\infty} \to C_4:=\cF^{\mathrm{top}}_\infty(\Gamma_2) \stackrel{C_\infty} \to \ldots
$$
Then by cofinality
$$
\varinjlim_i \cB(G_i) \simeq \varinjlim_i C_i \simeq \varinjlim_i \cF^{\mathrm{top}}_\infty(\Gamma_i) \simeq \mathrm{\mathrm{Fuk}}^{\mathrm{top}}_\infty(\widetilde{\Sigma}).
$$
\end{proof}
Now consider the inductive system in $\mathrm{DGCat}^{(2)}_{\mathrm{cont}}$ 
$$
\cB(G_{1}) \stackrel{S_\infty}\leftarrow \cB(G_{2}) \stackrel{S_\infty}\leftarrow \cB(G_{3}) \stackrel{S_\infty}\leftarrow   \ldots
$$
\begin{proposition}
\label{tiasf}
There is an equivalence 
$$ 
\mathrm{\mathrm{Fuk}}^{\mathrm{top}}_\infty(\widetilde{\Sigma}) \simeq \varprojlim_i \cB(G_{i}) \simeq \cB(\widetilde{G})$$
\end{proposition}
\begin{proof}
%By Lemma \ref{libas} limits of small diagrams that lie both in $\mathrm{DGCat}^{(2)}_{\mathrm{cont}}$ and in  $\mathrm{DGCat}^{(2)}_{\mathrm{cocont}}$ (i.e. such that the structure morphisms preserve small limits and colimits) can be computed equivalently in 
%$\mathrm{DGCat}^{(2)}_{\mathrm{cont}}$ and in  $\mathrm{DGCat}^{(2)}_{\mathrm{cocont}}$. 
%
By Lemma \ref{ambi}, $S_\infty$ is also the right adjoint of $T_\infty$.  By Proposition \ref{tiae}, we have that  
 $\mathrm{\mathrm{Fuk}}^{\mathrm{top}}_\infty(\widetilde{\Sigma}) \simeq  \varinjlim_i \cB(G_{i})$, 
where the colimit is computed in $\mathrm{DGCat}^{(2)}_{\mathrm{cont}}$. By Lemma \ref{libas} taking right adjoint yields   
$$
\mathrm{\mathrm{Fuk}}^{\mathrm{top}}_\infty(\widetilde{\Sigma}) \simeq \varprojlim_i \cB(G_{i}) \, \, \text{in} \, \, \mathrm{DGCat}^{(2)}_{\mathrm{cont}}. 
$$
%As we have explained, the limit on the right hand side can be equivalently be computed in $\mathrm{DGCat}^{(2)}_{\mathrm{cont}}$.  
Now, since $\{G_i\}_{i \in \mathbb{N}}$ is an open cover of $\widetilde{G}$, the limit $\varprojlim_i \cB(G_{i})$ computes the global sections of $\cB$, and is therefore equivalent to 
$\cB(\widetilde{G})$. 
%This concludes the proof. 
\end{proof}

As a consequence of Proposition \ref{tiasf} we can rewrite 
$\mathrm{\mathrm{Fuk}}^{\mathrm{top}}_\infty(\widetilde{\Sigma})$ as a limit of a diagram in $\mathrm{DGCat}^{(2)}_{\mathrm{cont}}$ having
\begin{enumerate}
\item as vertices, the categories $\cB_v$ and $\cB_e$, where $v$ and $e$ vary  over the vertices and edges of $\widetilde{G}$
\item as structure morphisms, the restrictions 
$\cB_v \to \cB_e$, whenever $e$ is incident to $v$. 
\end{enumerate}
More precisely, let us  denote by $\widetilde{L}$ be the poset of strata of $\widetilde{G}$, where $\widetilde{G}$ is stratified by its vertices and edges.  Since 
$\mathcal{B}$ is a constructible sheaf we can compute its global sections $\cB(\widetilde{G})$ by taking a limit over its stalks on strata: this limit is naturally indexed by the poset $\widetilde{L}$. We can express this by writing
$$
\mathrm{\mathrm{Fuk}}^{\mathrm{top}}_\infty(\widetilde{\Sigma}) \simeq \cB(\widetilde{G}) \simeq 
\varprojlim_{l \in \widetilde{L}} \cB_l
$$
where $l \in \widetilde{L}$ correspond to the strata of $\widetilde{G}$ (i.e. its vertices and edges), and $\cB_l$ denotes the stalk of $\cB(-)$ on the stratum $l$. 
%More compactly, as we pointed out in Section \ref{pdac} we can write 
%$\mathrm{\mathrm{Fuk}}^{\mathrm{top}}_\infty(\widetilde{\Sigma})$, as the equalizer of the diagram $$
%\xymatrix{
%\prod_{v \in \widetilde{G}} 
%\cB_{v} 
%\ar@<-.5ex>[r]_-*!/d0.5pt/{\labelstyle }
%\ar@<.5ex>[r]^ -*!/u0.7pt/{\labelstyle } 
%& \prod_{e \in \widetilde{G}} \cB_{e}. 
%}
%$$

The group $O_\pi$ of deck transformations of $\pi:\widetilde{\Sigma} \to \Sigma$ acts on $\mathrm{\mathrm{Fuk}}^{\mathrm{top}}_\infty(\widetilde{\Sigma})$. We will use the presentation of $\mathrm{\mathrm{Fuk}}^{\mathrm{top}}_\infty(\widetilde{\Sigma})$ we have just given  to compute the $O_\pi$-equivariant category $\mathrm{\mathrm{Fuk}}^{\mathrm{top}}_\infty(\widetilde{\Sigma})^{O_\pi}$. Recall that 
$O_\pi$ is also the group of deck transformation of the maximal abelian cover $q:\widetilde{G} \to G$.
\begin{proposition}
\label{fiwso}
There is an equivalence 
$$
\mathrm{\mathrm{Fuk}}^{\mathrm{top}}_\infty(\widetilde{\Sigma})^{O_\pi} \simeq \cB(\widetilde{G})^{O_\pi} \simeq \cB(G)
$$
\end{proposition}
\begin{proof}
As we explained, $\mathrm{\mathrm{Fuk}}^{\mathrm{top}}_\infty(\widetilde{\Sigma})$ can be realized as the limit of a diagram of categories in $\mathrm{DGCat}^{(2)}_{\mathrm{cont}}$ indexed by  the poset of strata of $\widetilde{G}$, which we denoted $\widetilde{L}$. Additionally, the $O_\pi$-action 
on $\mathrm{\mathrm{Fuk}}^{\mathrm{top}}_\infty(\widetilde{\Sigma})$ can be easily understood via this presentation. Namely the group $O_\pi$ acts on $\widetilde{L}$, and therefore on the right-hand side of 
$$
\mathrm{\mathrm{Fuk}}^{\mathrm{top}}_\infty(\widetilde{\Sigma})  \simeq 
\varprojlim_{l \in \widetilde{L}} \cB_l
$$
and this induces the action of $O_\pi$ on $\mathrm{\mathrm{Fuk}}^{\mathrm{top}}_\infty(\widetilde{\Sigma})$.
Note a that clearly $\widetilde{L}/G \cong L$, where $L$ is the poset of strata of $G$. Note also that 
$$
\cB(G) \simeq \varprojlim_{l \in  L} \cB_l
$$
Since passing to the invariant category is a limit, and limits commute with limits, we can write  
$$
\cB(\widetilde{G})^{O_{\pi}} \simeq \big (\varprojlim_{l \in \widetilde{L}} \cB_l \big )^{O_{\pi}} \simeq 
\varprojlim_{l \in (\widetilde{L}/O_{\pi}) } \cB_l \simeq \cB(G)
$$
and this concludes the proof.
\end{proof}

\subsection{The locally finite topological Fukaya category}
\label{lffc}
The account of the topological Fukaya category of non-finite type Riemann surfaces which we have given so far focused on the  Ind-completed Fukaya category. However for the applications we have in mind it will be  useful to work with small  categories. It turns out that for our purposes the topological Fukaya category of $\widetilde{\Sigma}$ 
$$
\mathrm{\mathrm{Fuk}}^{\mathrm{top}}(\widetilde{\Sigma}) = \mathrm{\mathrm{Fuk}}^{\mathrm{top}}_\infty(\widetilde{\Sigma})^\omega  \simeq \cB(\widetilde{G})^\omega
$$
 is not quite the right category, it is too small. In particular, although the $O_{\pi}$-action restricts to $\mathrm{\mathrm{Fuk}}^{\mathrm{top}}(\widetilde{\Sigma})$,  the $O_{\pi}$-equivariant category $\mathrm{\mathrm{Fuk}}^{\mathrm{top}}(\widetilde{\Sigma})^{O_{\pi}}$ is the zero-category.

 The key to resolving the issue is the following observation: whereas $\cF^{\mathrm{top}}_\infty(-)$ with restrictions   $R_\infty$ defines a sheaf with values in 
$\mathrm{DGCat^{(2)}_{cont}}$ (see Section \ref{potfc} (3)), this is no longer true after restricting to compact objects. To fix this we need to sheafify the assignment 
\begin{equation}
\label{presheaf}
U \subset \widetilde{\Gamma}\mapsto (\cF_\infty(U))^\omega \in \mathrm{DGCat^{(2)}_{small}}. 
\end{equation} 
This yields a sheaf $\cF^{\mathrm{lf}}(-)$ with values in 
$\mathrm{DGCat^{(2)}_{small}}$:  the superscript \emph{lf} stands for \emph{locally finite}, meaning that locally on the skeleton the sections of $\cF^{\mathrm{lf}}(-)$ are the compact objects inside  $\cF^{\mathrm{top}}_{\infty}(-)$. 

\begin{definition}
\label{wctgs}
We call the global sections of the sheaf $\cF^{\mathrm{lf}}(-)$ the \emph{locally finite} topological Fukaya category of $\widetilde{\Sigma}$, and we denote it $
\mathrm{\mathrm{Fuk}}^{\mathrm{top}, \mathrm{lf}}(\widetilde{\Sigma})$. 
\end{definition}

\begin{remark}
The category $
\mathrm{\mathrm{Fuk}}^{\mathrm{top}, \mathrm{lf}}(\widetilde{\Sigma})$ is, by construction, the Morita dual of $\mathrm{\mathrm{Fuk}}^{\mathrm{top}}(\widetilde{\Sigma})$: that is 
$$
\mathrm{\mathrm{Fuk}}^{\mathrm{top}, \mathrm{lf}}(\widetilde{\Sigma}) \simeq \mathrm{Fun}_{\mathrm{DGCat^{(2)}_{small}}}(\mathrm{\mathrm{Fuk}}^{\mathrm{top}}(\widetilde{\Sigma}), \Perf^{(2)}(k)). 
$$
As such it was considered in \cite{sibilla2014ribbon} and \cite{dyckerhoff2018triangulated} as a model for the compact, or \emph{infinitesimally wrapped} Fukaya category of a punctured Riemann surface.
\end{remark}

We are going to prove for $
\mathrm{\mathrm{Fuk}}^{\mathrm{top}, \mathrm{lf}}(\widetilde{\Sigma})$ some of the structure results we have established for $
\mathrm{\mathrm{Fuk}}^{\mathrm{top}}_\infty(\widetilde{\Sigma})$ in the previous section. Since the arguments are very similar for the two cases, we will limit ourselves to a somewhat abbreviated treatment.

Recall from Remark \ref{iwab} that  $\cB^\omega$ is 
the presheaf of small categories obtained from $\mathcal{B}$ by taking compact objects section-wise.
 \begin{proposition}
\label{flws}
There is an equivalence
$$
\mathrm{\mathrm{Fuk}}^{\mathrm{top}, \mathrm{lf}}(\widetilde{\Sigma})  \simeq \varprojlim_i \cB^\omega (G_{i}) \simeq \varprojlim_{i} \mathrm{Fuk}^{\mathrm{top}}(\widetilde{\Sigma}_i)
$$
\end{proposition} 
\begin{proof}
By Remark \ref{ntim} and Remark \ref{luvg}, for all $i$ the category 
$\cF^{\mathrm{top}}(\Gamma_i)$ is a smooth and proper, and  the restriction functor  
$R_\infty: \cF^{\mathrm{top}}_\infty(\Gamma_{i+1}) \to \cF^{\mathrm{top}}_\infty(\Gamma_{i+1})$ restricts to compact objects, that is we can write 
$R_\infty = \mathrm{Ind}(R)$. 
By definition we can write 
$$
\mathrm{\mathrm{Fuk}}^{\mathrm{top}, \mathrm{lf}}(\widetilde{\Sigma}) \simeq \varprojlim_i \cF^{\mathrm{top}}(\Gamma_i). 
$$

As we discussed, there is an equivalence 
$ \cB^\omega(G_{i}) \simeq \cF^{\mathrm{top}}(\Gamma_{i}')$.  
This yields factorizations $$
\xymatrix{
\cF^{\mathrm{top}}(\Gamma_i)     & \cB(G_{i+1}) \ar[l]_-{R}  & 
\cF^{\mathrm{top}}(\Gamma_{i+1}) \ar[l]_-{S} 
 \ar@/_18pt/@{->}[ll]_{R}
}
\quad 
\xymatrix{
\cB(G_i)^\omega   & \cF^{\mathrm{top}}(\Gamma_{i})  \ar[l]_-{S}  & \cB(G_{i+1})^\omega   \ar[l]_-{R} \ar@/_18pt/@{->}[ll]_{S_{\mathcal{B}} ^\omega} 
}
$$
Then we can conclude by a cofinality argument exactly as in the proof of Proposition \ref{tiae}, by looking at the inverse system
$$
 \mathcal{B}(G_1)^\omega \stackrel{S} \leftarrow  \cF^{\mathrm{top}}(\Gamma_1) \stackrel{R} \leftarrow 
 \mathcal{B}(G_2)^\omega \stackrel{S} \leftarrow  \cF^{\mathrm{top}}(\Gamma_2) \stackrel{R} \leftarrow \ldots
$$
\end{proof}
\begin{remark}
  \label{rem:lf-as-projlim}
Note that there are fully-faithful embeddings
$$
\mathrm{\mathrm{Fuk}}^{\mathrm{top}}(\widetilde{\Sigma}) \subset \mathrm{\mathrm{Fuk}}^{\mathrm{top}, \mathrm{lf}}(\widetilde{\Sigma}) \subset \mathrm{\mathrm{Fuk}}^{\mathrm{top}}_{\infty}(\widetilde{\Sigma}). 
$$
The second inclusion is an immediate consequence of Proposition \ref{flws}. As for the first inclusion, recall that by Proposition \ref{tiasf} $
\mathrm{\mathrm{Fuk}}^{\mathrm{top}}(\widetilde{\Sigma})  \simeq  \big ( \varprojlim \cB(G_{i}) \big )^\omega$. 
Then by Lemma \ref{libasi}, 
$$
\mathrm{\mathrm{Fuk}}^{\mathrm{top}}(\widetilde{\Sigma}) \subset  \varprojlim \cB(G_{i})^\omega \simeq \mathrm{\mathrm{Fuk}}^{\mathrm{top}, \mathrm{lf}}(\widetilde{\Sigma}). 
$$
% \textcolor{red}{we might introduce here for illustration the case of 
% $\Perf([\mathbb{P}^1/\mathbb{G}_m])$}
\end{remark}

\begin{proposition}
\label{toamrt}
The $O_\pi$-action on $\mathrm{\mathrm{Fuk}}^{\mathrm{top}}_{\infty}(\widetilde{\Sigma})$ restricts to $\mathrm{\mathrm{Fuk}}^{\mathrm{top}, \mathrm{lf}}(\widetilde{\Sigma})$ and there is an equivalence 
$$
\mathrm{\mathrm{Fuk}}^{\mathrm{top}, \mathrm{lf}}(\widetilde{\Sigma})^{O_\pi} \simeq \cB^\omega(G)
$$
\end{proposition}
\begin{proof}
The proof is  the same as the proof of Proposition \ref{fiwso}, from which we import notations. 
First note that we can rewrite the equivalence from Proposition 
\ref{flws} as
$$
\mathrm{\mathrm{Fuk}}^{\mathrm{top}, \mathrm{lf}}(\widetilde{\Sigma})  \simeq   \varprojlim_{l \in \widetilde{L}} (\cB_l)^\omega
$$
As in the proof of  Proposition \ref{fiwso}, this allows us to understand very concretely the $O_\pi$-action: indeed the group $O_\pi$ acts on $\widetilde{L}$, and therefore on the limit   $\varprojlim_{l \in \widetilde{L}} \cB_l$. Now $\widetilde{L}/G \cong L$, where $L$ is the poset of strata of $G$, and we  conclude  that 
$$
\mathrm{\mathrm{Fuk}}^{\mathrm{top}, \mathrm{lf}}(\widetilde{\Sigma})^{O_\pi} \simeq 
\varprojlim_{l \in (\widetilde{L}/O_{\pi}) } (\cB_l)^\omega \stackrel{(*)}\simeq \cB^\omega(G). 
$$
Let us comment briefly on equivalence $(*)$: since $\cB^\omega$ is in general just a pre-sheaf, it is not a priori clear that its sections on $G$ should satisfy descent, and therefore could be expressed as $\varprojlim_{l \in L } (\cB_l)^\omega $. However  $\cB^\omega$ is indeed a sheaf when  $G$ is finite, as in our case: see Remark \ref{iwab}. 
\end{proof}

%% file: completionsII.tex
%\title{Fukaya categories of higher-genus surfaces and pair of pants decompositions}
%\author{James Pascaleff}
%\author{Nicol\`{o} Sibilla}
%\maketitle
\section{Singularity categories and graphs}
\label{mfat}
In this section we explain how to express the singularity category of a 
normal crossing surface  in terms of the combinatorial sheaves of categories on graphs introduced in Section \ref{pdac}. This is a key ingredient in the proof of HMS for compact surfaces which will be carried out in Section \ref{hms}.
\subsection{Singularity categories of normal crossing surfaces}
\label{sconcs}
Let $T$ be a smooth variety of dimension $3$ and let $X \subset T$ be a simple normal crossing divisor of the form $X = f^{-1}(0)$, where $f: T \to \bA^{1}$ is a morphism. Denote by $Z$ the singular locus of $X$, and by $S$ the singular locus of $Z$. %Let $p: \widetilde{Z} \to Z$ be the normalization of $Z$. 
We make the following assumptions
\begin{enumerate}
\item The irreducible component of $Z$ are rational curves  isomorphic to either $\mathbb{G}_m$, $\mathbb{A}^1$ or $\mathbb{P}^1$
\item Let $C$ be a irreducible component of $Z$.   %, and let $\widetilde{C}$ be its normalization. 
Then the intersection between 
$C$ and $S$ is empty if $C \cong \mathbb{G}_m$, has cardinality at most one if $C \cong \mathbb{A}^1$, and has cardinality at most two if $C \cong \mathbb{P}^1$
\end{enumerate}
\begin{definition}
If these assumptions are satisfied we say that $X$ has a \emph{graph-like} singular locus. 
\end{definition}
\begin{remark}
\label{toricnti}
Note that if $X \subset Y$ is a toric divisor inside a smooth toric 3-fold, then $X$ has graph-like singular locus. 
\end{remark}
We fix once and for all parametrizations of  the  irreducible components of $Z$  compatible, in an appropriate sense, with the stratification of $Z$ given by $S$:  
\begin{itemize}
\item Let $C$ be a component of $Z$ isomorphic to $\mathbb{A}^1$, then we fix a parametrization  $\phi: \mathbb{A}^1 \stackrel{\cong}\to C$ such that $\phi(\mathbb{A}^1 - \{0\}) \cap S = \varnothing$
\item  Let $C$ be a component of $Z$ isomorphic to $\mathbb{P}^1$ then we fix  a parametrization  $\psi:\mathbb{P}^1 \stackrel{\cong}\to C$ such that 
$\psi(\mathbb{P}^1-\{0, \infty\}) \cap S = \varnothing$
\end{itemize}
Now let $S'$ be the subset of $Z$ obtained as the union of the images  $\phi(\{0\})$, and $\psi(\{0, \infty\})$, as $\phi$ and $\psi$ run over the fixed parametrizations of the components of $Z$. %which are isomorphic to $\mathbb{A}^1$ or $\mathbb{P}^1$. 
Note that $S \subset S'$. 

\begin{definition}
We will refer to the points in $S'$ as  \emph{marked points}. 
\end{definition}

\begin{example}
\label{beilt}
\begin{enumerate}
\item Let $T = \mathbb{P}^1 \times \mathbb{A}^2$, and let $X = X_1 \cup X_2$ where 
$X_1 = \mathbb{P}^1 \times \mathbb{A}^1 \times \{ 0 \}$ and 
$X_2 = \mathbb{P}^1 \times \{ 0 \} \times  \mathbb{A}^1$. Then $Z\cong\mathbb{P}^1$,  $S = \varnothing$, and $S' = \{(0,0,0) \, , \, (\infty,0,0)\} \in Z$. 
\item Let $T = T^*\mathbb{P}^1 \times \mathbb{A}^1$, and let $X = X_1 \cup X_2 \cup X_3$ where $X_1 = T^*_0\mathbb{P}^1 \times \mathbb{A}^1$, 
$X_2 =  \mathbb{P}^1 \times \mathbb{A}^1$ and 
$X_1 = T^*_\infty\mathbb{P}^1 \times \mathbb{A}^1$. Then $Z$ has five irreducible components, and $S = S'$  has cardinality two and is given by $(0,0),  (\infty,0) \in \mathbb{P}^1 \times \mathbb{A}^1 \subset T^*\mathbb{P}^1 \times \mathbb{A}^1$.
\end{enumerate}
\end{example}

Let $G(X)$ be the graph having as vertices the points of $S'$, and as edges the irreducible components of $Z$: if $C$ is an irreducible component of $Z$ and 
$s \in S'$, then the edge $t_C$ is is incident to the vertex $v_s$ if and only if 
$s$ lies on C. Note that $G(X)$ has no loops and the vertices of $G(X)$ have valency equal to either three or one. 

Note that the graph $G(X)$ does not depend on the choices of parametrization of the irreducible components of $Z$: any such choice gives rise to equivalent graphs.
%Over $\kappa=\mathbb{C}$, the graph $G(C)$ can be described more geometrically as follows. If $Z$ has $n$ irreducible components, it can be equipped with an action of $\mathbb{G}_m^n$: where the $i$-th  factor  of $(S^1)^n$ acts by rotation  on the $i$-th component of $Z$. Then the graph $G(C)$ is isomorphic to the image of $Z$ in $\mathbb{R}^n$ under the moment map. 

\begin{example}
Let us follow up on Example \ref{beilt}. In case $(1)$, $G(X)$ has two 1-valent vertices and one edge joining them. In case $(2)$, $G(X)$ has   two 3-valent vertices,  one compact edge joining them, and four non-compact ones.% each of the vertices is also incident to  two  non-compact edges. 
\end{example}
Observe that the dual intersection complex of the simple normal crossing divisor $X$ is naturally a triangulated real two-dimensional topological manifold.

The following theorem is the main result of this section.
\begin{theorem}
  \label{matrixgraph}
  Suppose that the dual intersection complex of $X$ is orientable.
Then there is an $\kappa$-linear equivalence of categories
$$
\Dsing^\infty(X) \simeq \mathcal{B}(G(X))
$$
\end{theorem}
We   prove Theorem \ref{matrixgraph} in Section \ref{sec:completions} below.

\begin{remark}
Before proceeding, we should clarify that in the statement of Theorem \ref{matrixgraph}, the $2$-periodicity of $\Dsing^{\infty}(X)$ is regarded as a \emph{property} rather than a \emph{structure}. If we choose a way of presenting $X$ as the zero fiber of a map $f : T \to \bA^{1}$, then, via the equivalence $\mathrm{MF}^{\infty}(T,f) \simeq \Dsing^{\infty}(X)$, the singularity category is endowed with a $2$-periodic \emph{structure}. The version of Theorem \ref{matrixgraph} that takes this structure into account is a $\kappa\rbu$-linear equivalence
  \begin{equation*}
    \mathrm{MF}^{\infty}(T,f) \simeq \cB(G(X),\gamma)
  \end{equation*}
  where $\gamma\in \kappa^{\times}$ accounts for a possible rescaling of the $2$-periodic structure (See Section \ref{sec:overall-dependence}). In the proof of Theorem \ref{matrixgraph} below, the weight $\gamma$ arises because we are rescaling the coordinates in the local models.
\end{remark}

\begin{remark}
In the case where the dual intersection complex of $X$ is not orientable, the first Stiefel-Whitney class of this $2$-manifold restricts to a class $w_{1} \in H^{1}(G,\Z/2\Z)$. We expect that the analog of Theorem \ref{matrixgraph} holds when the category $\cB(G(X))$ is twisted by $w_{1}$ as in \cite{segal2021line} (See Remark \ref{rem:segal-twist} above).
\end{remark}

\subsubsection{Matrix Factorizations and completions} 
\label{sec:completions} 
Let $H \subset \mathbb{A}^n$ be the union of the coordinate hyperplanes. 
We fix an open subset $Y \subset \mathbb{A}^n$. %be an open subset of $\mathbb{A}^n$. 
 The intersection $H \cap Y$ is a normal crossing divisor in $Y$, which with small abuse of notation we keep denoting  $H$.  

\begin{lemma}
\label{lxbas}
Let $T$ be a scheme of dimension $n$ and let $X \subset T$ be a normal crossing divisor. 
Let $Z_X \subset X$ be the singular locus of $X$. Assume that $Z_X$ is isomorphic to  the singular locus 
$Z_{H_X}$ of a normal crossing divisor $H_X \subset Y$ given by a subset of the components of $H$. Then the completion of $Z_X$ in $X$ is  isomorphic to the completion of $Z_{H_X}$ in $H_X$. 
\end{lemma}
\begin{proof}
The key input will be 
Theorem 1.5 of \cite{camacho2003neighborhoods},   which gives a criterion comparing the 
infinitesimal neighborhoods of a scheme in different ambient spaces.\footnote{The reference \cite{camacho2003neighborhoods} is about complex analytic varieties and therefore also Theorem 1.5 is formulated in  that generality: but the proof relies entirely on general sheaf cohomology techniques, which apply without variations to schemes over any ground field $\kappa$.}

Consider the  two embeddings 
$$
j_1: Z_X \subset X \subset T,  \quad j_2:Z_X \cong Z_{H_X} \subset H_X \subset Y
$$
Let $N_{Z_X/T}$ and $N_{Z_X/Y}$ be the normal sheaves relative to the embedding of  
$Z_X$ in $T$ and in $Y$. Let $Z_1, \ldots, Z_l$ be the irreducible components of $Z_X$, and denote $\iota_i:Z_i \to Z_X$ the embedding of the $i$-th irreducible component in $Z_X$. 
There is a commutative diagram of exact sequences 
$$
\xymatrix{
\bigoplus_i \iota_{i,*}  T_{Z_i} \ar[r] \ar[d]_\cong & j_1^*T_T \cong \cO_{Z_X}^n 
\ar[r] \ar[d]^\cong & N_{Z_X/T} \ar[r] & 0 \\
\bigoplus_i  \iota_{i,*} T_{Z_i} \ar[r]   & j_2^*T_Y \cong \cO_{Z_X}^n 
\ar[r]  & N_{Z_X/Y} \ar[r] & 0
}
$$
where $T_T$ and $T_Y$ are the tangent sheaves over $T$ and over $Y$ respectively. This yields an isomorphism of normal sheaves $N_{Z_X/T} \simeq N_{Z_X/Y}$. Since $Z_X$ is affine, higher cohomology of quasi-coherent sheaves vanish: thus \cite[Theorem 1.5]{camacho2003neighborhoods} implies that the completions of $T$ and $Y$ along $Z_X$ are isomorphic  
\begin{equation}
\label{cong}
 \widehat{Z_X}^{T} \cong \widehat{Z_X}^{Y}. 
\end{equation}

Now consider the restriction to $\widehat{Z_X}^{T}$ of  local functions cutting out the irreducible components of $X$. Under isomorphism (\ref{cong}), they correspond to  functions on $\widehat{Z_X}^{Y}$. Further, up to a change of coordinates on $\widehat{Z_X}^{Y}$, these coincide with the  functions on $Y$ cutting out the components of $H_X$. Thus we deduce that (\ref{cong}) restricts to an equivalence 
$$
\widehat{Z_X}^{X} \simeq \widehat{Z_X}^{T} \times_T X \cong \widehat{Z_X}^{Y} \times_{Y} H_X \simeq 
\widehat{Z_X}^{H_X}
$$
and this concludes the proof.  
 \end{proof}

 We are now ready to prove Theorem \ref{matrixgraph}. 
 
\begin{proof}[The proof of Theorem \ref{matrixgraph}] 
We will make use of two key properties of $\Dsing^\infty(-)$. For reference, we state them below 
\begin{itemize}
\item Let $U$ be a quasi-projective scheme.  By  Theorem 2.10 \cite{orlov2011formal}, $\Dsing^\infty(U)$ only depends on the formal neighborhood of the singular locus of  $U$.\footnote{In fact, Orlov proves that this is true only up to idempotent completion: however   throughout  the paper we work with categories up to Morita equivalence, and so in particular up to idempotent completion.}
\item  By Proposition A.3.1 of \cite{preygel2011thom} 
$\Dsing^\infty(-)$ satisfies \'etale (and thus Zariski) descent with respect to varieties that are presented as the zero fiber of a morphism.
\end{itemize}
%We will call them \emph{Property A} and \emph{Property B}. 

We will use the notations of Section \ref{sconcs}: $X$ is a 2-dimensional simple normal crossing divisor, $Z$ is its singular locus, $S$ is the singular locus of $Z$, and $S'$ is the set of marked points on $Z$.  We fix a parametrization of the components of $Z$. Note that if $S'$ is empty then the statement is trivial and, if $S$ is empty, the statement just follows from Zariski descent for $\Qcoh^{(2)}(-)$. We will assume that $S = S' \neq \varnothing$, the  case when $S \subsetneq S'$ follows in a similar way.

Let $s \in S$. There are exactly three components of $Z$ that are incident to $s$, we denote them $Z_{s,1}, Z_{s,2}, Z_{s,3}$ and set  $Z_s=\cup_{i=1}^3 Z_{s,i}$. We denote by $Z_s^c$ the union of the components of $Z$ that are not incident to $s$, $Z_s^c = \overline{Z-Z_s}$.  The fixed parametrization of the components of $Z$ yields an isomorphism 
between $Z-Z_s^c$ and the  coordinate lines in $\mathbb{A}^3$ which sends $s$ to $(0,0,0) \in \mathbb{A}^3$. 
If $s$ and $s'$ are adjacent vertices in $G(X)$, then there is exactly one irreducible component $Z_{s, s'}$ of $Z$ such that 
$s, s' \in Z_{s,s'}$. 

For every $s \in S$, let $U_s = X-Z_s^c$. Then $\mathcal{U} = \{U_i\}_{s \in S}$ is an open cover of $X$. Note that  $\mathcal{U}$ does not have any non-trivial triple overlaps. Since $\Dsing^\infty(-)$ satisfies Zariski descent we can realize $\Dsing^\infty(X)$ as a limit in   $\mathrm{DGCat^{(2)}_{cont}}$
$$
\xymatrix{
\Dsing^\infty(X) \to[ \underset{s\in S}{\bigoplus} \Dsing^\infty(U_s)
\ar@<-.5ex>[r]_-*!/d0.5pt/{\labelstyle }
\ar@<.5ex>[r]^ -*!/u0.7pt/{\labelstyle } 
& \underset{s, s' \in S}{\bigoplus} \Dsing^\infty(U_s \cap U_{s'})  ]
}
$$
We make the following observations: 
\begin{enumerate}
\item[(a)] $U_s$ is a normal crossing surface, with singular locus $Z_s$. By Lemma 
\ref{lxbas}  the formal neighborhood of $Z_s$ inside $U_s$ is isomorphic to the formal neighborhood of the  coordinate axes inside the coordinate hyperplanes $H \subset \mathbb{A}^3$. Orlov's results from  \cite{orlov2011formal}  imply $$
\Dsing^\infty(U_s) \simeq \Dsing^\infty(H) \simeq \MF^\infty(\mathbb{A}^3, x_1x_2x_3)
$$
\item[(b)] If $s \neq s'$ and $s$ and $s'$ are adjacent in $G(X)$, then 
$U_s \cap U_{s'} = X-(Z_s^c \cup Z_{s'}^c)$. The singular locus of 
 $U_s \cap U_{s'}$ is isomorphic to $\mathbb{G}_m$. Using again Lemma \ref{lxbas}, and Orlov's Theorem we deduce that
$$
 \Dsing^\infty(U_s \cap U_{s'}) \simeq \Qcoh^{(2)}(\mathbb{G}_m)
$$
 \item[(c)]  If $s \neq s'$, and 
 $s$ and $s'$ are not adjacent vertices in $G(X)$, then 
$U_s \cap U_{s'} = X-Z$. Since $X-Z$ is smooth 
$$\Dsing^\infty(U_s \cap U_{s'}) \simeq 0$$
\end{enumerate}
Using these observations we can rewrite, in terms of the quiver $Q_{G(X)}$, the equalizer that computes $\Dsing^\infty(X)$ as 
\begin{equation}
\label{xditub}
\xymatrix{
\Dsing^\infty(X) \to[ \underset{\bullet_v \in Q_{G(X)}} {\bigoplus} \mathcal{B}(\bullet_v) 
\ar@<-.5ex>[r]_-*!/d0.5pt/{\labelstyle }
\ar@<.5ex>[r]^ -*!/u0.7pt/{\labelstyle } 
& \underset{\circ_t \in Q_{G(X)}}{\bigoplus} \mathcal{B}(\circ_t) ]
}
\end{equation}
Now,  diagram (\ref{xditub}) has the same shape as   diagram (\ref{bbb}) from Remark \ref{atavo} which  computes $\mathcal{B}(G(X))$ (see Remark \ref{atavo}). As the last step of the proof, we need to show  that we can choose the equivalences from our observations (a) and (b) above,  in such a way that the functors appearing in (\ref{xditub}) match the ones appearing in   (\ref{bbb}).

The first issue to address is whether these functors are compatible with the conventions for resolving the shift ambiguity of the Kn\"{o}rrer periodicity that we set down in Section \ref{sec:local-model-res}. This is where we use the hypothesis that the dual intersection complex $C(X)$ of $X$ is orientable. Choose an orientation of $C(X)$. Since the graph $G$ embeds into $C(X)$, the orientation of $C(X)$ induces a cyclic ordering of the edges at each vertex of $G$. Taking these cyclic orderings together with arbitrarily chosen orientations of each edge, we obtain framing data for $G$ as in Section \ref{sec:local-model-res}.

Consider now an irreducible component $D$ of $X$; the structure sheaf $\cO_{D}$ gives rise to an object of $\Dsing^{\infty}(X)$. This object therefore gives rise to a compatible system of objects in the diagram \eqref{xditub}. For an adjacent pair of vertices $v,v'$ that lie on the boundary of $D$, and using the notation of Section \ref{sec:local-model-res}, we see that $\cO_{D}$ maps to objects of form $F_{ij}$ with $i,j$ cyclically ordered (with respect to the chosen framing data) at both $v$ and $v'$. These must therefore map to the same object over the edge $t$ joining $v$ to $v'$, and this is consistent with the convention set out in Section \ref{sec:local-model-res}.\footnote{This argument shows moreover that the convention of Section \ref{sec:local-model-res} is essentially determined by the requirement that $\cO_{D}$ should gives rise to an object of $\cB(G(X))$ in the present setting.}

It remains to show that we can rescale the local coordinates on each patch so that the functors in the diagram \eqref{xditub} match the diagram defining $\cB(G(X))$. Let $s \in S$, and let $Z_{s,1}, Z_{s,2}$ and $Z_{s,3}$ be the three components 
of $Z$ incident to $s$. Let us assume, for ease of exposition, that $Z_{s,1}, Z_{s,2}$ and $Z_{s,3}$ are all isomorphic to $\mathbb{P}^1$: this implies that, in addition to $s$, each component carries a second marked point. Let us name  these points $s_1$, $s_2$ and $s_3$.\footnote{The general case where some or all of the $Z_{s_i}$ are isomorphic to $\mathbb{A}^1$ is proved similarly.}  
For all $i=1, 2, 3$ we denote 
$$
j^*: \Dsing^\infty(U_s) \to \Dsing^\infty(U_s \cap U_{s_i}) %\quad  l_i^{*}: \Dsing^\infty(U_{s_i}) \to \Dsing^\infty(U_s \cap U_{s_i})
$$
the pull-back functor. 

Let us recall  the dictionary translating between  $X$ and the quiver $Q_{G(X)}$. The marked point $s$ corresponds to a black vertex $\bullet_{v}$ and the  components   $Z_{s,1}, Z_{s,2}$ and $Z_{s,3}$ correspond to  white vertices  
$\circ_{t_1}$, $\circ_{t_2}$ and $\circ_{t_3}$. Additionally we have three arrows 
$$
\xymatrix{
& \bullet_{v_s} \ar[dr]^-{x_3} \ar[d]^-{x_2} \ar[dl]_-{x_1} &\\
\circ_{t_1} & \circ_{t_2} & \circ_{t_3} 
}
$$
It is easy to see that we can make compatible choices of equivalences \emph{locally} around each vertex $s$ of $G(X)$. Namely for every $s \in S$ we can choose equivalences 
$$
\rho_s: \Dsing^\infty(U_s) \simeq \mathcal{B}(\bullet_{v_s}) \quad  \sigma_s: \Dsing^\infty(U_s \cap U_{s_i}) \simeq \mathcal{B}(\circ_{t_i})
$$
in such a way that for all $i$ we obtain a commutative diagram 
$$
\xymatrix{
\Dsing^\infty(U_s) \ar[r]^-{\rho_s} \ar[d]_-{j^*} & \mathcal{B}(\bullet_{v_s}) \ar[d]^-{\mathcal{B}(x_i)} \\
\Dsing^\infty(U_s \cap U_{s_i}) \ar[r]^-{\sigma_s} & \mathcal{B}(\circ_{t_i})
}
$$
Indeed, by definition, the functors $j^*$ and $\mathcal{B}(x_i)$ are both given by restriction to the $i$-th component of the singular locus.  

The key point %only point that needs to be clarified 
  is that we can actually make a \emph{global} choice of compatible equivalences.  %which are compatible with the restriction functors. 
   Let us take two adjacent vertices in $G(X)$, $s$ and $s'$. They both lie on a component $C$ of $Z$, $C \cap S = \{s, s'\}$. The points $s, s'$ and the curve $C$ give rise to a subquiver of $Q_{G(X)}$ of the form 
$$
\bullet_{v_s} \stackrel{x} \to \circ_{t_C} \stackrel{y} \leftarrow \bullet_{v_{s'}}
$$
This corresponds to a commutative diagram of categories 
$$
\xymatrix{
\Dsing^\infty(U_s) \ar[r]^-{\rho_s} \ar[d]_-{j^*} & \mathcal{B}(\bullet_{v_s}) \ar[dr]^-{\mathcal{B}(x)} & & \mathcal{B}(\bullet_{v_{s'}}) \ar[dl]_-{\mathcal{B}(y)} & \Dsing^\infty(U_{s'}) \ar[l]_-{\rho_{s'}} \ar[d]^-{j^*} \\
\Dsing^\infty(U_s \cap U_{s'}) \ar[rr]^-{\sigma_s} && \mathcal{B}(\circ_{t_C}) && 
\Dsing^\infty(U_s \cap U_{s'}) \ar[ll]_-{\sigma_{s'}}
}
$$
For clarity, let us assume that we have numbered the components of the singular locus of $U_{s'}$ and that 
$$
j^*: \Dsing^\infty(U_{s'}) \longrightarrow 
\Dsing^\infty(U_s \cap U_{s'})
$$
corresponds to restriction to the \emph{first component}. The compatibility of the equivalences $ \rho_s, \sigma_s $ with  $ \rho_{s'}, \sigma_{s'} $ amounts to the condition that %  $\sigma_s$ and $\sigma_{s'}$ coincide, that is
$$ 
\sigma_s \circ \sigma_{s'}^{-1} = \mathrm{Id}_{\Dsing^\infty(U_s \cap U_{s'})}
$$ 

Note that the composition $\sigma_s \circ \sigma_{s'}^{-1}$ is an auto-equivalence of 
$\Dsing^\infty(U_s \cap U_{s'}) \simeq \Qcoh^{(2)}(\mathbb{G}_m)$. We can be more precise than that: 
%The group of auto-equivalences of $\Qcoh^{(2)}(\mathbb{G}_m)$ is isomorphic to  
%$\mathbb{G}_m \times \mathbb{Z}_2$, where 
  %$\mathbb{Z}_2$ acts by shift, 
%and  $\mathbb{G}_m$ acts on itself by  multiplication. 
%Also note that 
since $\sigma_s \circ \sigma_{s'}^{-1}$ is a composition of pull-backs along isomorphisms of schemes we  have that 
 $$\sigma_s \circ \sigma_{s'}^{-1} = \lambda^*$$ 
 where we denote by $\lambda: \mathbb{G}_m \to \mathbb{G}_m$ the multiplication by $\lambda \in \bG_m$. We will see below how we can   set $\lambda$  equal to $1$ by modifying 
 $\rho_{s'}$ and $\sigma_{s'}$ by an appropriate auto-equivalence  of $\Dsing^\infty(U_{s'})$. 
 %and normalize this equiva in such a way as to reabsorb this constant, and obtain $\lambda=1$. 
 
The natural action of $(\mathbb{G}_m)^3$ on $\mathbb{A}^3$ restricts to $H$, and via pull-back induces an action on 
 $\Dsing^\infty({H})$. We denote
 $$
 (\lambda_1, \lambda_2, \lambda_3)^*: \Dsing^\infty({H}) \to \Dsing^\infty({H})
 $$
 the pull-back along the multiplication by $(\lambda_1, \lambda_2, \lambda_3) \in (\mathbb{G}_m)^3$. For all $i=1,2,3$ there is a commutative diagram 
 $$
 \xymatrix{
 \Dsing^\infty({H}) \ar[r]^-{\iota_i^*} \ar[d]_-{(\lambda_1, \lambda_2, \lambda_3)^*} & \Qcoh^{(2)}(\mathbb{G}_m) \ar[d]^-{\lambda_i^*} \\
  \Dsing^\infty({H}) \ar[r]^-{\iota_i^*}  & \Qcoh^{(2)}(\mathbb{G}_m)
 }
 $$
 where $\iota_i^*$ is the restriction to the $i$-th component of the singular locus. 
 
The equivalence $\Dsing^\infty(U_{s'}) \simeq \Dsing^\infty({H})$ induces a $(\mathbb{G}_m)^3$-action on $\Dsing^\infty(U_{s'})$. Recall  that  we stipulated that 
$ 
j^*: \Dsing^\infty(U_{s'}) \longrightarrow
\Dsing^\infty(U_s \cap U_{s'}) 
$ 
is the restriction to the first component of the singular locus of $U_{s'}$. Then it is enough to set  
$$
\rho_{s'}:= \rho_{s'} \circ (1, \lambda^{-1}, 1)^* \quad , \quad \sigma_{s'}:=\sigma_{s'} \circ (\lambda^{-1})^* 
$$
to obtain that $  
\sigma_s \circ \sigma_{s'}^{-1} = \mathrm{Id}_{\Dsing^\infty(U_s \cap U_{s'})}
$, as desired. 

Acting with $\mathbb{G}_m^3$ on $\Dsing^\infty(U_{s'})$ allows us to modify, by scalar multiplication, the restriction functors to the three components of the singular locus of $U_{s'}$. In this way, the equivalences $\rho_{s'}$ and $\sigma_{s'}$ at the vertex $s'$ can be made compatible with the equivalences $\rho_s$ and $\sigma_s$ coming from any neighboring vertex $s$. An induction on the cardinality of $S$ shows that there are \emph{global} choices of compatible equivalences, and this concludes the proof. 
 \end{proof}

\begin{remark}
Theorem \ref{matrixgraph} has the remarkable consequence that, as long as $X$ satisfies the stated requirements (it is the fiber of a morphism from a smooth variety to $\bA^{1}$, it has graph-like singular locus, and the dual intersection complex is orientable), the category
$\Dsing^\infty(X)$ does not actually depend on the infinitesimal neighborhood of the singular locus $Z$ of $X$, but only on the combinatorics of  $G(X)$. We must remark, however, that the requirement that $X$ arises as the fiber of a morphism $f: T\to \bA^{1}$ with $T$ smooth does entail a topological restriction on the infinitesimal neighborhood.

% Still, if $X$ and $X'$ are two different normal crossing surface satisfying the requirements above, and such that its singular locus is   isomorphic to $Z$.  %is isomorphic to the singular locus of both $X$ and $X'$. 
% By Lemma \ref{lxbas} the infinitesimal neighborhoods of  $Z$ inside  $X$ and $X'$  are \emph{locally} isomorphic, but in general this will not extend to a global isomorphism. Nonetheless Theorem \ref{matrixgraph} implies that   $\Dsing^\infty(X)$ is equivalent to 
% $\Dsing^\infty(X')$. 

We stress that the singularity categories  consider in Theorem \ref{matrixgraph} are exceptional in this respect. %, and is due to the fact that $Z$ has an especially simple shape. 
In general the singular locus $Z$   is not  by itself sufficient to determine $\Dsing^\infty(-)$: one does need information on the infinitesimal neighborhood of $Z$.  As an example consider a quadratic bundle $E$ over a scheme $X$. Note that $E$ carries a natural superpotential $W$ given by the quadratic form $q$. The singular locus of $W$ is always $X$, independently on the quadratic form; however by Theorem 9.3.4 of \cite{preygel2011thom} the category of matrix factorizations $\MF(E, W)$ is equivalent to modules over a sheaf  of Clifford algebras which does depend on $q$. 
 \end{remark}

%% file: coverings.tex
\section{Covering spaces and equivariance}
\label{csae}

In this section we set up the theory of the Fukaya category of a compact Riemann surface and develop the connection between covering spaces, anchored Lagrangians, and some notions of equivariance for categories.

\subsection{Variants of the Fukaya category for finite type surfaces}
\label{variantsFukaya}
We shall now recall several different variants of the Fukaya category for compact and noncompact surfaces. All of the versions are closely related, but having different ways of presenting the Fukaya category can make some arguments easier.

We fix a ground field $k$ of characteristic $0$. Let $\Sigma$ be surface of finite type, meaning that $\Sigma$ is either compact or it embeds into a compact surface, with symplectic form $\omega$. (We shall consider non-finite type surfaces below.) Recall the universal (over $\C$) Novikov field $\Lambda$,
\begin{equation}
  \Lambda = \left\{ \sum_{i = 0}^{\infty} a_{i} q^{\lambda_{i}} \mid a_{i} \in \C,\  \lambda_{i}\in \R,\  \lim_{i\to \infty} \lambda_{i}= \infty \right\},
\end{equation}
the field of formal series with real exponents such that there are only finitely many nonzero terms with exponent less than any given threshold. The most canonical version of the Fukaya category of $\Sigma$ is a $\Z/2$-graded triangulated split-closed $A_{\infty}$-category defined over $\Lambda$. We will denote it simply by $\mathrm{Fuk}(\Sigma)$. The generating objects are unobstructed immersed Lagrangians equipped with orientations and spin structures; an immersed Lagrangian is unobstructed when it is not null-homotopic and does not bound any teardrops.\footnote{This condition is equivalent to the condition that the pullback of the Lagrangian under the universal covering of $\Sigma$ is embedded. In the case where $\Sigma$ is a sphere, $\mathrm{Fuk}(\Sigma)$ as defined here is the zero category: it is well-known that in order to have a good Fukaya category for the sphere one must include \emph{weakly unobstructed} Lagrangians.} In the case where $\Sigma$ is noncompact, we also include noncompact properly embedded arcs that are treated in the manner of wrapped Floer cohomology. Most importantly, the $A_{\infty}$ structure maps count pseudoholomorphic polygons weighted by $q^{\text{Area}}$; it is the possibility of having infinitely many terms that necessitates using the Novikov field $\Lambda$.

In certain situations, it is possible to define versions of the Fukaya category over smaller fields such as $\C$.

First we introduce the balanced Fukaya category as follows. Let $\pi: S(T\Sigma) \to \Sigma$ be the circle bundle associated to the tangent bundle of $\Sigma$, defined as the set of oriented real lines in the tangent spaces. The form $\pi^{*}(\omega)$ defines a cohomology class in $H^{2}(S(T\Sigma),\R)$. This class vanishes provided that either
\begin{itemize}
\item $\Sigma$ is compact of genus $g \neq 1$, or
\item $\Sigma$ is not compact.
\end{itemize}
Suppose we are in such a case, so that we may choose a one-form $\theta$ on $S(T\Sigma)$ such that $d\theta = \pi^{*}\omega$. Given a oriented connected Lagrangian submanifold $L \subset \Sigma$, which is to say an oriented simple closed curve, we may form a tangent lift $\sigma : L \to S(TM)|_{L}$. We say $L$ is \emph{balanced} if $\int_{L} \sigma^{*}(\theta) = 0$. Contractible curves are never balanced, and every isotopy class of noncontractible simple closed curves contains a balanced representative, which is unique up to Hamiltonian isotopy. The generating objects we consider consist of a balanced curve equipped with an orientation and spin structure, and in the case were $\Sigma$ is non compact, we also include noncompact properly embedded arcs as before. The morphism complexes are defined over $\C$, and the $A_{\infty}$ structure maps count pseudoholomorphic polygons \emph{without} weighting by the area. We enlarge this category to a $\Z/2$-graded, triangulated, split-closed $A_{\infty}$-category over $\C$, which we call the \emph{balanced Fukaya category} $\Fbal(\Sigma,\theta)$.

The construction of $\Fbal(\Sigma,\theta)$ depends on $\theta$ modulo exact one-forms on $S(T\Sigma)$. Given two choices $\theta$ and $\theta'$, there is a class $[\theta' - \theta] \in H^{1}(S(T\Sigma),\R)$ measuring the difference. If $\Sigma$ is compact with $g \neq 1$, then $H^{1}(S(T\Sigma),\R) \cong H^{1}(\Sigma,\R)$, and after moving $\theta'$ by a symplectomorphism of $\Sigma$ with appropriate flux, we can ensure that this class vanishes. This symplectomorphism then induces an equivalence of categories $\Fbal(\Sigma,\theta) \cong \Fbal(\Sigma,\theta')$. On the other hand, if $\Sigma$ is not compact, then $H^{1}(S(T\Sigma),\R) \cong H^{1}(\Sigma, \R) \oplus \R$, and it is not necessarily possible to make the difference class vanish using a symplectomorphism of $\Sigma$ or its completion. 

The exact Fukaya category is another variant defined over $\C$ that is only defined when $\Sigma$ is noncompact. Here we choose a one-form $\lambda$ on $\Sigma$ such that $d\lambda = \omega$. We say a Lagrangian $L$ is \emph{exact} if $\lambda|_{L}$ is an exact one-form. This is the same thing as saying that $L$ is balanced with respect to the form $\theta = \pi^{*}\lambda$, and the construction of the \emph{exact Fukaya category} $\Fex(\Sigma,\lambda)$ follows exactly the same lines as the balanced Fukaya category in that case. Note that once again this is a ``wrapped'' variant of the Fukaya category. This construction depends on $\lambda$ modulo exact one-forms on $\Sigma$, and it is again possible to relate any two choices by a symplectomorphism of $\Sigma$ or its completion.

For noncompact $\Sigma$, all of the categories introduced so far are ``wrapped'' versions; if we wish to restrict to objects supported on compact Lagrangians we use the subscript ``c'', so we have $\Fc(\Sigma)$ (over $\Lambda$), $\Fbalc(\Sigma,\theta)$, and $\Fexc(\Sigma,\lambda)$ (over $\C$).

All variants of the Fukaya category are closely related to one another. The relationship between the balanced Fukaya category $\Fbal(\Sigma,\theta)$ and $\mathrm{Fuk}(\Sigma)$ works as follows. Let $\cC(\Sigma,\theta)$ be the full subcategory of $\mathrm{Fuk}(\Sigma)$ whose objects are branes supported on balanced curves, so that $\cC$ is a $\Lambda$-linear category where the operations weight curves by area. It is known that $\cC(\Sigma,\theta)$ split generates $\mathrm{Fuk}(\Sigma)$ over $\Lambda$ (for example, by combining Proposition 2.15 of \cite{auroux2020fukaya} with the observation that the curves used in that proposition can be taken to be balanced, since every isotopy class has a balanced representative). The condition that the curves are balanced means that all series appearing in the operations are finite, so that it is possible to set $q = 1$ in all operations, obtaining a $\C$-linear category $\cC_{q=1}(\Sigma,\theta)$, and this category embeds into $\Fbal(\Sigma,\theta)$ and split generates the latter category over $\C$. Because of these relationships, $\Fbal(\Sigma,\theta)$ can be thought of as giving a model of $\mathrm{Fuk}(\Sigma)$ over the smaller field $\C$.

Since the exact Fukaya category $\Fex(\Sigma,\lambda)$ is a special case of the balanced Fukaya category (with $\theta = \pi^{*}\lambda)$, everything in the previous paragraph is still valid. However in this case there is also a $\C$-linear embedding of $\Fex(\Sigma,\lambda)$ into $\mathrm{Fuk}(\Sigma)$. This functor is identity on objects and rescales every morphism by $t$ to its Floer action. Since this rescaling becomes trivial upon setting $q = 1$, this embedding followed by specialization recovers the identity on $\Fex(\Sigma,\lambda)$. Since $\Fex(\Sigma,\lambda)$ generates $\mathrm{Fuk}(\Sigma)$ over $\Lambda$, we conclude there $\mathrm{Fuk}(\Sigma)$ is equivalent to $\Fex(\Sigma,\lambda) \otimes_{\C} \Lambda$.

\subsection{Balanced versus exact categories for noncompact surfaces}

We now turn to the relationship between the various balanced and exact categories for noncompact surfaces. We have already seen that for compact surfaces, only the balanced category is defined, and it is independent of the choice of balancing form $\theta$ up to equivalence, since any two choices $\theta$ and $\theta'$ may be related by a symplectic isotopy of $\Sigma$. For noncompact surfaces the situation is a bit more subtle.

Recall that the difference between $\theta$ and $\theta'$ defines a class $[\theta' - \theta] \in H^{1}(S(T\Sigma),\R)$. Since $\Sigma$ is assumed noncompact, we may choose a trivialization $\eta : S(T\Sigma) \to \Sigma \times S^{1}$. Then $\eta$ induces a splitting $H^{1}(S(T\Sigma),\R) \cong H^{1}(\Sigma,\R)\oplus H^{1}(S^{1},\R)$. By applying a symplectic isotopy of $\Sigma$ to $\theta'$, say, it is possible to eliminate the component in $H^{1}(\Sigma,\R)$, so that the difference class lies in the subspace $H^{1}(S^{1},\R)$, but we cannot necessarily eliminate this latter component by the same method.

Let $\theta$ be given, and chose some $\lambda \in \Omega^{1}(\Sigma)$ that $d\lambda = \omega$, and consider the class $[\theta - \pi^{*}\lambda] \in H^{1}(\Sigma,\R)\oplus H^{1}(S^{1},\R)$. By adding a closed $1$-form to $\lambda$, we may assume that this class lies in $H^{1}(S^{1},\R)$. Suppose this class is $[\alpha \phi]$, where $\phi$ is the standard angular form on $S^{1}$ such that $\int_{S^{1}} \phi= 1$. Using the trivialization $\eta$, the form $\alpha \phi$ can be regarded as a form on $S(T\Sigma)$. Then we have that $[\theta - (\pi^{*}\lambda + \alpha \phi)] = 0$. As a result of this, we conclude that $\Fbal(\Sigma,\theta)$ is equivalent to $\Fbal(\Sigma,\pi^{*}\lambda + \alpha \phi)$ for some Liouville form $\lambda$ and some $\alpha \in \R$. Of course $\Fex(\Sigma,\lambda) = \Fbal(\Sigma,\pi^{*}\lambda)$, so we now pose the problem of relating the categories $\Fbal(\Sigma,\pi^{*}\lambda)$ and $\Fbal(\Sigma,\pi^{*}\lambda + \alpha \phi)$.

A compact Lagrangian $L$ gives rise to an object of $\Fbal(\Sigma,\pi^{*}\lambda)$ when $\int_{L} \lambda = 0$, while it gives rise to an object of $\Fbal(\Sigma,\pi^{*}\lambda + \alpha\phi)$ when $\int_{L}\lambda + \alpha \int_{L}\phi = 0$. Observe that $\int_{L}\phi$ is nothing but the rotation number of the curve $L$ with respect to the trivialization $\eta$. The trivialization $\eta$ (which is essentially a vector field on $\Sigma$) determines a line field $\Sigma$, and this rotation number is one-half of the Maslov class of $L$ computed with respect to that line field: we write $\int_{L}\phi = \mu_{L}/2$. In summary the condition for a compact Lagrangian to be balanced with respect to $\pi^{*}\lambda + \alpha \phi$ is the relation
\begin{equation}
  \int_{L}\lambda + (\alpha/2)\mu_{L} = 0.
\end{equation}
In particular, an $\eta$-gradable ($\mu_{L} = 0$) Lagrangian is balanced if and only if it is exact. Conversely, any non-gradable exact Lagrangian must be deformed by an amount proportional to $\mu_{L}$ in order to become balanced.

 Let $L_{0}$ be an exact Lagrangian, and let $L_{1}$ be the balanced Lagrangian obtained from $L_{0}$ by an isotopy with flux equal to $(-\alpha/2)\mu_{L}$. Then as modules over the HKK generators, $L_{1}$ (with the trivial local system) is equivalent to $L_{0}$ equipped with a local system whose holonomy is $q^{(-\alpha/2)\mu_{L}}$. In other words, every balanced object can be interpreted as an exact object with a local system that depends only on $\mu_{L}$. Upon setting $q = 1$, these local systems become trivial, and so $\Fbal(\Sigma,\pi^{*}\lambda + \alpha \phi) \cong \Fex(\Sigma,\lambda)$.

\subsection{Restriction functors}

A key element in our analysis of HMS for noncompact Riemann surfaces is the Viterbo restriction functor on exact categories. For a Liouville embedding $(\Sigma,\lambda) \to (\Sigma',\lambda')$, there is functor $r : \Fex(\Sigma',\lambda') \to \Fex(\Sigma,\lambda)$, which roughly speaking takes an exact Lagrangian in $\Sigma'$ to its intersection with the subsurface $\Sigma$. By extending scalars to $\Lambda$, we also obtain a restriction functor $r : \mathrm{Fuk}(\Sigma') \to \mathrm{Fuk}(\Sigma)$.

We remark that there is a bit of subtlety in computing the functor $r$ on a nonexact object of $\mathrm{Fuk}(\Sigma')$: it is not necessarily true that a Lagrangian $L$ in $\Sigma'$ disjoint from $\Sigma$ maps to the zero object, since after replacing $L$ with an isomorphic complex of exact Lagrangians, those Lagrangians may very well intersect $\Sigma$.

We may assume that $L$ is in minimal position with respect to the boundary components of $\Sigma$ ($L$ and the boundary components do not bound a disk), since this may always be achieved by a Hamiltonian isotopy. To be more precise about how restriction works, we distinguish two cases: either the contracting Liouville flow eventually pushes $L \subset \Sigma'$ so as to lie entirely within $\Sigma$, or not. Objects for which the contracting Liouville flow push $L$ entirely into $\Sigma$ include objects contained in $\Sigma$, and also curves in $\Sigma'$ that are parallel to a boundary component of $\Sigma$. Objects for which the contracting Liouville flow never pushes $L$ entirely into $\Sigma$ are either taken to zero or to some collection of arcs lying in $\Sigma$.

\begin{lemma}
  \label{geometricrestriction}
  Assume that $L \subset \Sigma'$ is in minimal position with respect to the boundary components of $\Sigma$.
  \begin{enumerate}
  \item If the contracting Liouville flow eventually pushes $L$ entirely into $\Sigma$, then restriction to $\Sigma$ maps $L$ to the same object thought of as an object in the Liouville completion of $\Sigma$.
  \item otherwise, restriction takes $L$ to an object isomorphic to one supported on $L \cap \Sigma$.
  \end{enumerate}
\end{lemma}

\begin{proof}
   We prove the second part. We express $L$ in terms of exact Lagrangians by applying contracting Liouville flow $\phi_{t}$ to collapse $L$ onto the core of $\Sigma'$, and then intersect with $\Sigma$. In the contraction process, new components of the intersection $\phi_{t}(L) \cap \Sigma$ may be created, due to portions of $L$ being pushed in to $\Sigma$. (If $L$ were not assumed to be in minimal position with respect to $\partial \Sigma$, it would also be possible for components of $\phi_{t}(L) \cap \Sigma$ to merge together.) However, any components of $\phi_{t}(L) \cap \Sigma$ that get created will bound a disk with $\partial \Sigma$, and hence will represent the zero object in $\mathrm{Fuk}(\Sigma)$.
\end{proof}

\subsection{Locally finite Fukaya category for non-finite-type Riemann surfaces}
\label{sec:loc-fin} 
Let $\widetilde{\Sigma}$ be a connected Riemann surface not of finite type. Choose an exhaustion of $\widetilde{\Sigma}$ by open subsurfaces of finite type,
\begin{equation}
   \widetilde{\Sigma} = \bigcup_{N=1}^{\infty} \widetilde{\Sigma}_{N}, \quad \widetilde{\Sigma}_{1} \subset \widetilde{\Sigma}_{2} \subset \dots \subset \widetilde{\Sigma}.
 \end{equation}
 Since each subsurface $\widetilde{\Sigma}_{N}$ is a punctured Riemann surface, and in particular an exact symplectic manifold, we may associate to it a Fukaya category, namely the wrapped Fukaya category $\Fw(\widetilde{\Sigma}_{N})$. Recall that this category contains both exact compact Lagrangian branes as well as properly embedded arcs, and that the morphism complexes are computed by wrapping around the punctures of $\widetilde{\Sigma}_{N}$. Next, for a pair of indices $N \leq M$, we have an embedding $\widetilde{\Sigma}_{N} \to \widetilde{\Sigma}_{M}$ (which may be taken to be an exact embedding), so there is a restriction functor,
 \begin{equation}
   r_{M,N} : \Fw(\widetilde{\Sigma}_{M}) \to \Fw(\widetilde{\Sigma}_{N}),
 \end{equation}
 known as the Viterbo restriction (constructed by Abouzaid-Seidel). As $N$ tends to infinity, the categories $\Fw(\widetilde{\Sigma}_{N})$ with the functors $r_{M,N}$ form an inverse system of $A_{\infty}$-categories, and we may pass to the homotopy limit, which we denote as
 \begin{equation}
   \Flf(\widetilde{\Sigma}) = \varprojlim_{N\to \infty} \Fw(\widetilde{\Sigma}_{N}).
 \end{equation}
 We call $\Flf(\widetilde{\Sigma})$ the \emph{locally finite Fukaya category} of $\widetilde{\Sigma}$. It is potentially different from other Fukaya categories one might reasonably associate to $\widetilde{\Sigma}$: for instance, an object of $\Flf(\widetilde{\Sigma})$ may be supported on a Lagrangian submanifold with infinitely many connected components. 

We remark that it is not difficult to show that $\Flf(\widetilde{\Sigma})$ does not depend up to equivalence on the choice of exhaustion.
 
For our purposes   
$\widetilde{\Sigma}$ will always be an infinite-sheeted covering of a compact Riemann surface 
$\Sigma$: namely, we take $\widetilde{\Sigma}$ to be the \emph{maximal tropical cover} of $\Sigma$ as defined in Section \ref{rsamtc}. Also we take the exhaustion by open subsurface of $\widetilde{\Sigma}$ to be compatible with the pants decomposition $\widetilde{P}$ of $\widetilde{\Sigma}$ in the sense specified in Section \ref{cbfts}. In this setting we have defined also the topological locally finite Fukaya category of $\widetilde{\Sigma}$. Recall that \cite{haiden2017flat} gives a dictionary between the wrapped and the topological Fukaya categories of the finite type subsurfaces. Then, using Proposition \ref{flws},   
we obtain also an equivalence between the locally finite Fukaya category and its topological counterpart. We record this simple observation in the following Proposition.
\begin{proposition}
\label{lftopo}
There is an equivalence 
$$
\Flf(\widetilde{\Sigma})  \simeq \mathrm{Fuk}^{\mathrm{top}, \mathrm{lf}}(\widetilde{\Sigma})
$$
where the Liouville structure of $\widetilde{\Sigma}$ is used to trivialize the dependence of the Fukaya category on the Novikov parameter.
\end{proposition}
\begin{proof}
The statement follows from the chain of equivalences 
$$
\Flf(\widetilde{\Sigma}) = \varprojlim_{N\to \infty} \Fw(\widetilde{\Sigma}_{N}) \simeq  \varprojlim_{N\to \infty} \mathrm{Fuk}^{\mathrm{top}}(\widetilde{\Sigma}_{N}) \stackrel{(*)}\simeq \mathrm{Fuk}^{\mathrm{top}, \mathrm{lf}}(\widetilde{\Sigma})
$$
where equivalence $(*)$ is given by Proposition \ref{flws}. 
\end{proof}

\subsection{Pullback for \'{e}tale maps}

We wish to understand how Fukaya categories pullback along $\pi: \widetilde{\Sigma}\to \Sigma$. Since $\widetilde{\Sigma}$ is not of finite-type, we consider an exhaustion $\widetilde{\Sigma} = \bigcup_{N=1}^{\infty} \widetilde{\Sigma}_{N}$ by open subsurfaces $\widetilde{\Sigma}_{N}$ that are of finite-type. Then the covering map $\pi : \widetilde{\Sigma} \to \Sigma$ restricts to maps $\pi_{N}: \widetilde{\Sigma}_{N} \to \Sigma$; $\pi_{N}$ is a local diffeomorphism (an \'{e}tale map) but not a covering map. Given a Lagrangian brane $L$, we construct an object $\pi^{-1}_{N}(L)$ in $\Fw(\widetilde{\Sigma}_{N})$ by pulling back the Lagrangian and brane structure under $\pi_{N}$. By construction, the objects $\pi^{-1}_{N}(L)$ are mapped to one another under the functors $r_{M,N}$, and so this system of objects gives rise to an object of the limit $\Flf(\widetilde{\Sigma})$. In remains to show that this object-by-object construction can be extended to a pullback functor, and to understand compatibility with the restriction functors.

In order to show that this assignment is functorial, we need to recall how morphism complexes and $A_{\infty}$ operations are computed in a wrapped category like $\Fw(\widetilde{\Sigma}_{N})$. Given two objects $L_{0}$ and $L_{1}$, the wrapped Floer complex $CW(L_{0},L_{1})$ is generated by chords of a Hamiltonian flow generated by a function $H$. In order to qualify as a ``wrapping Hamiltonian'', the function $H$ must be chosen so that it has convex growth on the cylindrical ends of $\widetilde{\Sigma}_{N}$; we also assume for convenience that $H$ is $C^{2}$-small away from the cylindrical ends. With this setup, the chords that generate $CW(L_{0},L_{1})$ naturally fall into two classes: the \emph{boundary chords} that live in one of the cylindrical ends of $\widetilde{\Sigma}_{N}$, and the rest, which are called \emph{interior chords}. The distinction between boundary chords and interior chords is not intrinsic to $\Fw(\widetilde{\Sigma}_{N})$ as a category; it something having to do with the specific geometrically-defined cochain complexes we use to construct the wrapped Fukaya category using Floer theory. This distinction allows us to make arguments at chain level thanks to the following lemma about Floer-theoretic operations:
\begin{lemma}
  If all of the inputs to an $A_{\infty}$ operation are interior chords, then the output is also an interior chord. 
\end{lemma}

\begin{proof}
  Given a pseudo-holomorphic curve contributing to such an operation, where all inputs are interior chords, the maximum principle applies to the radial coordinates on the cylindrical ends. This demonstrates that such curves cannot enter the cylindrical ends.
\end{proof}

Now we consider the problem of extending the association $L \mapsto \pi^{-1}_{N}(L)$ to an $A_{\infty}$-functor. Let $L_{0}$ and $L_{1}$ be objects of $\mathrm{Fuk}(\Sigma)$. The question is to relate $CF(L_{0},L_{1})$ with $CW(\lambda_{N}(L_{0}), \lambda_{N}(L_{1}))$. Now $CF(L_{0}, L_{1})$ is generated by the chords of some Hamiltonian function $H_{0}$ defined on $\Sigma$. When we pull back $H_{0}$ to $\widetilde{\Sigma}_{N}$, we do not get a wrapping Hamiltonian, since $H_{0}$ remains bounded on the ends; to remedy this we add a wrapping Hamiltonian $H_{1}$ that is zero away from the ends. This has the neat property that the interior chords of $(H_{0}\circ \pi_{N}) + H_{1}$ on $\widetilde{\Sigma}_{N}$ are nothing but the lifts with respect to $\pi_{N}$ of chords of $H_{0}$ on $\Sigma$. 

\begin{proposition}
  The association $L \mapsto \pi_{N}^{-1}(L)$ may be extended to an $A_{\infty}$-functor $\lambda^{\text{naive}}_{N} : \mathrm{Fuk}(\Sigma) \to \Fw(\widetilde{\Sigma}_{N})$.\end{proposition}

\begin{proof}
  The functor on objects is $L \mapsto \pi_{N}^{-1}(L)$. On morphism complexes, $\lambda^{\text{naive}}_{N}$ takes a chord $x \in CF(L_{0},L_{1})$ to the sum of all its lifts under $\pi_{N}$. By the discussion above, this is a well-defined element of $CW(\lambda_{N}(L_{0}),\lambda_{N}(L_{1}))$. Observe that $\lambda^{\text{naive}}_{N}(x)$ is always a sum of interior chords.

  It remains to show that the mapping $x \mapsto \lambda_{N}(x)$ directly matches all $A_{\infty}$ operations, so that $\lambda_{N}$ is an $A_{\infty}$-functor with vanishing higher-order components. To show this claim, observe that any $A_{\infty}$-operation contributing with inputs of the form $\lambda^{\text{naive}}_{N}(x)$ counts pseudoholomorphic curves whose images stay in the interior region (away from the ends of $\widetilde{\Sigma}_{N}$). In that region, all of the perturbations are pulled back under $\pi_{N}$, so these curves are nothing but the lifts of the pseudoholomorphic curves in $\Sigma$ that contribute to the corresponding operation in $\mathrm{Fuk}(\Sigma)$. (Note that every such curve does indeed lift, since its domain is simply connected, being a disk.) This establishes the existence of the functor $\lambda^{\text{naive}}_{N} : \mathrm{Fuk}(\Sigma) \to \Fw(\widetilde{\Sigma}_{N})$.
\end{proof}

We now introduce some notation having to do with the covering $\pi: \widetilde{\Sigma} \to \Sigma$. Recall that this is the normal covering space of $\Sigma$ associated to the kernel of the homomorphism
\begin{equation}
  \pi_{1}(\Sigma) \to \pi_{1}(G) \mapsto H_{1}(G) = O_{\pi} \cong \Z^{g},
\end{equation}
where $G$ is the graph associated to a pants decomposition. Now $G$ is a trivalent graph with first Betti number $g$; if we choose a spanning tree of $G$ and collapse it to a point, we obtain a graph $G'$ that is a wedge of $g$ circles. The map $G \to G'$ is a homotopy equivalence, and we have a map $\Sigma \to G \to G'$, so we may as well regard $\pi$ as the covering associated to the kernel of
\begin{equation}
  \pi_{1}(\Sigma) \to H_{1}(G') \cong O_{\pi}.
\end{equation}
Furthermore, the edges of $G'$ give us a natural basis for the lattice $O_{\pi}$. We may also consider the maximal abelian covering of $G'$, call it $\widetilde{G}'$; the graph $\widetilde{G}'$ is the Cayley graph of $O_{\pi}$ with respect to the chosen basis, so it is a $g$-dimensional rectangular lattice. Because $\pi : \widetilde{\Sigma} \to \Sigma$ is isomorphic to the pullback of $\widetilde{G}'\to G'$ along $\Sigma \to G'$, we have a map
\begin{equation}
  \psi : \widetilde{\Sigma} \to \widetilde{G}'
\end{equation}
that will be useful.

In terms of the topology of $\Sigma$, the choice of a spanning tree in $G$, and hence the choice of $G'$, corresponds to the choice of a \emph{cut system} on $\Sigma$. A cut system is a collection of $g$ simple closed curves such that cutting along these circles reduces $\Sigma$ to a $2g$-punctured sphere. From a spanning tree $T$ in $G$, we obtain a cut system by taking, for each edge in $G\setminus T$, the corresponding circle in $\Sigma$ from the pants decomposition. Conversely, any cut system contained in the given pants decomposition will determine a spanning tree in $G$.

The picture of $\widetilde{\Sigma}$ that follows from these considerations is as follows: We think of $\widetilde{\Sigma}$ as consisting of a collection of $2g$-punctured spheres arranged at the sites of a $g$-dimensional rectangular lattice $\widetilde{G}'$, and glued together according to that lattice. We shall assume that our exhaustion $\widetilde{\Sigma}_{N}$ of $\widetilde{\Sigma}$ consists of those punctured spheres at the sites whose coordinates lie in the cube $[-N,N]^{g}$.

Now we return to the construction of the pullback functor. 
  Let $L_{1}$ and $L_{2}$ be two Lagrangian branes in $\mathrm{Fuk}$. In order to understand the compatibility of the functors $\lambda^{\text{naive}}_{N}$ with restriction, as well as to compute morphism complexes in the limit category $\Flf(\widetilde{\Sigma})$, we want to understand geometrically the restriction map, for $M > N$,
  \begin{equation}
    r_{M,N} : CW(\pi^{-1}_{M}(L_{1}),\pi^{-1}_{M}(L_{2})) \to CW(\pi^{-1}_{N}(L_{1}),\pi^{-1}_{N}(L_{2})).
  \end{equation}

  Each closed curve in $\Sigma$ has a corresponding class in $O_{\pi}$. We distinguish two types of Lagrangians $L$ in $\Sigma$: Either
  \begin{enumerate}
  \item The class of $L$ in $O_{\pi}$ is zero, in which case every connected component of $\pi^{-1}(L)$ is a compact curve, or,
  \item The class of $L$ in $O_{\pi}$ is nonzero, in which case every connected component of $\pi^{-1}(L)$ is a noncompact arc.
  \end{enumerate}

  Given a pair of Lagrangians $L_{1}$ and $L_{2}$, we distinguish the following cases:
  \begin{enumerate}
  \item The class of at least one of $L_{1}$ or $L_{2}$ in $O_{\pi}$ is zero,
  \item The classes of both $L_{1}$ and $L_{2}$ in $O_{\pi}$ are nonzero, and distinct,
  \item The classes of $L_{1}$ and $L_{2}$ in $O_{\pi}$ are nonzero and equal, but $L_{1}$ is not isotopic to $L_{2}$,
  \item $L_{1}$ is isotopic to $L_{2}$, and they represent a nonzero class in $O_{\pi}$.
  \end{enumerate}

  In each case, we wish to understand whether two components $C_{i}$ of $\pi^{-1}(L_{i})$ may intersect or support boundary chords in $\widetilde{\Sigma}_{N}$. Whenever there is such an intersection point or chord, the images $\psi(C_{1})$ and $\psi(C_{2})$ must intersect: intersection points correspond to the situation where the images $\psi(C_{i})$ touch the same vertex, while chords correspond to the situation where they touch the same edge.
  
  \begin{lemma}
    \label{lem:compactprojection}
    Suppose that $L_{1}$ and $L_{2}$ are such that there is a component $C_{1}$ of $\pi^{-1}(L_{1})$ and a component $C_{2}$ of $\pi^{-1}(L_{2})$ such that $\psi(C_{1})\cap \psi(C_{2}) \subset \widetilde{G}'$ is not compact. Then the classes of $L_{1}$ and $L_{2}$ in $O_{\pi}$ are nonzero and equal.
  \end{lemma}

  \begin{proof}
    If the class of either $L_{1}$ or $L_{2}$ in $O_{\pi}$ is zero, then one of the components $C_{1}$ or $C_{2}$ is compact, so $\psi(C_{1}) \cap \psi(C_{2})$ must be compact as well.

    If the classes of $L_{1}$ and $L_{2}$ in $O_{\pi}$ are nonzero, then $C_{1}$ and $C_{2}$ are noncompact arcs. If the classes are distinct, then $\psi(C_{1})$ and $\psi(C_{2})$ eventually diverge from one another in the lattice $\widetilde{G}'$, so $\psi(C_{1})\cap \psi(C_{2})$ must again be compact.
  \end{proof}

  \begin{lemma}
    \label{lem:nonisotopic}
    Consider components $C_{i}$ of $\pi^{-1}(L_{i})$ ($i=1,2$) and their intersections $C_{i,N}$ with $\widetilde{\Sigma}_{N}$. Suppose that $L_{1}$ is not isotopic to $L_{2}$. Then for every $N$ and for sufficiently large $M > N$, the restriction map
    \begin{equation}
      r_{M,N}: CW(C_{1,M},C_{2,M}) \to CW(C_{1,N},C_{2,N})
    \end{equation}
    vanishes on boundary chords.
  \end{lemma}

  \begin{proof}
    First note that Lemma \ref{lem:compactprojection} implies that the hypothesis can only be satisfied if $L_{1}$ and $L_{2}$ represent the same nonzero class in $O_{\pi}$, since otherwise for $M$ sufficiently large the set of boundary chords is empty.
    
    In order for $r_{M,N}$ to be nonzero on a boundary chord in $\widetilde{\Sigma}_{M}$, there must be a strip with boundary on $C_{1}$ and $C_{2}$ joining that chord to some generator in $\widetilde{\Sigma}_{N}$. Applying the map $\psi: \widetilde{\Sigma}\to \widetilde{G}'$, we obtain a map from the strip to $\widetilde{G}'$ with boundary on $\psi(C_{1})$ and $\psi(C_{2})$. If we fix $N$ and allow $M$ to increase, then this strip necessarily becomes longer in the sense that its boundaries must traverse longer portions of $C_{1}$ and $C_{2}$. Since $C_{1}$ and $C_{2}$ are lifts of compact circles in $\Sigma$, for large enough $M$ relative to $N$ the boundary of the strip will traverse a whole fundamental domain for the coverings $C_{i}\to L_{i}$. Then we may use this strip to construct an isotopy of $L_{1}$ to $L_{2}$.
  \end{proof}

  It remains to consider the case where $L_{1}$ is isotopic to $L_{2}$, and the represent a nonzero class in $O_{\pi}$; in this case the same conclusion as above holds.

  \begin{lemma}
    \label{lem:isotopic}
        Consider components $C_{i}$ of $\pi^{-1}(L_{i})$ ($i=1,2$) and their intersections $C_{i,N}$ with $\widetilde{\Sigma}_{N}$. Suppose that $L_{1}$ is isotopic to $L_{2}$. Then for every $N$ and for sufficiently large $M > N$, the restriction map
    \begin{equation}
      r_{M,N}: CW(C_{1,M},C_{2,M}) \to CW(C_{1,N},C_{2,N})
    \end{equation}
    vanishes on boundary chords.
  \end{lemma}

  \begin{proof}
    Supposing the conclusion is not true, then as in the proof of the previous lemma, we see that $C_{1}$ and $C_{2}$ must themselves be isotopic in $\widetilde{\Sigma}$. In that case, we may focus our attention on a neighborhood of these arcs in $\widetilde{\Sigma}$ that is an infinite cylinder with a periodic sequence of punctures. Then the conclusion follows from a direct analysis of this case. 
  \end{proof}
  
  Now we turn to question of compatibility with restriction, that is, whether $\lambda^{\text{naive}}_{N} \cong r_{M,N}\circ \lambda_{M}$ for $M \geq N$. On the one hand, it is obvious from the construction that compatibility is true at the level of objects, but for morphisms is  less clear. In fact, the functors $\lambda^{\text{naive}}_{N}$ need to be ``corrected'' to include contributions from boundary chords in order to achieve compatibility, since it is possible that an interior chord maps to a boundary chord under restriction.

  To define the corrected functor $\lambda_{N}$ on morphisms, we apply $\lambda_{M}^{\text{naive}}$ for $M$ sufficiently large compared to $N$, and then apply Viterbo restriction $r_{M,N}$:
  \begin{equation}
    \lambda_{N} = r_{M,N}\circ \lambda^{\text{naive}}_{M} \quad (M \gg N). 
  \end{equation}
  Thus, for each $N$, $\lambda_{N}$ may differ from $\lambda_{N}^{\text{naive}}$ by contributions from boundary generators in $\widetilde{\Sigma}_{N}$. Now Lemmas \ref{lem:compactprojection}, \ref{lem:nonisotopic}, and \ref{lem:isotopic} imply that, after a further restriction, these boundary generators are sent to zero, so the difference disappears:
  \begin{equation}
    r_{M,N}\circ \lambda_{M} = r_{M,N} \circ \lambda^{\text{naive}}_{M} = \lambda_{N}.
  \end{equation}
  This proves the following proposition.

  \begin{proposition}
    The functors $\lambda_{N} : \mathrm{Fuk}(\Sigma) \to \Fw(\widetilde{\Sigma}_{N})$ are compatible with Viterbo restriction, and hence their limit defines a functor $\lambda : \mathrm{Fuk}(\Sigma) \to \Flf(\widetilde{\Sigma})$.
  \end{proposition}

\subsection{$O_\pi$-equivariant branes}

Now we seek to establish the equivalence $\mathrm{Fuk}(\Sigma) \cong \Flf(\widetilde{\Sigma})^{O_\pi}$. We shall take a two-step approach: In this section we show there is a fully faithful embedding $\mathrm{Fuk}(\Sigma)\to \Flf(\widetilde{\Sigma})^{O_\pi}$, and in the next section we show that this functor is essentially surjective.

Begin by considering a Lagrangian submanifold $L \subset \Sigma$. Let $\tilde{L}$ be the pullback of $L$ with respect to the covering $\pi : \widetilde{\Sigma} \to \Sigma$, so that $\tilde{L}$ is a (possibly disconnected) Lagrangian submanifold of $\widetilde{\Sigma}$. It is clear that $\tilde{L}$ is invariant under the action of $O_\pi$; we will proceed by showing that $\tilde{L}$ can be made into a $O_\pi$-equivariant object of $\Flf(\widetilde{\Sigma})$, and that the association $L \mapsto \tilde{L}$ realizes, at the level of objects, a functor $\lambda : \mathrm{Fuk}(\Sigma) \to \Flf(\widetilde{\Sigma})^{O_\pi}$.

We can be a bit more explicit about how the functor $\lambda: \mathrm{Fuk}(\Sigma) \to \Flf(\widetilde{\Sigma})$ works when applied to objects and morphisms. Recall that $\lambda$ is the limit of $\lambda_{N}: \mathrm{Fuk}\to \Flf(\widetilde{\Sigma}_{N})$, which on objects takes $L$ to $\pi^{-1}(L) \cap \widetilde{\Sigma}_{N}$, and it takes a morphism $x \in CW(L_{0},L_{1})$ to the sum of the lifts of $x$ that lie in $\widetilde{\Sigma}_{N}$, possibly plus some boundary generators in $\widetilde{\Sigma}_{N}$. As $N \to \infty$, the chain complexes $CW(\lambda_{N}(L_{0}),\lambda_{N}(L_{1}))$ form an projective system of chain complexes, and we can pass to the limit complex $\lim_{N\to \infty}CW(\lambda_{N}(L_{0}),\lambda_{N}(L_{1}))$.

\begin{lemma}
  The projective system $CW(\lambda_{N}(L_{0}),\lambda_{N}(L_{1}))$ satisfies the Mittag-Leffler condition. Thus $\lim^{1}$ vanishes, and $\lim_{N\to \infty}CW(\lambda_{N}(L_{0}),\lambda_{N}(L_{1}))$ computes the homotopy limit of the system.
\end{lemma}

\begin{proof}
  For a projective system of groups $(A_{N}, r_{M,N})$, the Mittag-Leffler condition is the statement that, for each $N$, the image of $r_{M,N}$ stabilizes for large $M$. Lemmas \ref{lem:compactprojection}, \ref{lem:nonisotopic}, and \ref{lem:isotopic} imply that, for large $M$, the image of $r_{M,N}$ equal to the image of the interior generators under $r_{M,N}$, and this set only increases as $M$ increases. 
\end{proof}

We introduce the notation $CF^{\mathrm{lf}}$ for the limit complex, which is the complex computing morphisms in $\Flf(\widetilde{\Sigma})$:
\begin{equation}
  CF^{\mathrm{lf}}(\lambda(L_{0}),\lambda(L_{1})) = \lim_{N\to \infty}CW(\lambda_{N}(L_{0}),\lambda_{N}(L_{1}))
\end{equation}
On the other hand, the limit on the right-hand can be described geometrically in terms of intersection points of the full preimages $\pi^{-1}(L_{0}), \pi^{-1}(L_{1})$, since it is the direct product (not direct sum) of one-dimensional vector spaces associated to each intersection point:
\begin{equation}
  CF^{\mathrm{lf}}(\lambda(L_{0}),\lambda(L_{1})) = \prod_{x \in \pi^{-1}(L_{0})\cap \pi^{-1}(L_{1}) }\Lambda_{x},
\end{equation}
where $\Lambda_{x}$ denotes the one-dimensional vector space incorporating orientation lines and possibly local systems on $L_{0}$ and $L_{1}$.

There is a natural functor $\iota: \Flf(\widetilde{\Sigma})^{O_\pi} \to \Flf(\widetilde{\Sigma})$ given by forgetting the $O_\pi$-equivariant structure.

\begin{proposition}
\label{fullyfaithful}
  The functor $\lambda$ factors through $\iota$, so there is a functor $\lambda^{O_\pi} : \mathrm{Fuk}(\Sigma)\to \Flf(\widetilde{\Sigma})^{O_\pi}$ such that $\lambda = \iota \circ \lambda^{O_\pi}$. The functor $\lambda^{O_\pi}$ is a quasi-equivalence onto its image.
\end{proposition}

\begin{proof}
  For any Lagrangian brane $L$, $\lambda(L)$ is a $O_\pi$-invariant object, since it is the limit of restrictions of supported on the $O_\pi$-invariant Lagrangian $\pi^{-1}(L)$, and $\lambda(L)$ carries a natural $O_\pi$-equivariant structure owing to this fact. If we take two Lagrangian branes $L_{0}$ and $L_{1}$, the the $O_\pi$-equivariant morphisms between the corresponding objects of $\Flf(\widetilde{\Sigma})^{O_\pi}$ is nothing but the $O_\pi$-invariant subspace of $CF^{lf}(\lambda(L_{0}),\lambda(L_{1})) = \prod_{x \in \pi^{-1}(L_{0})\cap \pi^{-1}(L_{1}) }\Lambda \cdot x$, which is naturally identified with $CF(L_{0},L_{1})$, the morphism complex in $\mathrm{Fuk}(\Sigma)$. Thus the functor $\lambda$ maps $CF(L_{0},L_{1})$ isomorphically onto the subspace $CF^{lf}(\lambda(L_{0}),\lambda(L_{1}))^{O_\pi}$. 
\end{proof}

\subsection{Essential surjectivity}

In this section we prove that the functor $\lambda^{O_\pi} : \mathrm{Fuk}(\Sigma)\to \Flf(\widetilde{\Sigma})^{O_\pi}$ is essentially surjective. One of the key ingredients is the geometrization result of Haiden--Katzarkov--Kontsevich \cite{haiden2017flat}, and its extension to the $\bZ_2$-graded setting given in \cite{auroux2020fukaya}. Let us recall the statement that we will need.

Let $S$ be a finite type Riemann surface either compact or with punctures. Following 
 \cite{auroux2020fukaya} we say that an object $L$ in $\mathrm{Fuk}(S)$ is \emph{geometric} if it is quasi--isomorphic to a union of immersed arcs or curves in $S$.\footnote{In fact as in \cite{auroux2020fukaya} if $S$ is non-compact one should also consider curves in the Liouville completion of $S$, but this will not affect our argument.}
 \begin{theorem}[\cite{haiden2017flat}, \cite{auroux2020fukaya}]
 \label{geometricity}
 If $S$ has more than three punctures than every object in $\mathrm{Fuk}^w(S)$ is geometric. 
 \end{theorem}
 \begin{proof}
Theorem 4.2 of \cite{auroux2020fukaya} states that if $S$ is a finite type Riemann surface with punctures every object of the category of twisted complexes of Lagrangian branes is geometric. In fact this was proved in \cite{haiden2017flat} in the $\bZ$-graded case, and Auroux--Smith show that the proof applies also to the $\bZ_2$-graded setting. Since our $\mathrm{Fuk}^w(S)$ is the split closure of the category of twisted complexes, this is not quite yet the statement we need. However Corollary 4.12 of \cite{auroux2020fukaya} shows that if $S$ has more than three punctures, then the category of twisted complexes of Lagrangian  branes is split closed, and this concludes the proof.
\end{proof}

\begin{corollary}
\label{geometricupstairs}
Every object in  $\Flf(\widetilde{\Sigma})$ is geometric.
\end{corollary}
\begin{proof}
Note that the notion of geometricity applies equally well to $\Flf(\widetilde{\Sigma})$ even though $\widetilde{\Sigma}$ is not of finite-type. The only difference with the finite-type case is that objects in $\Flf(\widetilde{\Sigma})$ might consist of infinite (but locally finite) unions of arcs and immersed curves. Recall from \cite{haiden2017flat} and \cite{auroux2020fukaya} that the geometric models for an objects in $\mathrm{Fuk}(S)$ given by Theorem \ref{geometricity} are in fact unique: more precisely, they are unique up to a unique isotopy. 

Let us place ourselves in the setting of Section \ref{sec:loc-fin}. Recall that we consider the limit 
$$
 \Flf(\widetilde{\Sigma}) = \varprojlim_{N\to \infty} \Fw(\widetilde{\Sigma}_{N})
 $$
 The surfaces $\widetilde{\Sigma}_{N}$ are finite type Riemann surfaces with more than $3$ punctures, so they fall within the scope of Theorem \ref{geometricity}. That is, all objects in 
 $\Fw(\widetilde{\Sigma}_{N})$ are geometric. The Viterbo restriction functors
$$
   r_{M,N} : \Fw(\widetilde{\Sigma}_{M}) \to \Fw(\widetilde{\Sigma}_{N}),
$$
admit an especially simple description on geometric objects due to Lemma \ref{geometricrestriction}: they send a geometric object $L \in \Fw(\widetilde{\Sigma}_{M})$ to its intersection with $\widetilde{\Sigma}_{N}$, or, in the case of an object parallel to a boundary component of $\widetilde{\Sigma}_{N}$, to the corresponding object sitting inside the Liouville completion of $\widetilde{\Sigma}_{N}$.

Now, by definition an object in 
$\Flf(\widetilde{\Sigma})$ is given by a tuple 
$$
L:=(\{L_N\}, \{\alpha_{N}\})_{N \in \bN}
$$ where $L_N$ is an object in 
$\Flf(\widetilde{\Sigma}_N)$ and $ 
\alpha_N: L_{N-1} \stackrel{\simeq} \to r_{N, N-1}(L_{N})
$ 
is a quasi-equivalence. Geometrization for $\widetilde{\Sigma}_N$ and $\widetilde{\Sigma}_{N-1}$ tells us that  if we fix a geometric model $C_{N-1}$ for $L_{N-1}$ we can find a compatible geometric model $C_N$ for $L_N$, in the sense that $C_{N-1}$ is obtained by intersecting 
$C_N$ with $\Sigma_{N-1}$. The idea is the following: we pick a geometric representative $C_N$ of $L_N$; then the intersection of $C_N$ with $\Sigma_{N-1}$ provides another geometric model for $L_{N-1}$, possibly different from $C_{N-1}$. By uniqueness however this new geometric model must differ from $C_{N-1}$ only up to an isotopy, which we can lift up to $\Sigma_N$ so as to get a new geometric representative of $L_N$ with the desired compatibility property with $C_{N-1}$.  This allows us to build up recursively  a geometric model for an object $L$ in $\Flf(\widetilde{\Sigma})$: namely, at each stage $\widetilde{\Sigma}_N$ we pick a geometric  model which is compatible with the geometric model selected at stage $N-1$. This concludes the proof. 
\end{proof}
 
\begin{proposition}
  \label{componentsofequivariant}
  Let $L$ be an object of $\Flf(\widetilde{\Sigma})^{O_{\pi}}$, and let $\overline{L}$ be the object of $\Flf(\widetilde{\Sigma})$ obtained by forgetting the equivariant structure. Then any indecomposable component $C$ of $\overline{L}$ is isomorphic to either a compact immersed curve with local system, or else an arc that is fixed by some subgroup $H$ of $O_{\pi}$ isomorphic to $\Z$. In either case, $C$ is isomorphic to a lift of a compact immersed curve in $\Sigma$.
\end{proposition}
\begin{proof}
  Let $L$ be an object of $\Flf(\widetilde{\Sigma})^{O_{\pi}}$. By forgetting the $O_{\pi}$-equivariant structure, we obtain an $O_{\pi}$-invariant object $\overline{L}$ in $\Flf(\widetilde{\Sigma})$. Applying geometrization to $\overline{L}$, we see that $\overline{L}$ is a direct sum of immersed curves with local system and arcs in $\widetilde{\Sigma}$, such that the whole collection is $O_{\pi}$-invariant.

  Let $C$ be a component of $\overline{L}$. Let $H$ be the subgroup of $O_{\pi}$ that fixes $C$ ($H$ may be trivial). Then $\overline{L}$ must contain the direct sum $\oplus_{g \in O_{\pi}/H}g\cdot C$. Once again we may regard $\widetilde{\Sigma}$ as a collection of $2g$-punctured spheres arranged on the lattice $O_{\pi}$ and glued together. Let $\widetilde{\Sigma}_{0}$ be one of the punctured spheres in this decomposition: it is a fundamental domain for the $O_{\pi}$-action. Because the object $\oplus_{g \in O_{\pi}/H}g\cdot C$ is assumed to be locally finite, the intersection of this object with $\widetilde{\Sigma}_{0}$ consists of finitely many immersed curves and arcs on $\widetilde{\Sigma}_{0}$. Matching the corresponding boundaries of $\widetilde{\Sigma}_{0}$ gives the original surface $\Sigma$, and the immersed curves and arcs must match up to give a compact immersed submanifold of $\Sigma$. It is then clear that $\oplus_{g \in O_{\pi}/H}g\cdot C$ is nothing but the pullback of this submanifold under the covering map $\pi: \widetilde{\Sigma} \to \Sigma$. This shows that our original component $C$ is a lift with respect to $\pi$ of an immersed curve in $\Sigma$.

  Recall that there is a homomorphism $\tau: \pi_{1}(\Sigma) \to O_{\pi}$. If $\iota: S^{1}\to \Sigma$ is an immersion, we may consider $\tau \circ \iota_{*}: \pi_{1}(S^{1}) \to O_{\pi}$. Either $\tau \circ \iota_{*}$ vanishes or it does not. If it vanishes, then $\iota$ can be lifted to an immersion $\tilde{\iota} : S^{1} \to \widetilde{\Sigma}$; if it does not vanish, then it may be lifted to an immersion $\tilde{\iota}: \R \to \widetilde{\Sigma}$. In the former case, the lift is not fixed by any subgroup of $O_{\pi}$, while in the latter case, the lift is fixed by a subgroup $H\cong \Z$: indeed, $H$ is nothing but the image of $\tau\circ \iota_{*}$. This establishes that any noncompact component of $\overline{L}$ must be fixed by a subgroup of $O_{\pi}$ that is isomorphic to $\Z$.
\end{proof}

To understand the next step in the argument, consider the following. Let $L$ be an object of $\Flf(\widetilde{\Sigma})^{O_{\pi}}$, and let $\overline{L}$ be the underlying object in $\Flf(\widetilde{\Sigma})$. By Proposition \ref{componentsofequivariant}, we know that $\overline{L}$ decomposes into a direct sum of objects of the form $\tilde{C} = \bigoplus_{g \in O_{\pi}/H} g \cdot C$, where $C$ is a lift of an immersed curve $\overline{C}$ in $\Sigma$ (equipped with a local system), and $H$ is either trivial or isomorphic to $\Z$. 

Now we remember the $O_{\pi}$-equivariant structure on our original object $L$. It is possible that the equivariant structure, which is a certain collection of isomorphisms $\Phi_{g}: \overline{L} \to \overline{L}$, could mix together different components in the direct sum decomposition. This does indeed occur in the case where $L$ is the pullback of a simple closed curve, with nontrivial class in $O_{\pi}$, equipped with higher-rank local system. Nevertheless, we wish to show that any equivariant structure on $\overline{L}$ is isomorphic to one where there is no mixing between components corresponding to nonisotopic immersed curves in $\Sigma$. For this we use again the geometrization result applied to intermediate coverings $\widetilde{\Sigma}/H$ for various subgroups $H \subset O_{\pi}$. Recall that a saturate sublattice $H \subset O_{\pi}$ is a subgroup such that if $n\in \mathbb{Z}$ and $ng \in H$ then $g \in H$. For a saturated sublattice we have $O_{\pi} \cong H \times O_{\pi}/H$.

\begin{lemma}
  \label{lem:split}
  Let $\overline{L}$ be as above, suppose that $H \subset O_{\pi}$ is a proper saturated sublattice. The components $C$ of $\overline{L}$ into two kinds:
  \begin{enumerate}
  \item those such that the class of $\pi(C)$ in $O_{\pi}$ lies in $H$, and
  \item those such that the class of $\pi(C)$ does not lie in $H$.
  \end{enumerate}
   Then any equivariant structure on $\overline{L}$ is isomorphic to one for which there is no mixing between these two kinds, and furthermore there is no mixing between summands of the first kind corresponding to nonisotopic immersed curves.
\end{lemma}

\begin{proof}
  Consider the surface $\widetilde{\Sigma}/H$. Because $H$ is a proper saturated sublattice, it has rank less than $g$, and so $\widetilde{\Sigma}/H$ is a non-finite-type Stein surface, equipped with a pants decomposition. Hence the geometry of this surface may be expressed in terms of topological Fukaya categories, and so several of the results we have shown for $\widetilde{\Sigma}$ hold just as well for $\widetilde{\Sigma}/H$. In particular, every object of $\Flf(\widetilde{\Sigma}/H)$ is geometric, and we have a descent equivalence
  \begin{equation}
    \Flf(\widetilde{\Sigma})^{H} \cong \Flf(\widetilde{\Sigma}/H).
  \end{equation}
  From this it follows that
  \begin{equation}
    \Flf(\widetilde{\Sigma})^{O_{\pi}} \cong \Flf(\widetilde{\Sigma})^{H \times O_{\pi}/H} \cong \Flf(\widetilde{\Sigma}/H)^{O_{\pi}/H},
  \end{equation}
  So an $O_{\pi}$-equivariant structure on $\overline{L}$ is obtained from an $O_{\pi}/H$-equivariant object $L'$ on $\widetilde{\Sigma}/H$ by pullback.

  Now the underlying object $\overline{L}'$ in $\Flf(\widetilde{\Sigma}/H)$ is geometric; components $C$ of the first kind correspond to compact components of $\overline{L}'$, while those of the second kind correspond to noncompact components, which are fixed by some subgroup of $O_{\pi}/H$ isomorphic to $\Z$. We refer to the components of $\overline{L}'$ as being of the first or second kind respectively. Now in $\widetilde{\Sigma}/H$, there is no mixing between components of the first kind, simply by geometricity. Furthermore, because the group $O_{\pi}/H$ freely permutes the summands of the first kind (none are fixed by any nontrivial subgroup), the $O_{\pi}/H$-equivariant structure on $\overline{L}'$ is isomorphic to one for which there is no mixing between components of the first and second kinds. Pulling back such an equivariant structure gives the desired result.
\end{proof}

\begin{proposition}
  Any object $L$ in $\Flf(\widetilde{\Sigma})^{O_{\pi}}$ splits $O_{\pi}$-equivariantly into a direct sum of objects of the form $\tilde{C}^{\oplus r}$, where $\tilde{C} = \bigoplus_{g \in O_{\pi}/H} g \cdot C$, where $C$ is a lift of an immersed curve $\overline{C}$ in $\Sigma$ equipped with a local system, and $H$ is either trivial or isomorphic to $\Z$. 
\end{proposition}

\begin{proof}
  \label{prop:split}
  Since $O_{\pi}$ is a lattice of rank $g \geq 2$, every element of $O_{\pi}$ is contained in some nontrivial proper saturated sublattice $H$. Applying Lemma \ref{lem:split} to all such sublattices, we see that $L$ splits equivariantly into a direct sum according to different sublattices in $O_{\pi}$, and furthermore that within each sublattice there is a splitting according to isotopy classes of the projection $\pi(C)$. 
\end{proof}

\begin{theorem}
\label{essential}
  The functor  $\lambda^{O_\pi} : \mathrm{Fuk}(\Sigma)\to \Flf(\widetilde{\Sigma})^{O_\pi}$ is an equivalence. 
  \end{theorem}
\begin{proof}
  Since we have already proved that $\lambda^{O_\pi}$ is fully-faithful, it remains to show that it is essentially surjective. By \ref{prop:split}, it suffices to show that an equivariant object of the form $\tilde{C}^{\oplus r}$ where $\tilde{C} = \bigoplus_{g \in O_{\pi}/H} g \cdot C$ and $C$ is a lift of an immersed curve $\overline{C}$ is in the image. In the case where $H$ is trivial, $L$ is equivariantly isomorphic to the pullback of $\overline{C}$ with its natural equivariant structure. If $H \cong \Z$, then the $O_{\pi}$-equivariant structure on $\tilde{C}^{\oplus r}$ induces an $H$-equivariant structure on $C^{\oplus r}$. Since $\overline{C} = C/H$, this equivariant structure defines a rank $r$ local system on $C$, such that $\tilde{C}^{\oplus r}$ is equivariantly isomorphic to the pullback of $C$ with this local system.
\end{proof}

Corollary \ref{mainlocalcor} below is our main  local-to-global result for the Fukaya category of a Riemann surface 
$\Sigma$ of genus $g \geq 2$. We will actually prove two parallel results: one  for the  Fukaya category (which is defined over the Novikov field 
$\Lambda$) and one for the balanced Fukaya category (which is defined over a fixed ground field $k$). The statement is phrased in terms of the sheaf of categories on graphs $\cB^\omega$ which we introduced in Section \ref{Grcat}, and that it is defined over a ground ring $\kappa$.
%We will use the notations  $\cB^\omega_k$ and $\cB^\omega_\Lambda$  to distinguish the cases $\kappa=\Lambda$ and $\kappa=k$. 

Let $G$ be the graph associated to the pants decomposition of $\Sigma$. We equip $G$ with weights $\alpha(v) = q^{-A(v)}$ and $\beta(t) = q^{B(t)}$ where $A(v)$ is the area of the pair of pants corresponding to the vertex $v$ and $B(t)$ is the area of the annulus corresponding to the edge $t$. Set $\gamma$ equal to the product of these weights, so that $\gamma = q^{-\text{Area}(\Sigma)}$. We also choose a set of framings $f$.  The covering graph $\widetilde{G}$ is given the pull-back weights and framings.

From Definition \ref{def:graphcat}, we have $\cB(G,\gamma) = \cB(G,\alpha,\beta,f)$. We denote the compact objects in this category by $\cB^{\omega}_{\Lambda}(G,\gamma)$ to emphasize the coefficient field. We can also set all wieights to $1$ and work over $k$ to get a category $\cB^{\omega}_{k}(G,1)$.

\begin{corollary}
\label{mainlocalcor}
There are $\kappa\rbu$-linear equivalences of categories 
$$
\mathrm{Fuk}(\Sigma) \simeq \cB^\omega_\Lambda(G,\gamma), \quad \mathrm{Fuk}^{bal}(\Sigma, \theta) \simeq \cB^\omega_k(G,1)
$$
where in the former case $\kappa = \Lambda$ and in the latter $\kappa = k$.
\end{corollary}
\begin{proof}
We will  show first that $\mathrm{Fuk}(\Sigma) \simeq \cB^\omega_\Lambda(G,\gamma)$. We begin with the equivalence
\begin{equation*}
\Flf(\widetilde{\Sigma}) \simeq \mathrm{\mathrm{Fuk}}_\Lambda^{\mathrm{top}, \mathrm{lf}}(\widetilde{\Sigma})  
\end{equation*}
from Proposition \ref{lftopo}, where the Liouville structure $\widetilde{\Sigma}$ has been used to trivialize the dependence on the Novikov parameter. This equivalence is not compatible with the $O_{\pi}$ action on $\Flf(\widetilde{\Sigma})$ because the action of $O_{\pi}$ does not preserve the Liouville structure. To remedy this, recall the equivalence from Proposition \ref{flws}
\begin{equation*}
   \varprojlim \cB^{\omega}(G_{i}) \simeq \mathrm{\mathrm{Fuk}}^{\mathrm{top}, \mathrm{lf}}(\widetilde{\Sigma}). 
\end{equation*}
where $G_{i}$ is an exhaustion of $\widetilde{G}$ by finite-type subgraphs, and all weights are set to $1$ in this equivalence. Combining these, we obtain an equivalence
\begin{equation*}
 \Flf(\widetilde{\Sigma}) \simeq \varprojlim \cB^{\omega}(G_{i})
\end{equation*}
Now we may restore the weights on the right-hand side and the $q$-dependence on the left-hand side, so that this equivalence is $O_{\pi}$-equivariant. Then we have a chain of equivalences 
$$ 
\mathrm{Fuk}(\Sigma) \stackrel{(*)} \simeq \Flf(\widetilde{\Sigma})^{O_\pi}  \stackrel{(**)} \simeq (\varprojlim \cB(G_{i})^\omega )^{O_\pi}  \stackrel{(***)} \simeq \cB^\omega_\Lambda(G),$$ where equivalence $(*)$ was proved in Proposition \ref{essential}, equivalence $(**)$ was just proven, and equivalence $(***)$ follows from the argument in Proposition \ref{toamrt}. Composing them we obtain $\mathrm{Fuk}(\Sigma) \simeq \cB^\omega_\Lambda(G)$, which is what we wanted to show. 

Now let us pass to the balanced case. Recall that to construct the balanced category one considers first a version of the balanced Fukaya category  which is linear over $\Lambda^{fin} \subset \Lambda$, i.e. the subring of Novikov series with finitely many non-vanishing terms. This category was denoted $\cC(\Sigma,\theta)$ in Section \ref{variantsFukaya}; extension of scalars gives a $\Lambda^{fin}$-linear functor  
$$
\cC(\Sigma,\theta) \to \mathrm{Fuk}(\Sigma) \simeq \cC(\Sigma,\theta) \otimes_{\Lambda^{fin}} \Lambda
$$ 
%such that the $\Lambda$-extension of scalars of the image splits generates $\mathrm{Fuk}(\Sigma)$. 
%Note that  we can choose generators of $\cC(\Sigma,\theta)$ such that the extension of scalars is faithful on the Hom-spaces between these generators. 
We get a diagram %Then  we can define an intermediate category interpolating between the Fukaya category and its balanced version as in the following diagram 
$$
\mathrm{Fuk}(\Sigma)  \leftarrow \cC(\Sigma,\theta) \stackrel{(a)}\rightarrow 
\mathrm{Fuk}^{bal}(\Sigma, \theta). 
$$
 where $(a)$ is specialization at $q=1$.

The statement we care about is for $
\mathrm{Fuk}^{bal}(\Sigma, \theta)$; but it is actually more convenient to deduce it from  the analogous statement for $\cC(\Sigma,\theta)$. Namely, we will prove that 
\begin{equation}
\label{eqbal}
\cC(\Sigma,\theta)  \simeq \cB^\omega_{\Lambda^{fin}}(G).   
\end{equation} 
To get the result we want for $\mathrm{Fuk}^{bal}(\Sigma, \theta)$ is then enough to specialize at $q=1$.

To prove 
equivalence (\ref{eqbal}) we need to work around the following issue: being balanced is a global property, so gluing together local balanced branes we do not necessarily get a balanced brane. In other words, it is not immediately obvious that $\cC(\Sigma,\theta)$ is the limit of the categories $\cC(P, \theta)$, where $P$ are the pants making up the decomposition of $\Sigma$. %But this  is precisely what we want to prove: equivalence  (\ref{eqbal}) means exactly that 
%$\mathrm{Fuk}^{bal}_{\Lambda^{fin}}(\Sigma, \theta)$ \emph{is} equivalent to the limit of $\mathrm{Fuk}^{bal}_{\Lambda^{fin}}(P, \theta)$. 
However we have the following diagram
$$
\xymatrix{
\mathrm{Fuk}(\Sigma) \ar[r]^\simeq &   \cB^\omega_{\Lambda}(G) \\
\cC(\Sigma,\theta) \ar[u] \ar[r]^{\iota} &  \cB^\omega_{\Lambda^{fin}}(G) \ar[u]
}
$$
where 
\begin{itemize}
\item the vertical arrows are extension of scalars
\item the top equivalence was showed in the first part of the proof 
\item and $\iota$ is fully-faithful. 
\end{itemize}
The fully-faithfulness of $\iota$ can be deduced just from the local statement that, if $P$ is the pair-of-pants, then 
$\cC(P, \theta) \simeq \cB^\omega_{\Lambda^{fin}}(G_P)$, which  can be easily checked directly.\footnote{Here $G_P$ is the graph with one trivalent vertex.} Using that we can write $$
 \varprojlim \cC(P, \theta) \simeq \varprojlim \cB^\omega_{\Lambda^{fin}}(G_P) \simeq  \cB^\omega_{\Lambda^{fin}}(G)
$$
where $P$ runs over the pants decomposition of $\Sigma$. 
Balanced branes in $\Sigma$ can be  obtained by gluing together local balanced branes in the pants making up the decomposition. This gives us the fully-faithful embedding $\iota$. However in principle $\iota$ could fail to be essentially surjective because being balanced is a global property: in the limit $ \varprojlim \cC(P, \theta) $ we will have  branes which restrict to balanced branes on every pair-of-pants $P$,  but which might fail to be globally balanced.

However one can see that $\iota$ is in fact an equivalence. 
% The idea is very simple. %: we can choose explicit generators of the categories $\cB^\omega_\Lambda$, as these are built from matrix factorization categories, and are in fact globally equivalent to singularity categories of appropriately chosen normal crossing surfaces. Now the equivalence  
%$\mathrm{Fuk}(\Sigma)  \simeq \cB^\omega_{\Lambda}(G)$ depends on the fact that we can match generators of $\cB^\omega_{\Lambda}(G)$ with generators of 
%$\mathrm{Fuk}(\Sigma)$. 
 Choose balanced generators of $\mathrm{Fuk}(\Sigma)$ which we denote  $L_1, \ldots, L_N$. Let $M_i$ the corresponding generators of $\cB^\omega_{\Lambda}(G)$ under the equivalence $\mathrm{Fuk}(\Sigma)  \simeq \cB^\omega_{\Lambda}(G)$. 
 % The equivalence is geometric, so the $M_i$-s are not random  generators of $\cB^\omega_{\Lambda}(G)$:
 The $L_i$ are  geometric balanced branes, which can be defined in the Fukaya category independently of the ground ring; %, and are therefore independent on the ground ring of definition of the Fukaya category; 
  thus the $M_i$  will also be constructed in terms of the geometry of the singularity categories making up  $\cB^\omega_{\Lambda}(G)$ in a way that does not depend on the ground ring. In particular  we can lift them to generators of $\cB^\omega_{\Lambda^{fin}}$, which we still denote $M_i$. At the same time, by construction, the $L_i$ can be lifted  to  objects in $\cC(\Sigma,\theta)$, and clearly under $\iota$ the $L_i$ are sent to the $M_i$
  $$
  L_i \in \cC(\Sigma,\theta) \mapsto \iota(L_i) \simeq M_i \in \cB^\omega_{\Lambda^{fin}}.
  $$
  So $\iota$ is a fully-faithful embedding such that its image contains a collections of generators of $\cB^\omega_{\Lambda^{fin}}$. Thus is has to be an equivalence, which is what we needed to show. 
 \end{proof}

We record a remarkable consequence of Proposition \ref{essential} and its proof, namely that all objects in the Fukaya category of a closed Riemann surface of genus $g \geq 2$ are geometric. Corollary \ref{geometric} provides an answer to the second open Problem in Section 7 of \cite{haiden2017flat}.

The geometric idea is simple. The objects in $\Flf(\widetilde{\Sigma})^{O_\pi}$ can be viewed as $O_\pi$-invariant branes in $\Flf(\widetilde{\Sigma})$ equipped with an equivariant structure. In particular, by the geometricity of $\widetilde{\Sigma}$ they can be viewed as immersed Lagrangians in $\widetilde{\Sigma}$. It is easy to see that $(\lambda^{O_\pi})^{-1}$ has a very simple definition on geometric objects: it is just the push-forward of the immersed Lagrangian along $\pi: \widetilde{\Sigma} \to \Sigma$; then we use the equivariant structure to equip the resulting immersed Lagrangian in $\Sigma$ with a brane structure. This shows that, using Proposition \ref{essential}, geometricity for $\widetilde{\Sigma}$ readily implies geometricity for $\Sigma$.

\begin{corollary}
\label{geometric}
Let $S$ a compact Riemann surface of genus $g \geq 2$. Then every object in $\mathrm{Fuk}(S)$ is geometric. 
\end{corollary}
\begin{proof}
Let us give a more formal argument based on the proof of Proposition \ref{essential}. There we proved that the functor $\lambda^{O_\pi} : \mathrm{Fuk}(\Sigma)\to \Flf(\widetilde{\Sigma})^{O_\pi}$ is essentially surjective. In fact we showed a  stronger result: namely for every object $L$ in $\Flf(\widetilde{\Sigma})^{O_\pi}$ we constructed a geometric object $L'$ in $\mathrm{Fuk}(\Sigma)$ such that $\lambda^{O_\pi}(L') \simeq L$. Now since $\lambda^{O_\pi}$ is an equivalence, this immediately implies that every object in $\mathrm{Fuk}(\Sigma)$ is equivalent to a geometric object, which is what we wanted to show. 
\end{proof}

%% file: hms.tex
\section{HMS for compact surfaces}
\label{hms}
Our description of the Fukaya category of a compact Riemann surface of genus $g \geq 2$ has  significant applications to HMS. In particular, we recover  Seidel and Efimov's HMS  results  for curves of genus $g \geq 2$ \cite{seidel2011homological} \cite{efimov2012homological}. Our method of proof  however is different,  and can help clarify why these mirror constructions actually work. %Both Seidel and Efimov start from carefully designed, and somewhat ad hoc, superpotentials. The proof then goes through familiar but delicate steps in HMS: matching generators, and deformation theory.  
%In particular, as pointed out by Seidel in the Introduction of \cite{seidel2011homological}, his proof does not apply to the other constructions of LG mirrors of curves   proposed in the literature. 
% 
%Our approach makes manifest that it is only the combinatorics of the singular locus of the superpotential that comes into play. This makes it completely straightforward to compare different mirror constructions: it is enough to check the shape of the singular locus of the superpotential. 
%This viewpoint   has several   benefits: our proof of HMS applies to all genera, and works equally well in the compact or punctured setting; and  it clarifies that HMS depends on a simple geometric relationship between the surface and its mirror. 
%As we pointed out earlier, 
Surfaces have actually  many different geometrically meaningful mirrors. The techniques developed in this paper allow us to easily check that these different models indeed give rise to an HMS equivalence, alternative to Seidel's and Efimov's. We focus on two constructions:  
\begin{enumerate}
\item Hori--Vafa mirror symmetry matches a toric CY 3-folds $X_\Sigma$,  equipped with a superpotential $W_\Sigma$ to a punctured surface $\Sigma$ equipped with a pants decomposition. We show that the Hori--Vafa picture holds in complete generality for all toric 3-folds, regardless of the CY assumption. %, and even in the compact setting where no superpotential is available.
\item We show that the Mumford degeneration of  the Jacobian provides a natural mirror LG model for  Riemann surfaces of genus two. The same construction appears, for the other direction of HMS, in beautiful recent work of Cannizzo \cite{cannizzo2020categorical}.
\end{enumerate}

One prefatory remark about the results of this section is in order. We consider the Fukaya category either over $\kappa = \Lambda$, or the balanced Fukaya category over $\kappa = k$; in either case it is considered as a $2$-periodic, that is $\kappa\rbu$-linear, category. The equivalences described below can be interpreted either as $\kappa$-linear equivalences, or they can be upgraded to $\kappa \rbu$-linear equivalences with the caveat that the $\kappa\rbu$-linear structure on the B-side is only determined up to a rescaling of $u$. We do not pin down this rescaling factor precisely, though in principle this can be accomplished with a closer analysis of the categories of matrix factorizations. We also remark that this scaling factor can be trivialized in the case of open surfaces and their mirrors.

  Let $G$ be a finite graph with no loops, and with vertices of valency $1$ or $3$. Let $G^\circ$ be the graph obtained from $G$ by removing the $1$-valent vertices. Up to symplectomorphism, there is a unique Riemann surface  $\Sigma_G$ (possibly with boundary) such that $G^\circ$ is the dual intersection complex of a pants decomposition of $\Sigma$. The $1$-valent vertices of 
  $G$ correspond naturally to  a subset of the components of $\partial \Sigma_G$: we decorate each of these components with a stop; let 
  $S_G \subset \partial \Sigma_G$  be the collection of stops obtained in this way.

    \begin{theorem}
    \label{main}
    Let $T$ be a smooth variety of dimension $3$ and let $X \subset T$ be a simple normal crossing divisor of the form $X = f^{-1}(0)$ for some morphism $f : T \to \bA^{1}$. Assume that $X$ has graph-like singular locus, and that the dual intersection complex of $X$ is orientable. Then there is an equivalence of categories
    $$
   \mathrm{Fuk}(\Sigma_{G(X)}, S_{G(X)}) \simeq 
    \Dsing(X) 
    $$
  \end{theorem}
%   Before giving a proof, let us comments on the objects appearing in the statement. First of all, the statement holds over a ground ring $\kappa$ of characteristic $0$ which is equal to either a fixed ground field $k$ (of char. 0), or the universal (over $k$) Novikov field $\Lambda$. 
% 
% 
% The variety $T$ and $X$ are defined over $\kappa$, and therefore  $\Dsing(X)$ is a $\kappa$-linear triangulated dg category. On the other hand, the category $\mathrm{Fuk}(\Sigma_{G(X)}, S_{G(X)})$ is to be interpreted in the following way
% \begin{itemize}
% \item if $\Sigma_{G(X)}$ has a   boundary, by $\mathrm{Fuk}(\Sigma_{G(X)}, S_{G(X)}) $ we mean the (partially)  wrapped Fukaya category of  $\Sigma_{G(X)}$, which is an exact symplectic surface. If $\Sigma_{G(X)}$ is a compact Riemann surface without boundary, then $S_{G(X)} = \varnothing$ and $\mathrm{Fuk}(\Sigma_{G(X)}, S_{G(X)})$ stands for the ordinary Fukaya category of $\Sigma_{G(X)}$. 
% \item 
% \end{itemize}
%  

  \begin{proof} 
Let us first make a comment of notation:  
if $\Sigma_{G(X)}$ has a   boundary, by $\mathrm{Fuk}(\Sigma_{G(X)}, S_{G(X)}) $ we mean the (partially)  wrapped Fukaya category of  $\Sigma_{G(X)}$, which is an exact symplectic surface. If $\Sigma_{G(X)}$ is a compact Riemann surface without boundary, then $S_{G(X)} = \varnothing$ and $\mathrm{Fuk}(\Sigma_{G(X)}, S_{G(X)})$ stands for the balanced Fukaya category of $\Sigma_{G(X)}$. The theorem is an immediate consequence of the results we obtained in Section \ref{mfat} and \ref{csae}, and in our previous article \cite{pascaleff2019topological}. Indeed,  the claim is obtained by composing the equivalences
 $$
\Dsing(X) \stackrel{(*)}\simeq \mathcal{B}^\omega(G(X)) \stackrel{(**)} \simeq 
\mathrm{Fuk}(\Sigma_{G(X)}, S_{G(X)})
 $$
 Equivalence $(*)$ is given by Theorem \ref{matrixgraph}, after taking compact objects. Let us make some comments on equivalence $(**)$: 
 \begin{itemize}
 \item If $G(X)$ is not compact, or is compact but has 1-valent vertices, then $\Sigma_{G(X)}$ is Riemann surface with boundary, and therefore we are in the exact setting. In this case equivalence $(**)$ follows from Theorem 8.3  in \cite{pascaleff2019topological}. 
 \item If $G(X)$ is a compact  and 3-valent graph, then $\Sigma_{G(X)}$ is a compact Riemann surface without boundary. In this setting, equivalence $(**)$ was obtained in Corollary \ref{mainlocalcor}. \end{itemize}
  \end{proof}
% \begin{remark}
%Theorem \ref{main} is formulated in  a generality which is sufficient for our purposes, but it could be extended further to the case where $X$ is a stacky normal crossing divisor.  Let us briefly how such a theory would look like, deferring to future work for a full development. First one would have to work \emph{weighted} graphs $G$, where vertices and edges are labelled by finite groups. On the algebraic side, this would enco the positive integers would indicate appropriate covering 
%  \end{remark}

Let $\Sigma_g$ be a compact Riemann surface of genus $g \geq 2$ without boundary. Let us denote by $(Y_g, W_g)$ the mirror LG model of $\Sigma_g$ proposed by Seidel (for $g=2$) and Efimov (for $g >2$) in \cite{seidel2011homological} \cite{efimov2012homological}. We refer the reader to the original papers for the explicit construction of $(Y_g, W_g)$. 
\begin{corollary}[Theorem 1.1 \cite{seidel2011homological}, Theorem 1.1 \cite{efimov2012homological}]
There is an equivalence of categories
$$
 \mathrm{Fuk}(\Sigma_g) \simeq  \mathrm{MF}(Y_g, W_g) 
$$
\end{corollary}
\begin{proof}
The equivalence follows from Theorem \ref{main}. We only need to check that $X_g:=W_g^{-1}(0)$ has graph-like singular locus, and that $G(X_g)$ is isomorphic to the dual intersection complex of a triangulation of $\Sigma_g$.   For $g=2$, see Figure 1 in \cite{seidel2011homological}. As for $g >2$, note that Efimov's construction extends all steps of Seidel's proof to higher genus: the singular locus of $W_g$ has an analogous shape to the one of Seidel, except the genus can be arbitrary; in particular, as in Seidel's proof, $G(X_g)$ encodes the dual intersection complex of a triangulation of $\Sigma_g$. \end{proof}

Let $Y$ be a smooth toric $3$-fold, and let $X$ be its toric boundary divisor. As pointed out in Remark \ref{toricnti}, $X$ has graph-like singular locus.  The graph $G(X)$ is trivalent, and it is compact if $Y$ is compact. However, $X$ is merely the zero locus of a section of a line bundle on $Y$ rather than that zero fiber of a morphism $Y \to \bA^{1}$. To fix this, choose a pencil that contains $X$; then we get a rational map $f : Y \dashrightarrow \bP^{1}$ with $X = f^{-1}(0)$ as a fiber. By resolving the singularities of the closure of the graph of $f$, we obtain a variety $\tilde{Y}$ that is birational to $Y$ and a morphism $\tilde{f} : \tilde{Y} \to \bP^{1}$. The resolution process may modify the fiber $X$ by a birational transformation: we call the corresponding fiber $\tilde{X} = \tilde{f}^{-1}(0)$. In many cases it is possible to choose a resolution such that $G(\tilde{X}) \cong G(X)$, that is, the combinatorics of the singular locus is unchanged. As a special case of Theorem \ref{main} we obtain the following. 
\begin{corollary}
\label{horivafa}
There is an equivalence of categories 
$$
\mathrm{Fuk}(\Sigma_{G(\tilde{X})}) \simeq \Dsing(\tilde{X}) $$
\end{corollary}
\begin{remark}
Let us explain why it is worthwhile to highlight, as  Corollary \ref{horivafa}, this specific setting of Theorem \ref{main}. The statement of Corollary \ref{horivafa} is a   generalization of Hori--Vafa mirror symmetry, which can be recovered as the special case  when $Y$ is Calabi--Yau. Hori--Vafa mirror symmetry is a widely  studied problem, with  contributions by many mathematicians starting with \cite{abouzaid2016lagrangian} \cite{bocklandt2016noncommutative} \cite{lee2016homological}; we addressed Hori--Vafa mirror symmetry from the perspective of the topological Fukaya category in our previous paper \cite{pascaleff2019topological}. If $Y$ is a smooth toric 3-fold we obtain in particular new mirrors for compact curves of genus $g \geq 3$ (the case $g=3$ corresponding to $Y=\mathbb{P}^3$). In future work, we will pursue generalizations of this picture to higher dimensions.
\end{remark}

\subsection{HMS for genus two curves and Mumford degenerations of Abelian surfaces}
We conclude this section by explaining how our methods give a  proof of a version of HMS for the genus two curve $\Sigma$, where the mirror is a Mumford degeneration of abelian surfaces. This type of  mirror symmetry was first suggested by Seidel in \cite{seidel2012some} in the more general contexts of hypersurfaces inside abelian varieties. In \cite{cannizzo2020categorical}  Cannizzo proves HMS for the genus two curve in this framework, but in the opposite direction: she compares the derived category of  $\Sigma$ with a kind of Fukaya--Seidel category of the mirror Mumford degeneration. We complement Cannizzo's result by showing that the reverse HMS  equivalence also holds. This is an immediate consequence of Theorem \ref{main}.

We give an informal treatment of the    set-up, referring to Section 10  of \cite{abouzaid2016lagrangian} (see especially Example 10.6) and Section 3 of \cite{cannizzo2020categorical} for a more thorough account of the underlying geometry.  Let $\Sigma$ be a Riemann surface of genus two. The starting point of the construction is a choice of  embedding of $\Sigma$ inside an abelian surface $A$, such as the Jacobian of $\Sigma$. We can consider a simultaneous degeneration of 
$\Sigma$ inside $A$ over a small disk $\mathbb{D}_\varepsilon$: the central fiber of the degeneration of the ambient abelian surface has two irreducible components, such that their normalizations are both isomorphic to $\mathbb{P}^2$;  %For simplicity we will work analytically over $\mathbb{C}$, but one can instead work over a completed DVR. 
the curve  
$\Sigma$ degenerates to a nodal curve $\Sigma_0$ given by two $\bP^1$ which lie in separate components of $A_0$, and meet transversely in three distinct points: $\Sigma_0$ is sometimes  called the  banana curve.
%, i.e. the so called banana curve. 

 The mirror family  can be computed via Legendre transform of these degeneration data, and is also  a Mumford degeneration of abelian surfaces. We obtain a family $W: Y \to \mathbb{D}_\varepsilon$ such that the smooth fibers are abelian surfaces, and the normalization of the singular fiber $Y_0$ is a blow-up of 
$\bP^2$ at the three torus-fixed points.   The LG model $(Y, W)$ gives another mirror of $\Sigma$,  which is different from Seidel's original construction. 

\begin{corollary}
There is an equivalence of categories
$$
 \mathrm{Fuk}(\Sigma) \simeq \mathrm{MF}(Y, W).
$$
\end{corollary}
\begin{proof}
The singular locus of $Y_0$ is a banana curve, see Example 10.6 in \cite{abouzaid2016lagrangian}.  Thus $Y_0$ has graph-like singular locus and $\Sigma = \Sigma_{G(Y_0)}$, then the statement follows from Theorem \ref{main}. 
\end{proof}